\documentclass[11pt,reqno]{amsart}

\newcommand{\la}{\langle}
\newcommand{\ra}{\rangle}

\renewcommand{\Im}{\operatorname{Im}}

\usepackage{amssymb}
\usepackage{amsthm}
\usepackage{amsxtra}
\usepackage{caption}
\usepackage{fancyhdr}
\usepackage{graphicx}
\usepackage{color}
\usepackage{amsmath}
\usepackage{listings} 
\usepackage{float}
\usepackage{cases} 
\usepackage{threeparttable}
\usepackage{dcolumn}
\usepackage{multirow}
\usepackage{booktabs}
\usepackage{ulem}

\newcommand{\R}{\mathbb R}

\newcommand{\N}{\mathbb N}

\newtheorem{theorem}{Theorem}[section]

\newtheorem{lemma}[theorem]{Lemma}
\newtheorem{corollary}[theorem]{Corollary}

\newtheorem{conjecture}[theorem]{Conjecture}

\theoremstyle{remark}
\newtheorem{remark}[theorem]{Remark}

\numberwithin{equation}{section}

\usepackage{comment}
\title[1D combined NLS]{The Nonlinear Schr\"odinger equation with \\ combined nonlinearities in 1D}
\date{}

\author[O. Ria\~{n}o]{Oscar Ria\~{n}o }
\address{Universidad Nacional de Colombia, Bogot\'a, Colombia}
\curraddr{}
\email{ogrianoc@unal.edu.co}

\author[A. D. Rodriguez]{Alex D. Rodriguez}
\address{Department of Mathematics  \& Statistics\\Florida International University,  Miami, FL, USA}
\curraddr{}
\email{arodr1128@fiu.edu}

\author[S. Roudenko]{Svetlana Roudenko}
\address{Department of Mathematics \& Statistics\\Florida International University,  Miami, FL, USA}
\curraddr{}
\email{sroudenko@fiu.edu}

\subjclass{35A01, 35B40, 35Q40, 35Q55}

\keywords{nonlinear Schr\"odinger equation, well-posedness, weighted spaces,
low power nonlinearity, combined nonlinearities, exponential nonlinearity}

\begin{document}

\begin{abstract}
We consider the one-dimensional nonlinear Schr\"{o}dinger equation 
$$
iu_t + u_{xx} + \mathcal{N}(u)u=0, \quad x,t \in \mathbb R,
$$ 
with the nonlinearity term that is expressed as a sum of powers, possibly infinite: 
$$
\mathcal{N}(u) = \sum d_k |u|^{\alpha_k}, \quad \alpha_k > 0.
$$

We first investigate the local well-posedness of this equation for any positive powers of $\alpha_k$ in a certain weighted class of initial data, subset of $H^1 (\mathbb R)$. For that we use an approach of Cazenave-Naumkin \cite{CazNaum2016}, thus, avoiding any Strichartz estimates. Then, using the pseudo-conformal transformation, we extend the local result to the global one for the initial data with a quadratic phase. 
Furthermore, we investigate the asymptotic behavior of such global solutions and prove scattering for data with the quadratic phase $e^{ib|x|^2}$ with sufficiently large positive $b$, in $H^1(\mathbb R)$. One of the advantages of considering an infinite sum in the nonlinearity term is being able to consider exponential nonlinearities, such as $e^{\gamma |u|^{k}} u$, as well as sine or cosine nonlinearities, and obtain well-posedness in those cases, the first such result for most of those nonlinearities. 

To conclude, we show numerical simulations for various examples of combined nonlinearities, including the double nonlinearity and an exponential one, then investigate the behavior of solutions with positive or negative initial $b$ in a quadratic phase data. Furthermore, we also show that a ground state in the NLS equation with combined nonlinearities no longer provides a sharp threshold for global behavior such as scattering vs. finite time blow-up, instead the equation has a much richer dynamics. 
\end{abstract}

\maketitle

\tableofcontents

\section{Introduction}
We consider the Cauchy problem associated to the $1d$ nonlinear Schr\"odinger equation  
\begin{equation}\label{NLS}
	\left\{\begin{aligned}
		&i\partial_t u + \partial_{x}^{2}u + \mathcal{N}(u)\,u = 0, \quad x\in \mathbb{R},\, \, \, t\in \mathbb{R}, \\
		&u(x,0) = u_{0}(x), 
	\end{aligned}    \right.
\end{equation}
where the potential $\mathcal N$ could be viewed as a combination of power nonlinearities
\begin{equation}\label{nonlinearterm}
	\mathcal{N}(u)=\sum
 d_k |u|^{\alpha_k}, \quad \alpha_k > 0, ~ d_k \in \mathbb C.
\end{equation}

We are interested in developing well-posedness of the Cauchy problem \eqref{NLS} for a general potential $\mathcal N(u)$ given as finite or infinite series (including exponential-type $e^{a |u|^r}$, or oscillatory-type such as $\sin(a |u|^r)$ or $\cos(b|u|^r)$, or low power nonlinearity $|u|^\alpha $ with $\alpha<1$, and hence, of low regularity), with initial data $u_0$ in some Sobolev space (possibly weighted), specified later. 


The Schr\"odinger equation, with several nonlinear terms added, started to be considered back in 1960s and 70s, for example, in {\it optics}, when a continuous wave optical beam, propagating in some media, exhibited self-focusing effects. From the first works of Askaryan \cite{A1962} and Zakharov \cite{Z1967}, it became customary to model the self-focusing of light in optics by the two dimensional nonlinear Schr\"odinger (NLS) equation with a cubic nonlinearity ($|u|^2 u$), considering a potential form $\mathcal N (u) = |u|^2$ as a first approximation for the self-focusing nonlinear medium. Besides the first order approximation, other types of nonlinearities were then considered in optics, including competing, saturated and transitive nonlinearities, see for instance \cite[\S 2.3.1, 2.3.2]{Agrawal2001}, \cite[\S 1.5.1]{KivsharAgrawal2003}. 

In other fields of physics, such as propagation of plasma waves, similar `first order' approximations were adapted too, \cite{Z1972} or \cite{Z1971}. 
It was also understood that a more appropriate physical realization of waveguide propagation would need small departures from the cubic nonlinearity (i.e., from the nonlinearity representing the medium): perturbations by {\it higher order powers} would need to be incorporated (see e.g., \cite{ZakRub1973}, \cite[\S 3]{Z1971}, \cite{BGMP1988}, \cite{PAV1996}), or modifications for other types of nonlinear media such as {\it saturated} or {\it exponential} would have to be considered \cite[\S 2]{ZSS1971}, \cite{GH1991}, \cite{HLRZ2018}, \cite{B1998}. 
In some cases the saturated nonlinearity was approximated by the first few terms from the Taylor expansion, for example, as in cubic-quintic \cite{Pushkarov1979}, \cite{Malkin1993}, \cite{BG2001},\cite{UST}, \cite{M2019}, cubic-quartic \cite{AS2023}, quadratic-cubic \cite{M2019}, or cubic-quintic-septimal models (to approximate nonlinear behavior of metal-dielectric nanocomposites in optics) \cite{RA2017, RMA2015}, even a quadruple nonlinearity such as cubic-quintic-septic-nonic has recently appeared in optical physics literature, \cite{4nonlin}. Mathematical studies of combined nonlinearities in NLS have been on a rise as well; studies of the NLS with two terms can be found in many works, in particular, well-posedness questions and scattering have been investigated, for instance, in \cite{TVZ2007, Zhang2006, KOPV2012, KMV2021, HKW2023, DT2024}, existence of solitary waves and their stability, for example, in \cite{Ohta1995,Reika2003,K2007,Soave2020,CS2021,CKS2023,JJLV2022,KLCT2022,BFG2023}, the NLS with the specific three combined terms (e.g., quadratic-cubic-quartic) was recently investigated in \cite{Liu2021}, \cite{Phan2024}.

A particular example of the above equation with nonlinearities of exponential type have been long used in self-trapped beams in plasma, see \cite{LLT1977}. 
Since the work of Cazenave \cite{Cazenave1979},
the two-dimensional case with an exponential nonlinearity $ u\big(e^{4\pi^2|u|^2}-1\big)$ has been popular when studying exponential nonlinearities as it exhibits energy-critical behavior, the existence of small data global solutions with low regularity (e.g., $H^s(\mathbb{R}^n)$, $s\leq \frac{n}{2}$) was obtained in \cite{NO1998}; see further developments for this nonlinearity and for $\pm \big(e^{\lambda|u|^2}-1- c|u|^2\big)u$, typically studied in dimension 2,  in \cite{Cazenave2003, Chen1990,NO1998,CIMM2009,B2016,A2018,DKM,BIP}.
We remark that nonlinear estimates for exponential nonlinearities typically use Moser-Trudinger type inequality \cite{Trudinger1967,Moser1970, AT2000}, which we don't use in our approach, as we write the exponential nonlinearity via infinite series and only use the commutation of polynomial weights with the derivatives, thus, enabling us to consider any type of exponential nonlinearity as long as the series is convergent.  

For a more complete list of different nonlinear potentials used for physical realizations with the NLS model, see for example, \cite[Table 1]{FP1998}, \cite[Table1.1]{FP2000}, also reviews in \cite{B1998}, \cite{SS1999}, \cite{M2019}.
\smallskip

In this paper, we consider the NLS equation \eqref{NLS} with combined nonlinearities of a finite or infinite sum, thus, fully approximating potentials such as exponential or oscillatory type ($\sin$ or $\cos$). 
We are interested in a more unified approach to treat these type of nonlinearities, and hence, we write the potential term $\mathcal N(u)$ as a power series $\sum d_k |u|^{\alpha_k}$. 

Furthermore, we allow the powers $\alpha_k$ to be not necessarily integers, more specifically, we allow {\it any} positive power $\alpha_k>0$ in the series in \eqref{nonlinearterm}. 
In this work, we consider only the one-dimensional setting of such combined nonlinearities and address the well-posedness theory via an approach that {\it does not rely on Strichartz estimates}. This is because we would like to include 
small powers of $\alpha_k$, where the standard methods to obtain existence and uniqueness of solutions (e.g., via Strichartz estimates) are not suitable or applicable (for example, see discussion in \cite{CazNaum2016}),  
thus, forcing one to search for an alternative method to Strichartz estimates. Moreover, even in the case of a finite number of nonlinear terms, the scaling invariance breaks, which makes it more difficult to, for instance,  choose Strichartz pairs and the time-space norms which would  work simultaneously for all terms. Therefore, we use an approach of Cazenave-Naumkin \cite{CazNaum2016}, a more elementary method that relies on polynomial weights
in the definition of the weighted $L^2$-based spaces, which can be propagated with the solutions of the linear Schr\"odinger evolution and their derivatives (see Section \ref{S:Idea}), or commuted with derivatives in the Fourier transform sense, see Lemmas \ref{derivexp2}, \ref{linearEst}. In summary, do not rely on any scaling or Strichartz pairs. 
The approach of Cazenave-Naumkin \cite{CazNaum2016} has been applied to obtain well-posedness (and global existence and scattering in some cases) in different nonlinear models that typically lack regularity, see, for instance, \cite{AroraRianoRoudenko2021,cazenave2021asymptotic,CazNaum2018,LinaresPonceGleison2019,LinaresPonceGleison2019II,LinaresMiyaGus2019,Miyazaki2020,FriedRianoRoudSonYang2022,PRR,RAWR}.

We address the question of the well-posedness, first, locally in time, and then, in some cases, globally, including scattering. Since global behavior is challenging to investigate analytically in many cases, we also include some numerical simulations, which confirm our analytical results and then provide further extensions.  
Before stating our results, we recall that a Cauchy problem is well-posed if one can assure existence, uniqueness, and continuous dependence of the data-to-solution flow map, either locally or globally in time.

\subsection{Statement of Results}

We start with some notation. {Denote $\langle x \rangle=(1+|x|^2)^{\frac{1}{2}}$.}

Let $n \in \mathbb{R}^{+}$, $r, M \in \mathbb{Z}^{+}$ be such that  
\begin{equation}\label{deffparame}
	\begin{aligned}
		{n > 0}, \; r\geq 3, \, \, \text{ and }\, \, M \geq  n+ r.
	\end{aligned}
\end{equation}
We define the space $\mathcal{X}$
as
{\small \begin{equation}{\label{Xspace}}
    \mathcal{X} = \bigg\{ u \in {\dot H^M(\mathbb{R}) \cap \dot H^1(\mathbb R) } :
    \left( {
    \begin{array}{cl}
        {\rm I.} & \langle{x}\rangle^{n}u\in{L^{\infty}}(\mathbb{R})\\
        {\rm II.} & \langle{x}\rangle^{n}\partial_{x}^\beta{u}\in{L^2}(\mathbb{R}) \; ~\text{for}~ \; 1 \leq \beta \leq r\\
        {\rm III.} & ~~\partial_{x}^\beta{u} \in L^2(\mathbb{R}) \; ~\text{for}~ \; r+1 \leq \beta \leq M
    \end{array}
    }
    \right) ~
    \bigg\},
\end{equation}}
equipped with the norm 
\begin{equation}{\label{Xnorm}}
	\|u\|_{\mathcal{X}} = \|\langle x \rangle^n u\|_{L^{\infty}} + \sum_{k=1}^{r} \|\langle x \rangle^{n} \partial_{x}^{k} u\|_{L^2} + {\sum_{k=r+1}^{M} \|\partial_{x}^{k} u\|_{L^2}}.
\end{equation}
\begin{remark}
Note that in 1D
$\|u\|_{L^2}\leq \|\langle x \rangle^{-n}\|_{L^2}\|\langle x \rangle^n u\|_{L^{\infty}}<\infty$ if $n>\frac12$, and thus, in this case the space $\mathcal{X}$ is embedded into the inhomogeneous Sobolev space $H^M(\mathbb{R})$; hence, in \eqref{Xspace} $u \in \dot{H}^M(\mathbb R) \cap \dot{H}^1(\mathbb R)$ can be replaced with $u \in H^M(\mathbb R)$ as well as the last sum in \eqref{Xnorm} can be replaced with $\|J^M u\|_{L^2}$, where $\widehat{J^M u}(\xi) = \la \xi \ra^M \hat{u}(\xi)$. 
\end{remark}

\smallskip

Since we would like to include small $\alpha$ in nonlinearity $|u|^\alpha u$, in order to be able to differentiate it (and thus, handle the term $|u|^{\alpha-1}$), we introduce the {\it non-vanishing} condition 
\begin{equation}{\label{Xinf0}}
 \inf_{x\in\mathbb{R}} ~\langle x \rangle^n |u(x)| > 0.
\end{equation}

As the local existence time will depend on the infimum value in \eqref{Xinf0}, we rewrite this condition specifying a lower bound as
\begin{equation}\label{Xinf}
		\qquad \qquad \inf_{x\in\mathbb{R}} \langle x \rangle^n |u_0(x)| \geq \lambda>0 \quad \mbox{for ~~some} \quad \lambda >0.
\end{equation}

Our first result establishes the local well-posedness of solutions in the space $\mathcal{X}$ for initial data bounded away from zero as in \eqref{Xinf}. 

\begin{theorem}\label{LWP-X}
{\bf (Local well-posedness in $\mathcal{X}$)}
Let $\{d_k\}$ be a sequence of complex numbers and $\{\alpha_k\}$ be a sequence of positive real numbers. Consider $n,r, M$ as in \eqref{deffparame}.  
 Suppose that for any $R_0>0$ the following condition holds for sequences $\{d_k\}$ and $\{\alpha_k\}$: 
	\begin{equation}\label{coeffcond}
		\sum_{k=0}^{\infty}\sum_{0\leq \beta \leq M}|d_k|\big(1+C(\alpha_k,\beta))(R_0^{|\alpha_k-2\beta|}+|\alpha_k-2\beta|R_0^{|\alpha_k-2\beta-1|})<\infty,
	\end{equation}
where 
\begin{equation}\label{E:C(a,b)}
C(\alpha_k,\beta)=\left\{
\begin{array}{ll}
 0& \quad \rm{if} ~ \beta=0, \\
\underbrace{|\alpha_k||\alpha_k-2|\dots |\alpha_k-2(\beta-1)|}_{\beta-\text{times}} 
&\quad \rm{if} ~ \beta\geq 1.
\end{array}
\right.
\end{equation}
If $u_0 \in \mathcal{X}$ satisfies \eqref{Xinf}, then there exist $T>0$ and a unique solution $u \in C([-T,T], \mathcal{X})$ of \eqref{NLS} with $\mathcal{N}(u)$ given by \eqref{nonlinearterm} such that
\begin{equation}\label{infcondsolut}
	   \sup_{t\in [-T,T]}\|\langle x\rangle^n(u(t)-u_0)\|_{L^{\infty}(\mathbb R)}\leq \frac{\lambda}{2}. 
	\end{equation}
Moreover, the map $u_0 \mapsto u(\cdot,t)$ is continuous in the following sense: for any $0<\widetilde{T}<T$, there exists a neighborhood $V$ of $u_0$ in $\mathcal{X}$ satisfying \eqref{Xinf} such that the map data-to-solution is Lipschitz continuous from $V$ into the class $C([-\widetilde{T},\widetilde{T}],\mathcal{X})$.
\end{theorem}

\begin{remark}\label{R:1} 
An example of an initial condition that satisfies the hypothesis of Theorem \ref{LWP-X} is 
\begin{equation}\label{E:ID}
u_0(x)=\frac{2\lambda e^{i\theta}}{\langle x \rangle^n}+\varphi,
\end{equation}
where $\theta$ is an arbitrary angle, power {$n>0$}, and $\varphi$ is a Schwartz class function with $\|\langle x \rangle^n \varphi\|_{L^{\infty}}\leq \lambda$.
\end{remark}

Before giving examples of nonlinearities that satisfy \eqref{coeffcond}, we state global existence and scattering result (obtained via pseudo-conformal transformation similar to \cite{CazNaum2016}), which is a consequence of the above local well-posedness theorem. 
\begin{theorem}\label{scatres} 
{\bf (Global existence \& uniqueness for $e^{i|x|^2} \mathcal X$ and scattering in $H^s$)}
Let $\{d_k\}$ be a sequence of complex numbers and $\{\alpha_k\}$ be a sequence of positive real numbers such that
\begin{equation}\label{E:a2}
\alpha_k>2 \quad \mbox{for~any~index} \quad k \geq 0.
\end{equation}
Consider { $n > \frac{1}{2}$} and $r, M$ as in \eqref{deffparame}.  
Suppose that for any $R_0>0$ the following condition holds for sequences $\{d_k\}$ and $\{\alpha_k\}$: 
\begin{equation}\label{coeffcond2.1}
\sum_{k=0}^{\infty}\sum_{0\leq \beta \leq M}\frac{|d_k|}{\alpha_k-2}\big(1+C(\alpha_k,\beta))(R_0^{|\alpha_k-2\beta|}+|\alpha_k-2\beta|R_0^{|\alpha_k-2\beta-1|})<\infty,
\end{equation}
where $C(\alpha_k,\beta)$ is as in \eqref{E:C(a,b)}. 
For $v_0\in \mathcal{X}$ satisfying \eqref{Xinf0} and $b \in \mathbb R$, define $u_0=e^{\frac{ib|x|^2}{4}}v_0$.
Let 
\begin{equation}\label{E:Sn_cond}
    s = s_n \in
    \begin{cases}
        \{1\} & \text{ if } n > \frac{3}{2} \\
        \left(0, n-\frac{1}{2}\right) & \text{ if } n \in \left(\frac{1}{2},\frac{3}{2}\right].
    \end{cases}
\end{equation}

Then, for $b>0$ sufficiently large, there exists a unique global solution $u(t,x)$ of \eqref{NLS} with initial data $u_0$ and the nonlinearity given by \eqref{nonlinearterm} subject to \eqref{E:a2} in 
$$
C([0,\infty);H^{s_n}(\mathbb{R}))\cap L^{\infty}\big([0,\infty);L^{\infty}(\langle x \rangle^{\frac{1}{2}}\, dx)\big).
$$ 
Moreover, $u$ scatters in $H^{s_n}(\mathbb R)$, i.e., there exists $u_{+}\in H^{s_n}(\mathbb{R})$ such that
\begin{equation*}
\lim_{\substack{t \to \infty \\ t>0}} \|u(t)- e^{it\partial_x^2}u_{+}\|_{H^{s_n}}=0.
\end{equation*}
In addition,
\begin{equation*}
\sup_{t>0} \, (1+t)^{\frac{1}{2}}\|u(t)\|_{L^{\infty}}<\infty.
\end{equation*}
\end{theorem}

\begin{remark}
In the case of two nonlinearities~ $\mathcal N(u)=d_1|u|^{\alpha_1}+d_2|u|^{\alpha_2}$, $\alpha_1<\alpha_2$, scattering was shown in \cite[Theorem 1.8]{TVZ2007} for the case of $\alpha_1 > \frac{1+\sqrt{17}}{2}$ and $d_1, d_2 <0$ (defocusing nonlinearities), provided $u_0 \in H^1(\mathbb{R})$ with finite variance. We improve in that case the range for scattering for data of the type  $u_0(x) = e^{\frac{ib|x|^2}{4}}v_0(x)$ to $\alpha_1 > 2$ in Theorem \ref{scatres}. In the same paper in the case $0 < \alpha_1 < \alpha_2 < 4$ for $d_2 > 0$ (focusing larger nonlinearity) it was indicated that scattering is unknown, our Theorem \ref{scatres} shows scattering in this case for data $u_0(x) = e^{\frac{ib|x|^2}{4}}v_0(x)$, provided $2<\alpha_1 < \alpha_2 < 4$.  
\end{remark}

Observe that when the nonlinearity in \eqref{nonlinearterm} is sufficiently regular, i.e., the series of coefficients $\{\alpha_k\}$ are such that $\alpha_k\geq 1$ for all $k\geq 1$, then it is possible to obtain the local well-posedness in $H^1(\mathbb{R})$ for the combined nonlinearity $\mathcal N(u)$ without using the weighted space $\mathcal{X}$ (essential for the low powers together with the infimum condition), which we state next.

\begin{theorem}\label{LocalwellpossH1}
 {\bf (i) (Local well-posedness in $H^1$)}. Let $\{d_k\}$ be a sequence of complex numbers and $\{\alpha_k\}$ be a sequence of positive real numbers such that  
\begin{equation}\label{E:a1}
\alpha_k \geq 1 \quad \mbox{for~any~index} \quad k\geq 0.
\end{equation}
Suppose that for any $R_0>0$ the following condition holds for sequences $\{d_k\}$ and $\{\alpha_k\}$: 
\begin{equation}\label{coeffcond2}
\sum_{k=0}^{\infty}|d_k| \big( 1+|\alpha_k|+|\alpha_k|^2 \big)\, R_0^{\alpha_k}<\infty.
\end{equation}
Then for any $u_0 \in H^1(\mathbb R)$, there exist $T>0$ and a unique solution $u \in C([-T,T], H^1(\mathbb{R}))$ of \eqref{NLS} with $\mathcal{N}(u)$ given by \eqref{nonlinearterm} subject to \eqref{E:a1}. 
Additionally, given $u_0\in H^1(\mathbb{R})$, there exist a time $T>0$ and a neighborhood $V$ of $u_0$ in $H^1(\mathbb{R})$ such that the flow-map data-to-solution from $V$ into $C([-T,T];H^1(\mathbb{R}))$ is continuous.
\smallskip

{\bf (ii)  (Global existence \& uniqueness for $e^{i|x|^2}H^1$ and scattering in $H^1$)}. Assume $\{\alpha_k\}$ satisfy 
\eqref{E:a2}.  
Additionally, suppose that for any $R_0>0$ the following condition holds for sequences $\{d_k\}$ and $\{\alpha_k\}$: 
	\begin{equation}\label{coeffcond3}
\sum_{k=0}^{\infty}\frac{|d_k|}{\alpha_k-2} \big(1+|\alpha_k|+|\alpha_k|^2)R_0^{\alpha_k}<\infty.
	\end{equation}
For $v_0 \in  H^1(\mathbb{R})\cap L^2(|x|^2\, dx)$ and $b \in \mathbb R$, define $u_0=e^{\frac{ib|x|^2}{4}}v_0$. 

If $b>0$ is sufficiently large, then there exists a unique global solution $u$ of \eqref{NLS} with initial data $u_0$ such that
\begin{equation*}
  u \in C([0,\infty); H^1(\mathbb{R})). 
\end{equation*}
Moreover, $u$ scatters in $H^1(\mathbb{R})$, i.e., there exists $u_{+}\in H^1(\mathbb{R})$ such that
\begin{equation*}
    \lim \limits_{\substack{t\to \infty \\ t>0}}\|u(t)-e^{it\partial_x^2} u_+\|_{H^1}=0.
\end{equation*}
\end{theorem}

The proof of Theorem \ref{LocalwellpossH1} uses basic properties of $H^1(\mathbb{R})$ and arguments similar to those used in the computation of the $H^M$-norm in Theorem \ref{LWP-X}. Similarly, the global existence and scattering results in Theorem \ref{LocalwellpossH1} are obtained by a careful adaptation of the proof of Theorem \ref{scatres}.
\medskip

\begin{remark} 
We mention that the global well-posedness in $H^1(\mathbb{R})$ for the Cauchy problem  \eqref{NLS} with a {\it single} nonlinear term with  $0<\alpha<4$ (and for sufficiently small data in an appropriate sense for $\alpha \geq 4$) is well-known, for book references, e.g., see \cite[Theorem 6.1]{LinaresPonce2015}, \cite[Chapter 3]{Tao2006}, and \cite[Chapters 4 and 6]{Cazenave2003}. In particular, the initial conditions with faster decay than the polynomial, such as exponential as in the ground state solutions, generate global solutions. 
\end{remark}

We next present several consequences of Theorems \ref{LWP-X}, \ref{scatres}, and \ref{LocalwellpossH1} for a combined number of nonlinearities.
\begin{remark} Observe that \eqref{coeffcond} and \eqref{coeffcond2.1} hold trivially for a finite number of nonlinearities, i.e., given an integer $N \geq 0$ and $d_k=0$ for all $k > N$ ($d_k > 0$ for $0 \leq k \leq N$), consider the potential
\begin{equation*}
	\mathcal{N}(u)=\sum_{k=0}^N d_k|u|^{\alpha_k}.
\end{equation*} 
As a consequence of Theorems \ref{LWP-X} and \ref{scatres}, we deduce the following statement. 
\end{remark}

\begin{corollary}[{\bf Finite number of combined nonlinearities, space $\mathcal X$}]\label{corollwellposed1}
Let $N\geq 0$ be an integer, $d_0,\dots, d_N \in \mathbb{C} \setminus\{0\}$, and $0<\alpha_0 \leq \dots \leq \alpha_N$. If $u_0 \in \mathcal{X}$ satisfies \eqref{Xinf0}, then there exist $T>0$ and a unique solution $u \in C([-T,T], \mathcal{X})$ of the Cauchy problem
\begin{equation}\label{SNNLS}
\left\{\begin{aligned}
    & i\partial_t u + \partial_{x}^{2}u + d_0|u|^{\alpha_0}u+\dots+d_N|u|^{\alpha_N}u =0, \quad x\in \mathbb{R},\, \, \, t\in \mathbb{R}, \\
    &u(x,0) = u_{0}(x), 
		\end{aligned}    \right.
	\end{equation}
such that \eqref{infcondsolut} is also valid. 
Additionally, the map $u_0 \mapsto u(\cdot,t)$ is continuous in the following sense: for any time $0<\widetilde{T}<T$, there exists a neighborhood $V$ of $u_0$ in $\mathcal{X}$ satisfying \eqref{Xinf0} such that the map data-to-solution is Lipschitz continuous from $V$ into the class $C([-\widetilde{T},\widetilde{T}],\mathcal{X})$.
	
Moreover, suppose $2<\alpha_0 \leq \dots \leq \alpha_N$ are fixed in \eqref{SNNLS}, $v_0\in \mathcal{X}$ satisfies \eqref{Xinf0}, and $u_0=e^{\frac{ib|x|^2}{4}}v_0$, $b \in \mathbb R$. Let $s_n$ be as in \eqref{E:Sn_cond}.
 Then, for $b>0$ sufficiently large and the initial condition $u_0$, there exists a unique global solution $u$ of \eqref{SNNLS} in
$$
C([0,\infty);H^{s_n}(\mathbb{R}))\cap L^{\infty}\big([0,\infty);L^{\infty}(\langle x \rangle^{\frac{1}{2}}\, dx)\big).
$$ 
Furthermore, $u$ scatters in $H^{s_n}(\mathbb R)$, i.e., there exists
$u_{+}\in H^{s_n}(\mathbb{R})$ such that
\begin{equation*}
\lim_{\substack{t \to \infty \\ t>0}} \|u(t)- e^{it\partial_x^2} u_{+}\|_{H^{s_n}}=0.
\end{equation*}
\end{corollary}

\begin{remark}
We note that when $\mathcal{N}(u)=\sigma|u|^{\alpha}$, $\alpha>0$, $\sigma \in \mathbb{R}$, our well-posedness result in Corollary \ref{corollwellposed1} extends the one-dimensional results of Cazenave-Naumkin \cite{CazNaum2016} to fractional weights (i.e., in the term $\langle x \rangle ^n$ we allow $n\in \R^+$ rather than only positive integers $\N$). Furthermore, we also extend scattering to $H^{s_n}(\mathbb{R})$ with $s_n=1$ if $n>\frac{3}{2}$, and $0<s_n<n-\frac{1}{2}$, if $\frac{1}{2}<n\leq \frac{3}{2}$, which is more flexible with respect to fractional values of the weight power $n\in \mathbb{R}^{+}$.
\end{remark}

In the case $1\leq \alpha_0 \leq \dots \leq \alpha_N$, Theorem \ref{LocalwellpossH1} implies the local well-posedness in $H^1(\mathbb{R})$ and if $2 \leq \alpha_0 \leq \dots \leq \alpha_N$, then we also have global existence, uniqueness, and scattering in $H^1$ for the data with a quadratic phase. 

\begin{corollary}[{\bf Finite number of combined nonlinearities,
space $H^1$}]\label{Cor10}

Let $N\geq 0$ be an integer, $d_0,\dots, d_N \in \mathbb{C} \setminus\{0\}$, and $1\leq \alpha_0 \leq \dots \leq \alpha_N$. Then for $u_0 \in H^1(\mathbb R)$ the Cauchy problem \eqref{SNNLS} is locally well-posed in $H^1(\mathbb{R})$. 

Moreover, let $2<\alpha_0 \leq \dots \leq \alpha_N$ are fixed in \eqref{SNNLS}, and 
for $v_0\in H^1(\mathbb{R})\cap L^2(|x|^2\, dx)$ define $u_0=e^{\frac{ib|x|^2}{4}}v_0$, $b\in \mathbb R$. Then, for $b>0$ sufficiently large and the initial data $u_0 \in H^1(\mathbb R)$,
there exists a unique global solution $u$ of \eqref{SNNLS} in $C([0,\infty);H^{1}(\mathbb{R}))$. 

Furthermore, $u$ scatters in $H^1(\mathbb R)$, i.e., there exists $u_{+}\in H^{1}(\mathbb{R})$ such that
\begin{equation*}
\lim_{\substack{t \to \infty \\ t>0}} \|u(t) - e^{it\partial_x^2} u_{+}\|_{H^{1}}=0.
\end{equation*}
\end{corollary}

\begin{remark} Our results provide well-posedness for the nonlinearities that are expressed by infinite power series, and, in particular, for some well-known functions such as \textit{exponential}, \textit{sine}, and \textit{cosine}. For example, for the Cauchy problem \eqref{NLS} with $e^{|u|} u$ nonlinearity, fix\footnote{Our results will be valid for $c \in \mathbb C$, however, for the purpose of this paper, it suffices to consider real $c$.} $r>0$, $c \in \mathbb{R}$, and consider
\begin{equation}\label{Nonlinearityexp}
	\mathcal{N}(u)=\sum_{k=1}^{\infty} \frac{(c|u|^r)^{k}}{k!},
\end{equation}
with coefficients $d_k = \frac{c^k}{k!}$ and powers $\alpha_k = rk >0$, which satisfy \eqref{coeffcond}. 
Note that the nonlinearity in \eqref{nonlinearterm} does not contain a linear term (all $\alpha_k>0$), for consistency the term $k=0$ is not included in the series above. 
Then, given $v_0 \in \mathcal{X}$ such that it also satisfies  \eqref{Xinf}, by Theorem \ref{LWP-X} there exist a time $T>0$ and a unique solution $u\in C([-T,T];\mathcal{X})$ of \eqref{NLS} with
$\mathcal{N}(u)$ given by \eqref{Nonlinearityexp} and the initial condition $u(x,0)=v_0(x)$. Thus, setting $v=e^{it}u \in C([-T,T];\mathcal{X})$ and noting that
$$
e^{c|v|^{r}}=\sum_{k=0}^{\infty}\frac{(c|v|^r)^{k}}{k!} \equiv \sum_{k=0}^{\infty}\frac{(c|u|^r)^{k}}{k!}
= \mathcal N (u) + 1,
$$
we obtain that $v$ solves
\begin{equation}\label{expNLS} 
	\left\{\begin{aligned}
		&i\partial_t v + \partial_{x}^{2}v + e^{c|v|^{r}}v = 0, \quad c \in \mathbb R\setminus\{0\}, ~ r>0, ~ x \in \mathbb{R}, ~ t\in \mathbb{R}, \\
		&v(x,0) = v_{0}(x). 
	\end{aligned}    \right.
\end{equation}
\end{remark}

Summarizing, we have the following consequence of Theorem \ref{LWP-X} for \eqref{expNLS}. 
\begin{corollary}[{\bf LWP in $\mathcal{X}$ for exponential nonlinearity}]\label{Cor1-exp}
Let $v_0 \in \mathcal{X}$ satisfy \eqref{Xinf}. Then there exist $T>0$ and a unique solution $v \in C([-T,T], \mathcal{X})$ of \eqref{expNLS} such that $\sup_{t\in[-T,T]}\|\langle x\rangle^n(v(t)-v_0)\|_{L^{\infty}}\leq \frac{\lambda}{2}$. Moreover, the map $v_0 \mapsto v(\cdot,t)$ is continuous in the following sense: for any $0<\widetilde{T}<T$, there exists a neighborhood $V$ of $v_0$ in $\mathcal{X}$ satisfying \eqref{Xinf} such that the map is Lipschitz continuous from $V$ into the class $C([-\widetilde{T},\widetilde{T}],\mathcal{X})$.
\end{corollary}
If $r\geq 1$, the nonlinearity \eqref{Nonlinearityexp} satisfies the hypothesis of Theorem \ref{LocalwellpossH1}, and hence, the previous Corollary \ref{Cor1-exp} simplifies as follows.
\begin{corollary}[{\bf LWP in $H^1$ for exponential nonlinearity}]\label{Cor2-exp}
Let $r\geq 1$ and $v_0 \in H^1(\mathbb R)$. Then \eqref{expNLS} is locally well-posed in $H^1(\mathbb{R})$.
\end{corollary}

As far as the global existence, we remark that a direct application of Theorem \ref{scatres} does not work for the nonlinearity \eqref{Nonlinearityexp} (need $\alpha_k >2$). Therefore, if we set
\begin{equation}\label{E:exp-global}
 \mathcal{N}(u)=\sum_{k>\frac{2}{r}}^{\infty} \frac{(c|u|^r)^{k}}{k!},
\end{equation}
then Theorem \ref{scatres} holds true for \eqref{NLS} with $\mathcal N(u)$ as in \eqref{E:exp-global}. More precisely, consider the problem 
\begin{equation}\label{expNLS2}
\left\{\begin{aligned}
    & i\partial_t v + \partial_{x}^{2}v +\Big( e^{c|v|^{r}}-\sum_{0\leq k\leq\frac{2}{r}} \frac{(c|v|^r)^{k}}{k!}\Big)v = 0, \quad c \in \mathbb C\setminus\{0\}, ~r>0, ~ x\in \mathbb{R}, ~ t\in \mathbb{R}, \\
    & v(x,0) = v_{0}(x). 
    \end{aligned}    \right.
\end{equation}

For clarity, we make a remark about the nonlinear potential above. 

Let $r>0$ be arbitrary, then we have 
\begin{equation*}
 \Big( e^{c|v|^{r}}-\sum_{0\leq k\leq\frac{2}{r}} \frac{(c|v|^r)^{k}}{k!}\Big)=\sum_{k>\frac{2}{r}}^{\infty} \frac{(c|v|^r)^{k}}{k!}.
\end{equation*}
To apply Theorem \ref{scatres} to the above nonlinearity, we need all the powers $(|v|^{r})^k=|v|^{rk}$ to satisfy $rk>2$. Note that $r>0$ implies that $\frac{2}{r}>0$, then $\sum_{0\leq k\leq\frac{2}{r}} \frac{(c|v|^r)^{k}}{k!}$ has at least the first factor $k=0$. For example, 
{\small
\begin{itemize}
\item If $0<\frac{2}{r}<1$, i.e., $r>2$, we have
\begin{equation*}
 \Big( e^{c|v|^{r}}-\sum_{0\leq k\leq\frac{2}{r}} \frac{(c|v|^r)^{k}}{k!}\Big)=\sum_{k=1}^{\infty} \frac{(c|v|^r)^{k}}{k!}=\sum_{k=0}^{\infty} \frac{(c|v|^r)^{k+1}}{(k+1)!},
\end{equation*}
thus, $\alpha_k=r(k+1)$, $k\geq 0$.

\item If $1\leq \frac{2}{r}<2$, we have
\begin{equation*}
 \Big( e^{c|v|^{r}}-\sum_{0\leq k\leq\frac{2}{r}} \frac{(c|v|^r)^{k}}{k!}\Big)=\sum_{k=2}^{\infty} \frac{(c|v|^r)^{k}}{k!}=\sum_{k=0}^{\infty} \frac{(c|v|^r)^{k+2}}{(k+2)!},
\end{equation*}
thus, $\alpha_k=r(k+2)$, $k\geq 0$.

\item If $l\leq \frac{2}{r}<l+1$ for some $l\geq 1$, we have
\begin{equation*}
 \Big( e^{c|v|^{r}}-\sum_{0\leq k\leq\frac{2}{r}} \frac{(c|v|^r)^{k}}{k!}\Big)=\sum_{k=l+1}^{\infty} \frac{(c|v|^r)^{k}}{k!}=\sum_{k=0}^{\infty} \frac{(c|v|^r)^{k+l+1}}{(k+l+1)!},
\end{equation*}
thus, $\alpha_k=r(k+l+1)$, $k\geq 0$.
\end{itemize}
}
Hence, we can deduce for an exponential nonlinearity the following statement. 

\begin{corollary}[{\bf Global well-posedness and scattering in $H^{s_n}$ for exponential nonlinearity}]
Take $r>0$, $w_0\in \mathcal{X}$ satisfying \eqref{Xinf} and set $v_0=e^{\frac{ib|x|^2}{4}}w_0$, $b\in \mathbb R$. Let $s_n=1$ if $n>\frac{3}{2}$, and $0<s_n<n-\frac{1}{2}$ if $\frac{1}{2}<n\leq \frac{3}{2}$. 
Then, for $b>0$ sufficiently large and the initial condition $v_0$, there exists a unique global solution $v$ of \eqref{expNLS2} in 
$$
C([0,\infty);H^{s_n}(\mathbb{R}))\cap L^{\infty}\big([0,\infty); L^{\infty}(\langle x \rangle^{\frac{1}{2}}\, dx)\big).
$$ 
Moreover, $v$ scatters in $H^{s_n}(\mathbb R)$, i.e., there exists $v_{+}\in H^{s_n}(\mathbb{R})$ such that
\begin{equation*}
\lim_{\substack{t \to \infty \\ t>0}} \|v(t)- e^{it\partial_x^2} v_{+}\|_{H^{s_n}}=0.
\end{equation*}
\end{corollary}   
A further consequence of Theorem \ref{LocalwellpossH1} part (ii) is the following statement. 
\begin{corollary}[{\bf Global well-posedness and scattering in $H^{1}$ for exponential nonlinearity}]\label{C:gwp-exp}
Let $r>0$, $w_0 \in  H^1(\mathbb{R})$ and set $v_0=e^{\frac{ib|x|^2}{4}}w_0$, $b\in \mathbb R$.
Then, for $b>0$ sufficiently large and initial data $v_0$, there exists a unique global solution $v$ of \eqref{expNLS2}  in  $C([0,\infty);H^{1}(\mathbb{R}))$. Moreover, $v$ scatters in $H^1$, i.e., there exists $v_{+}\in H^{1}(\mathbb{R})$ such that
\begin{equation*}
\lim_{\substack{t \to \infty \\ t>0}} \|v(t) - e^{it\partial_x^2} v_{+}\|_{H^{1}}=0.
\end{equation*}
\end{corollary}   

\begin{remark} Besides exponential nonlinearity, 
we can also consider other nonlinearities
\begin{equation}\label{extraexamples1}
\begin{aligned}
    \mathcal{N}_1(u)&=\sin(|u|^{r})=\sum_{k=0}^{\infty}(-1)^k\frac{|u|^{r(2k+1)}}{(2k+1)!},\\
    \mathcal{N}_2(u)&=\cos(|u|^{r})=\sum_{k=0}^{\infty}(-1)^k\frac{|u|^{2r k}}{(2k)!}.
    \end{aligned}
\end{equation}
Our local existence result in Theorems \ref{LWP-X} and \ref{LocalwellpossH1} is valid for the nonlinearity $\mathcal{N}_1(u)$. In the case of $\mathcal{N}_2(u)$, we can find solutions for NLS with $\widetilde{\mathcal{N}}_2(u)=\cos(|u|^{\eta})-1$, and then as above with the exponential, applying $v = e^{it}u$, we  obtain solutions for $\mathcal{N}_2(u)=\cos(|u|^{\eta})$. Furthermore, when $r\geq 1$ in \eqref{extraexamples1}, the local well-posedness from Theorem \ref{LocalwellpossH1} (i) applies to these examples. 

On the other hand, for any fixed $\eta>0$, the global results of Theorem \ref{scatres} and Theorem \ref{LocalwellpossH1} (ii) are valid for the nonlinearities
\begin{equation} \label{extraexamples2}
    \begin{aligned}
    \mathcal{N}_1(u)&=\sin(|u|^{\eta})-\sum_{0\leq k \leq \frac{1}{\eta}-\frac{1}{2}}(-1)^k\frac{|u|^{\eta(2k+1)}}{(2k+1)!}=\sum_{ k>\frac{1}{\eta}-\frac{1}{2}}(-1)^k\frac{|u|^{\eta(2k+1)}}{(2k+1)!},\\
    \mathcal{N}_2(u)&=\cos(|u|^{\eta})-\sum_{0\leq k \leq \frac{1}{\eta}}(-1)^k\frac{|u|^{2\eta k}}{(2k)!}=\sum_{k>\frac{1}{\eta}}(-1)^k\frac{|u|^{2\eta k}}{(2k)!}.
    \end{aligned}
\end{equation}
(The convention for the empty summation is defined as zero, e.g., when $\eta>2$, then $\sum_{0\leq k \leq \frac{1}{\eta}-\frac{1}{2}}(\dots)=0$.) 
\end{remark}

\begin{remark}
(i)  The deduction of Theorem \ref{LocalwellpossH1} shows that the Sobolev embedding $H^1(\mathbb{R})\hookrightarrow L^{\infty}(\mathbb{R})$ is sufficient to treat combinations of nonlinearities of the form $|u|^{\alpha}u$ with $\alpha\geq 1$. In contrast, assuming the condition \eqref{Xinf}, the weighted space $\mathcal{X}$ is convenient to obtain local existence results for any nonlinearity $\alpha>0$, which {\it broadens} the range of nonlinearities considered in this paper. In this sense, Theorems \ref{LWP-X} and \ref{LocalwellpossH1} complement each other. Furthermore, the results of Theorem \ref{LWP-X} are also of an independent interest: to study how solutions of \eqref{NLS} with power series nonlinearity propagate fractional weights and condition \eqref{Xinf}.

(ii) A key observation in the proof of Theorem \ref{LWP-X} is that the space $\mathcal{X}$ defined by \eqref{deffparame}, \eqref{Xspace} and \eqref{Xnorm} does not depend on the nonlinearity in the equation \eqref{NLS}. This partially justifies why we can consider nonlinearities given by a combination of different powers, and even more so, by infinite series. However, it is not clear if one can construct such space in higher dimensions. As a matter of fact, in the case of the Cauchy problem for the nonlinear Sch\"odinger equation with one nonlinearity
	\begin{equation}
		\left\{\begin{aligned}
			&i\partial_t u + \Delta u + \mu |u|^{\alpha}u = 0, \qquad x\in \mathbb{R}^N,\, \, \, t\in \mathbb{R}, \, \, \mu \neq 0\\
			&u(x,0) = u_{0}(x), 
		\end{aligned}    \right.
	\end{equation}
Cazenave and Naumkin \cite{CazNaum2016} obtained well-posedness results similar to those in Theorem \ref{LWP-X}, where, in particular, the initial condition satisfies $\langle x \rangle^n u_0(x) \in L^{\infty}(\mathbb{R}^N)$ with $n\geq \{\frac{N}{2}+1, \frac{N}{2\alpha}\}$. Thus, the condition $n\geq \frac{N}{2\alpha}$ yields extra difficulties to deal with nonlinearities of the form \eqref{nonlinearterm}. We overcome such restriction {(and completely remove it)} by working in the {\it one-dimensional} setting and using fractional weights (see Lemma \ref{derivexp2} below). An interesting problem is to extend the results of Theorems \ref{LWP-X} and \ref{scatres} to higher dimensions. 
	
(iii) We emphasize that for each of the previous examples (see, \eqref{SNNLS}, \eqref{expNLS}, \eqref{expNLS2}, \eqref{extraexamples1}, and \eqref{extraexamples2}), it may be possible to define an appropriate weighted space similar to $\mathcal{X}$, in which it is possible to solve each of these equation without using explicitly the power series. However, we emphasize that our results apply to a broader family of equations, including the case of infinite nonlinearities.
\end{remark}

Our results above discuss the global solutions and scattering in the case when $b>0$ is large. We expect that scattering holds for any $b>0$. We provide confirmation for that (and some of our other results stated above) in Section \ref{SectionNum}, where we show a few numerical simulations of several cases of the combined NLS equation, including the case of infinite number of terms via an exponential nonlinearity. 
There, we first investigate a {\it single} nonlinearity (see Section \ref{S6.1}) with data slowly decaying as $|x|^{-n}$, $n=1, \frac23, \frac12$ ({note that} the last one is not in $L^2$ {but still in $\mathcal X$}). Secondly, in Section \ref{S6.2} we show numerical simulations for a {\it double} combined nonlinearity with slow and fast decaying initial data, including low powers combined nonlinearity (such as $\alpha=\frac19, \frac79$) as well as several integer powers. We show long term behavior for double powers, finding that a simple dichotomy for the scattering vs. blow-up behavior does not necessarily hold around the soliton perturbations and is much richer (see similar behaviors in 1d bi-harmonic NLS \cite{KPRS}). Thirdly, we examine an example of infinitely many terms in nonlinearity, with an example of the exponential nonlinearity in Section \ref{S6.3}, and show different types of global behavior (scattering, asymptotic oscillations, blow-up) for the exponential nonlinearities (see Figure \ref{fig:Exp0025_Exp005}). 

Finally, we investigate the case of Theorems \ref{scatres} and \ref{LocalwellpossH1} (ii) for $b>0$ and $b<0$ and make a numerical confirmation of the following conjecture. 

\begin{conjecture}\label{Conj}
Let $u_0 \in \mathcal X$ or $H^1$ and for $b \in \mathbb R$ set $v_0 = e^{\frac{i b |x|^2}{4}}u_0$. Then 

(a) for $b>0$ the solution with the given initial condition $v_0$ exists globally and scatters;

(b) for $b<0$ and sufficiently localized $u_0$, the solution with the initial condition $v_0$ blows up in finite time.  
\end{conjecture}

\bigskip

This paper is organized as follows: in Section \ref{SectionLinearNonl} we first give an idea of how polynomial weights exchange with the linear Schr\"odinger flow and its derivatives; then we obtain estimates in space $\mathcal X$ for the linear Schr\"odinger evolution as well as for a nonlinear term. Section \ref{SectionLWPX} contains a proof of Theorem \ref{LWP-X}, while Section \ref{SectionScat} contains a proof of Theorem \ref{scatres}. Section \ref{WPHSCH1} deals with the local well-posedness and scattering results in Theorem \ref{LocalwellpossH1}. Lastly, Section \ref{SectionNum} showcases a variety of numerical simulations for the combined NLS equation: an example of low powers of nonlinearity and slow decay initial data, several examples with two combined nonlinear terms (of various powers), and then of an exponential nonlinearity, confirming our analytical findings,  giving further extensions, observing a non-existence of a sharp threshold via a corresponding ground state, and confirmations to Conjecture \ref{Conj}.
\smallskip

{\bf Acknowledgments.}
The research of the authors was partially supported by the NSF grants DMS-1927258, 2055130, 2221491, and 2452782 (PI: S. Roudenko). 
We also thank Kai Yang for help on the initial numerics.
\medskip


{\bf Notation.}
For $s\in \mathbb{R}$, the Bessel potential of order $-s$ is denoted by $J^s=(1-\partial_x^2)^{\frac{s}{2}}$, equivalently, $J^s$ is defined by the Fourier multiplier with symbol $\langle \xi \rangle^{s}=(1+|\xi|^2)^{\frac{s}{2}}$. The Riesz potential of order $-s$ is denoted $D^s=(-\partial_x^2)^{\frac{s}{2}}$, i.e., $D^s$ is the Fourier multiplier operator determined by the function $|\xi|^s$. We use the standard Lebesgue spaces $L^p(\mathbb{R})$, $1\leq p\leq \infty$, with the usual norm $\|f\|_{L^p}$. We also use the notation $\|f(x,t)\|_{L^p_x}$ to specify that the $L^p$-norm is acting on the $x$-variable. The space $H^{s}(\mathbb{R})$ denotes the Sobolev space of order $s$. We denote by $e^{i t \partial_x^2}$, $t\in \mathbb{R}$, the unitary group that generates solutions to the linear Schr\"odinger equation, in other words, the Fourier multiplier operator associated to $e^{-it|\xi|^2}$.


\section{Linear and nonlinear estimates}\label{SectionLinearNonl}

To give a basic idea of the approach used in this paper, in simplest terms, we start with a preliminary explanation of how the weight $\langle x \rangle^n$ exchanges with the linear flow $e^{it\partial_x^2}$. Generalization of this is then given in Lemma \ref{derivexp2}. After that we develop the weighted estimates of the linear flow in Lemma \ref{linearEst} and then on the nonlinear part in Lemma \ref{lemmanonlEs} with an interpolation result in Lemma \ref{interlemma} before that.  

\subsection{The idea of exchanging $\langle x \rangle^n $ with $e^{it\partial_x^2}$}\label{S:Idea}

We start with recalling that if $f$ is a sufficiently regular function with enough decay, then $u=e^{it\partial_x^2}f$ solves the linear equation (with $u(x,0)=f(x)$)
\begin{equation}\label{linearSchrodingerEq}
    i\partial_tu+\partial_x^2 u=0.
\end{equation}
Let $b>0$, multiplying the equation \eqref{linearSchrodingerEq} by $( \langle x \rangle^{2b}\overline{u})$ and integrating over the spatial variable yields 
\begin{equation*}
\begin{aligned}
    \int i\partial_t u ( \langle x \rangle^{2b}\overline{u})\, dx+\int \partial_x^2 u ( \langle x \rangle^{2b}\overline{u})\, dx=0.
\end{aligned}    
\end{equation*}
Integrating by parts and taking the imaginary part of the previous identity, we get
\begin{equation*}
\begin{aligned}
    \frac{1}{2}\frac{d}{dt}\int \langle x \rangle^{2b}|u|^2\, dx= 
    \Im\int \partial_x u \,\overline{u} \,\partial_x \big( \langle x \rangle^{2b} \big)\, dx
    \leq 
    2b\|\langle x\rangle^{b}u\|_{L^2}\|\langle x\rangle^{b-1}\partial_x u\|_{L^2},
\end{aligned}    
\end{equation*}
where we have used Cauchy–Schwarz inequality together with $|\partial_x ( \langle x \rangle^{2b})|\leq 2b  \langle x \rangle^{2b-1}$. Then, Gronwall's inequality establishes
\begin{equation}\label{afterGronwa}
\begin{aligned}
\|\langle x\rangle^{b}u(t)\|_{L^2}\leq \|\langle x\rangle^{b}f\|_{L^2}+2b|t|\big(\sup_{t\in \mathbb{R}}\|\langle x\rangle^{b-1}\partial_x u\|_{L^2}\big).
\end{aligned}    
\end{equation}
The above equation formally shows a relation between propagation of polynomial weights and regularity of solutions to the linear Schr\"odinger  equation. For example, if $b=1$ in \eqref{afterGronwa}, using that $u(t)=e^{it\partial_x^2}f$ and $e^{it\partial_x^2}$ is a unitary operator on $H^1(\mathbb{R})$, the inequality \eqref{afterGronwa} implies
\begin{equation*}
\begin{aligned}
\|\langle x\rangle e^{it\partial_x^2}f\|_{L^2}\leq \|\langle x\rangle f\|_{L^2}+2|t|\|\partial_x f\|_{L^2}.
\end{aligned}    
\end{equation*}
Thus, we see that for the linear Schr\"odinger equation to propagate a polynomial weight of order 1, it is required to control derivatives of the same order. We present a more detailed relation between decay and regularity in Lemma \ref{derivexp2} below, proof of which can be found in 
\cite[Lemma 2]{NahasPonce2009} and \cite[Lemma 2.9]{AroraRianoRoudenko2021}.

\begin{lemma}\label{derivexp2}
Let $b \in \mathbb{R}^{+}$. Then for any $t\in \mathbb{R}$, there exist $C>0$ such that
\begin{equation*}
\|\langle x \rangle^{b}e^{it\partial_x^2}f\|_{L^2} \leq C \langle t \rangle^b\big( \|J^{b}f\|_{L^2}+\|\langle x\rangle^b f\|_{L^2} \big).
\end{equation*}
\end{lemma}
We mention that for a simpler treatment of integer weights $n \in \mathbb{N}$, the reader is referred to \cite[Lemma 4.1]{RAWR}. 
The previous lemma allows us to obtain estimates for solutions of the linear Sch\"odinger equation in the space $\mathcal{X}$ introduced in \eqref{Xspace}. 

\begin{lemma}\label{linearEst}
Let $f \in \mathcal{X}$. Then there exists a constant $C>0$ such that for all $t \in \mathbb{R}$,
    \begin{equation}\label{eqliner1}
        \|e^{it \partial_x^2}f\|_{\mathcal{X}}\leq C\langle t \rangle^{n+1}\|f\|_{\mathcal{X}},
    \end{equation}
and
    \begin{equation}\label{eqliner2}
        \|\langle x \rangle^n \big(e^{it \partial_x^2}f-f\big)\|_{L^{\infty}} \leq C|t|\langle t \rangle^n \|f\|_{\mathcal{X}}.
    \end{equation}
\end{lemma}

\begin{proof}
Writing (by the fundamental theorem of calculus, or mean value theorem)
\begin{equation*}
    e^{it \partial_x^2}f=f+i\int_0^t e^{i s \partial_x^2}\partial_x^2 f \, ds,
\end{equation*}
and multiplying the above identity by $\langle x \rangle^n$, we apply Sobolev embedding $H^{1}(\mathbb{R})\hookrightarrow L^{\infty}(\mathbb{R})$ and Lemma \ref{derivexp2} to get
\begin{equation*}
    \begin{aligned}
    \|\langle x \rangle^n e^{it \partial_x^2}f\|_{L^{\infty}}&\leq  \|\langle x \rangle^n f\|_{L^{\infty}}+C\int_0^t \big(\|\langle x \rangle^{n}e^{i s \partial_x^2} \partial_x^2 f\|_{L^2}+\|\langle x \rangle^{n}e^{i s \partial_x^2} \partial_x^3 f\|_{L^2}\big)\, ds \\
    & \leq  \|\langle x \rangle^n f\|_{L^{\infty}}+|t|\langle t \rangle^n C\big(\|{J^{n+2} \partial_xf}\|_{L^2}+\sum_{j=1}^3\|\langle x \rangle^{n} \partial_x^j f\|_{L^2}\big) \\
        & \leq  C(1+|t|\langle t \rangle^n )\|f\|_{\mathcal{X}}. 
    \end{aligned}
\end{equation*}
Then { for $r \geq 1$} Lemma \ref{derivexp2} yields
\begin{equation*}
        \sum_{k=1}^{r}\|\langle x \rangle^n e^{it\partial_x^2}\partial_x^k f\|_{L^2} \leq C\langle t \rangle^n \big(\| {J^{n+r-1} \partial_x f } \|_{L^2}+\sum_{k=1}^{r}\|\langle x \rangle^n\partial_x^k f\|_{L^2}\big)
        \leq C\langle t \rangle^n \|f\|_{\mathcal{X}}.
\end{equation*}
Combining the previous estimates, implies \eqref{eqliner1} and \eqref{eqliner2}. 
\end{proof}

Next, we would like to obtain the nonlinear estimates, but prior we need the following interpolation result. 
\begin{lemma}\label{interlemma}
Let $M,n>0$, $r\in \mathbb{Z}^{+}$ with $M>r$. Assume that $f\in \dot{H}^{M}(\mathbb{R})$ and $\langle x \rangle^n \partial^{r}_xf\in L^2(\mathbb{R})$\footnote{Notice that $\langle x \rangle^n \partial^{r}_xf\in L^2(\mathbb{R})$ implies that $ \partial^{r}_xf\in L^2(\mathbb{R})$. Thus, it follows that $f\in \dot{H}^M(\mathbb{R})\cap \dot{H}^r(\mathbb{R})$ (and $\partial_x^l u \in L^2(\mathbb{R})$, for any $r\leq l\leq M$).}. Then for any integer $l$ with $r\leq l\leq M$, there exists $C>0$ such that
\begin{equation*}
    \|\langle x\rangle^{n\big(\frac{M-l}{M-r}\big)} \partial^{l}_x f\|_{L^2}\leq C \|{ J^{M-r} \partial_x^rf} \|^{\frac{l-r}{M-r}}_{L^2}\|\langle x \rangle^n \partial_x^r f\|_{L^2}^{\frac{M-l}{M-r}}.
\end{equation*}
\end{lemma}
\begin{proof}
Write $l=l_1+r$ and denote $n_1=n\big(\frac{M-l}{M-r}\big)\equiv n\big(\frac{M-l_1-r}{M-r}\big)$. Distributing the derivative of order $l_1$, we get
\begin{equation*}
\langle x \rangle^{n_1}\partial_x^l f=\langle x \rangle^{n_1}\partial_x^{l_1}\partial_x^{r}f=\sum_{k_1+k_2=l_1}c_{k_1,k_2}\partial_x^{k_1}\big(\partial_x^{k_2}(\langle x\rangle^{n_1})\partial_x^r f\big)
\end{equation*}
for some constants $c_{k_1,k_2}\in \mathbb{R}$. By using the fact that $|\partial_x^j\big(\langle x\rangle^{-n_1}\partial_x^{k_2}(\langle x\rangle^{n_1})\big)|\lesssim 1$ for each of the integers $k_2,j\geq 0$, we distribute the derivative of order $k_1\leq l_1$ to obtain
\begin{equation}
\|\partial_x^{k_1}\big(\partial_x^{k_2}(\langle x\rangle^{n_1})\partial_x^r f\big)\|_{L^2}= \|\partial_x^{k_1}\Big(\frac{\partial_x^{k_2}(\langle x\rangle^{n_1})}{\langle x\rangle^{n_1}} \langle x\rangle^{n_1}\partial_x^r f\Big)\|_{L^2} 
\leq C  \|J^{l_1}(\langle x \rangle^{n_1} \partial_x^{r}f)\|_{L^2}. \label{E:3}
\end{equation}
To complete the estimate of the inequality above, we recall the following interpolation inequality from \cite{NahasPonce2009}: for $a,b>0$, and $\theta\in (0,1)$, we have
\begin{equation*}
\|J^{\theta a}\big(\langle x \rangle^{(1-\theta)b}f\big)\|_{L^2}\leq C \|J^{a} f\|_{L^2}^{\theta} \|\langle x \rangle^{b}f\|^{1-\theta}_{L^2}.
\end{equation*}
Applying this to \eqref{E:3}, we get  
\begin{equation*}
\begin{aligned}
 \|J^{l_1}(\langle x \rangle^{n_1} \partial_x^{r}f)\|_{L^2}\leq C  \|{J^{M-r} \partial_x^rf}\|^{\frac{l_1}{M-r}}_{L^2}\|\langle x \rangle^n \partial_x^r f\|_{L^2}^{\frac{M-r-l_1}{M-r}},  
\end{aligned}
\end{equation*}
thus, completing the proof.
\end{proof}

We now deduce the key nonlinear estimates, which involves estimating each part in the space $\mathcal X$ separately. 

\begin{lemma}\label{lemmanonlEs} 
Let $u\in \mathcal{X}$ and $\lambda>0$ be such that
\begin{equation}\label{cond1}
     \inf_{x\in \mathbb{R}}|\langle x \rangle^n u(x)|\geq \lambda>0.
\end{equation}
Then there exists a constant $C=C(M,r,n)>0$, independent of $\alpha$, such that
\begin{equation}\label{E:nonlin}
    \||u|^{\alpha}u\|_{\mathcal{X}}\leq C\|u\|_{\mathcal{X}}^{\alpha+1}+
    \sum_{1\leq \beta \leq M}C_{\beta}\big(\lambda^{-|\alpha-2\beta|}+\|u\|_{\mathcal{X}}^{|\alpha-2\beta|}\big)\|u\|_{\mathcal{X}}^{2\beta+1},
\end{equation}
where
\begin{equation}\label{constantdepen}
    |C_{\beta}|\leq C_1\underbrace{\Big|\frac{\alpha}{2}\big(\frac{\alpha}{2}-1\Big)\dots\Big(\frac{\alpha}{2}-(\beta-1)\Big)\Big|}_{\beta-\text{times}},
\end{equation}
with $C_1=C_1(M,r,n)>0$ independent of $\alpha$ for all $1\leq \beta \leq M$.

Moreover, if also $v\in \mathcal{X}$ satisfies \eqref{cond1}, then
\begin{equation} \label{E:nonlin-dif}
\begin{aligned}
    \big\||u|^{\alpha}u -|v|^{\alpha}v \big\|_{\mathcal{X}} &
    \leq C |\alpha|\big(\lambda^{-|\alpha-1|}+(\|u\|_{\mathcal{X}}+\|v\|_{\mathcal{X}})^{|\alpha-1|}\big)\|u\|_{\mathcal{X}}\|u-v\|_{\mathcal{X}}
+C\big(\|u\|_{\mathcal{X}}+\|v\|_{\mathcal{X}}\big)^{\alpha}\|u-v\|_{\mathcal{X}}\\
    &+\sum_{1\leq \beta\leq M} C_{\beta}\bigg(|\frac{\alpha}{2}-\beta|\big(\lambda^{-|\alpha-2\beta-1|}+(\|u\|_{\mathcal{X}}+\|v\|_{\mathcal{X}})^{|\alpha-2\beta-1|}\big)\big(\|u\|_{\mathcal{X}}+\|v\|_{\mathcal{X}}\big)^{2\beta+1}\\
    &\hspace{1.2cm} +\big(\lambda^{-|\alpha-2\beta|}+(\|u\|_{\mathcal{X}}+\|v\|_{\mathcal{X}})^{|\alpha-2\beta|}\big)\big(\|u\|_{\mathcal{X}}+\|v\|_{\mathcal{X}}\big)^{2\beta}\bigg)\|u-v\|_{\mathcal{X}},
\end{aligned}
\end{equation}
where $C_{\beta}>0$ satisfies \eqref{constantdepen}.
\end{lemma}

\begin{proof}
Given $\beta\in \mathbb{R}$, the definition of the space $\mathcal{X}$ and the condition \eqref{cond1} imply
\begin{equation}\label{estim1}
    \begin{aligned}
        |u|^{\beta} &\leq \left\{ \begin{aligned}
&\lambda^{\beta}\langle x \rangle^{-n\beta}, \hspace{2cm} \text{ if }  \beta<0, \\
&\|\langle x \rangle^n u\|_{L^{\infty}}^{\beta}\langle x \rangle^{-n\beta}, \qquad \text{ if } \beta\geq 0,
\end{aligned}\right.\\
& \leq (\lambda^{-|\beta|}+\|u\|_{\mathcal{X}}^{|\beta|})\langle x \rangle^{-n\beta},
    \end{aligned}
\end{equation}
and the mean value inequality gives
\begin{equation}\label{estim2}
\begin{aligned}
\Big||u|^\beta-|v|^\beta\Big|& \leq \left\{ \begin{aligned}
&|\beta|\lambda^{\beta-1}\langle x \rangle^{-n(\beta-1)}|u-v|, \hspace{5cm} \text{ if }  \beta<1, \\
&|\beta|(\max\big\{\|\langle x \rangle^n u\|_{L^{\infty}_x}^{\beta-1},\|\langle x \rangle^n v\|_{L^{\infty}_x}^{\beta-1}\big\})\langle x \rangle^{-n(\beta-1)}|u-v|, \quad \text{ if } \beta\geq 1,
\end{aligned}\right.\\
& \leq |\beta|(\lambda^{-|\beta-1|}+(\|u\|_{\mathcal{X}}+\|v\|_{\mathcal{X}})^{|\beta-1|})\langle x \rangle^{-n(\beta-1)}|u-v|.
\end{aligned}
\end{equation}

In order to prove the nonlinear estimate \eqref{E:nonlin} and the difference \eqref{E:nonlin-dif}, we divide our arguments according to the norms in the definition of the space $\mathcal{X}$.

{$\bullet$ \bf $L^{\infty}$-norm in weighted space}. By using that $1\leq \langle x \rangle^n$, we have
\begin{equation*}
    \|\langle x \rangle^n|u|^{\alpha}u\|_{L^{\infty}}\leq \|u\|_{L^{\infty}}^{\alpha}\|\langle x \rangle^n u\|_{L^{\infty}}\leq \|u\|_{\mathcal{X}}^{\alpha+1}.
\end{equation*}
For the weighted difference, we apply \eqref{estim2} to obtain
\begin{equation*}
\begin{aligned}
\|\langle x \rangle^n\big(|u|^{\alpha}u- &|v|^{\alpha}v\big)\|_{L^{\infty}}
    \leq \|\langle x \rangle^n\big(|u|^{\alpha}-|v|^{\alpha}\big)u\|_{L^{\infty}}+\|\langle x \rangle^n|v|^{\alpha}\big(u-v\big)\|_{L^{\infty}}\\
    & \leq \||u|^{\alpha}-|v|^{\alpha}\|_{L^{\infty}}\|\langle x \rangle^n u\|_{L^{\infty}}+\|v\|^{\alpha}_{L^{\infty}}\|\langle x \rangle^n(u-v)\|_{L^{\infty}}\\
    & \leq |\alpha|\|\langle x\rangle^{-n\alpha}\|_{L^{\infty}}\big(\lambda^{-|\alpha-1|}+(\|u\|_{\mathcal{X}}+\|v\|_{\mathcal{X}})^{|\alpha-1|}\big)\|\langle x \rangle^n(u-v)\|_{L^{\infty}}\|\langle x \rangle^n u\|_{L^{\infty}}
    +\|v\|^{\alpha}_{\mathcal{X}}\|u-v\|_{\mathcal{X}}\\
    & \leq |\alpha|\big(\lambda^{-|\alpha-1|}+(\|u\|_{\mathcal{X}}+\|v\|_{\mathcal{X}})^{|\alpha-1|}\big)\|u\|_{\mathcal{X}}\|u-v\|_{\mathcal{X}} +\|v\|^{\alpha}_{\mathcal{X}}\|u-v\|_{\mathcal{X}}.
\end{aligned}
\end{equation*}

{$\bullet$ \bf $L^{2}$-norm derivatives in weighted space}. 
Let $1 \leq k \leq r$. Using the Leibniz's rule, we  express an order $k$  derivative $\partial_{x}^{k}(|u|^{\alpha} u)$ as
\begin{equation}{\label{NLderivative}}
 \partial_{x}^k(|u|^{\alpha} u) = \sum_{\gamma=1}^k \binom{k}{\gamma} \partial_{x}^{\gamma}(|u|^{\alpha})\partial_{x}^{k-\gamma}u + |u|^{\alpha}\partial_{x}^{k}u,   
\end{equation}
where the derivative of $|u|^{\alpha}$ could be written as
\begin{equation}\label{GeneralNLderivative}
    \partial_{x}^{\gamma}(|u|^{\alpha}) = \sum_{\beta=1}^{\gamma}|u|^{\alpha - 2\beta}\Big(\sum_{\substack{\ell_{1}+...+\ell_{\beta} = \gamma \\ \ell_{j} \geq 1}} c_{\gamma,\ell_{1},\dots,\ell_{\beta}} \partial_{x}^{\ell_{1}}(|u|^{2})...\partial_{x}^{\ell_{\beta}}(|u|^{2})\Big),
\end{equation}
with the constant $c_{\gamma,\ell_1,\dots,\ell_{\beta}}$ upper bounded as
\begin{equation}\label{constantcond}
    |c_{\gamma,\ell_1,\dots,\ell_{\beta}}|\leq 2^{\gamma}\underbrace{\big|\tfrac{\alpha}{2}\big(\tfrac{\alpha}{2}-1\big)\dots \big(\tfrac{\alpha}{2}-(\beta-1)\big)\big|}_{\beta-\text{times}}.
\end{equation}
In particular, $|c_{\gamma,\ell_1}|\leq 2^{\gamma}|\frac{\alpha}{2}|$ for each integer $l_1$. Then
\begin{equation}\label{estimexpre1}
\begin{aligned}
 \|\langle x \rangle^n \partial_x^k(|u|^{\alpha}u)\|_{L^{2}} 
        \leq & \sum_{\gamma=1}^k\sum_{\beta=1}^{\gamma}\sum_{\substack{\ell_1+\dots+\ell_{\beta}=\gamma \\ \ell_j \geq 1}{}} \tbinom{k}{\gamma}|c_{\gamma,\ell_1,\dots,\ell_{\beta}}| \, \big\|\langle x \rangle^{n}|u|^{\alpha-2\beta}\partial_{x}^{\ell_{1}}(|u|^{2})...\partial_{x}^{\ell_{\beta}}(|u|^{2}) \partial_x^{k-\gamma}u \big\|_{L^{2}}\\
        &+\|\langle x \rangle^n|u|^{\alpha}\partial_x^k u\|_{L^{2}}.
    \end{aligned}
\end{equation}
Observe that the last term is easily bounded as
\begin{equation}\label{eqfirst1}
    \|\langle x \rangle^n|u|^{\alpha}\partial_x^k u\|_{L^{2}}\leq \|u\|_{L^{\infty}}^{\alpha}\|\langle x \rangle^n \partial_x^k u\|_{L^2}\leq \|u\|_{\mathcal{X}}^{\alpha+1}.
\end{equation}
The term in the sum of \eqref{estimexpre1} when $\gamma=k$ and $\beta=1$ is estimated, using \eqref{estim1}, as
\begin{align}
\|\langle x \rangle^{n} |u|^{\alpha-2}\partial_x^k(|u|^2)u\|_{L^2}&
\leq (\lambda^{-|\alpha-2|}+\|u\|_{\mathcal{X}}^{|\alpha-2|})\|\langle x\rangle^{-n(\alpha-2)+n}\partial_x^k(|u|^2)u\|_{L^{2}} \label{E:first1} \\
     &\leq \sum_{m=0}^k \tbinom{k}{m} (\lambda^{-|\alpha-2|}+\|u\|_{\mathcal{X}}^{|\alpha-2|})\|\langle x\rangle^{-n(\alpha-2)+n}\partial_x^m u \partial_x^{k-m}u \,u\|_{L^{2}} \notag \\
      &\leq 2(\lambda^{-|\alpha-2|}+\|u\|_{\mathcal{X}}^{|\alpha-2|})\|\langle x \rangle^{-n\alpha}\|_{L^{\infty}}\|\langle x\rangle^n u\|_{L^{\infty}}^2\|\langle x\rangle^n \partial_x^k u\|_{L^2} \label{E:penultimate1}\\
      &\qquad+\sum_{m=1}^{k-1} \tbinom{k}{m} (\lambda^{-|\alpha-2|}+\|u\|_{\mathcal{X}}^{|\alpha-2|})\|\langle x\rangle^{-n(\alpha-2)+n}\partial_x^m u \partial_x^{k-m}u \,u\|_{L^{2}}. \label{E:last1}
\end{align} 
Using the Sobolev embedding $H^1(\mathbb{R}) \hookrightarrow L^{\infty}(\mathbb{R})$, we obtain
\begin{equation}\label{suppcond}
    \sup_{1\leq \beta\leq r-1}\|\langle x \rangle^n\partial_x^{\beta}u\|_{L^{\infty}}\leq (n+1)C\sum_{\beta=1}^r\|\langle x \rangle^n\partial_x^{\beta}u\|_{L^{2}} \leq (n+1)C\|u\|_{\mathcal{X}},
\end{equation}
for some universal constant $C\geq 1$. Then, when $1\leq m\leq k-1$, the last term in \eqref{E:last1} is bounded as
\begin{equation*}
    \begin{aligned}
        \|\langle x\rangle^{-n(\alpha-2)+n}\partial_x^m u \partial_x^{k-m}u \, u\|_{L^{2}} &\leq \|\langle x\rangle^{-n\alpha}\|_{L^{\infty}}\|\langle x\rangle^{n} \partial_x^{k-m}u \|_{L^{\infty}}\|\langle x\rangle^{n} u\|_{L^{\infty}}\|\langle x\rangle^{n}\partial_x^m u \|_{L^{2}} 
        \leq (n+1)C\|u\|_{\mathcal{X}}^3.
    \end{aligned}
\end{equation*}
Substituting the last two estimates into \eqref{E:penultimate1} and \eqref{E:last1}, the upper bound for \eqref{E:first1} yields 
\begin{equation}\label{eqfirst2}
\begin{aligned}
     \|\langle x \rangle^{n}|u|^{\alpha-2}\partial_x^k(|u|^2)u\|_{L^2} &\leq \Big(\lambda^{-|\alpha-2|}+\|u\|_{\mathcal{X}}^{|\alpha-2|}\Big)\Big(2+(n+1)C \sum_{m=1}^{k-1} \tbinom{k}{m} \Big) \|u\|_{\mathcal{X}}^3.
\end{aligned}
\end{equation}
Next, we set $1\leq \gamma \leq k$, $1\leq \beta \leq \gamma$ and $\ell_1+\dots+\ell_{\beta}=\gamma$ with $\ell_j\geq 1$, $j=1,\dots,\beta$, and with $\beta \neq 1$ if $\gamma=k$. Applying \eqref{estim1} to the summand in \eqref{estimexpre1}, we get  
\begin{equation}
    \begin{aligned}
        \|\langle x \rangle^{n} |u|^{\alpha-2\beta}&\partial_{x}^{\ell_{1}}(|u|^{2})...\partial_{x}^{\ell_{\beta}}(|u|^{2}) \partial_x^{k-\gamma}u\|_{L^{2}}\\
        &\leq \big(\lambda^{-|\alpha-2\beta|}+\|u\|_{\mathcal{X}}^{|\alpha-2\beta|}\big)\|\langle x \rangle^{-n(\alpha-2\beta)+n}\partial_{x}^{\ell_{1}}(|u|^{2})...\partial_{x}^{\ell_{\beta}}(|u|^{2}) \partial_x^{k-\gamma}u\|_{L^2}\\
        &\leq\ \big(\lambda^{-|\alpha-2\beta|}+\|u\|_{\mathcal{X}}^{|\alpha-2\beta|}\big)\\
        &\hspace{1cm}\times\sum_{\substack{ 0\leq m_j\leq \ell_j \\ j=1,\dots, \beta}} \tbinom{\ell_1}{m_1}\dots \tbinom{\ell_{\beta}}{m_{\beta}}
        \|\langle x \rangle^{-n(\alpha-2\beta)+n}\partial_{x}^{m_1}u\partial_{x}^{\ell_{1}-m_1}u...\partial_x^{m_{\beta}}u\partial_{x}^{\ell_{\beta}-m_{\beta}}u \partial_x^{k-\gamma}u\|_{L^2}.
    \end{aligned}
\end{equation}
The above restrictions on indexes imply that $\ell_j<k\leq r$ and $k-\gamma<r$. Applying \eqref{suppcond} { and keeping the highest-order derivative in $L^2(\mathbb{R})$}, we get
\begin{equation}\label{eqfirst3}
    \begin{aligned}
        \|\langle x \rangle^{n} |u|^{\alpha-2\beta} &\partial_{x}^{\ell_{1}}(|u|^{2})...\partial_{x}^{\ell_{\beta}}(|u|^{2}) \partial_x^{k-\gamma}u\|_{L^{2}}
         \leq \big(\lambda^{-|\alpha-2\beta|}+\|u\|_{\mathcal{X}}^{|\alpha-2\beta|}\big)\\
        &\hspace{0.5cm}\times\sum_{\substack{ 0\leq m_j\leq \ell_j \\ j=1,\dots, \beta}} \tbinom{\ell_1}{m_1}\dots \tbinom{\ell_{\beta}}{m_{\beta}}
        \|\langle x \rangle^{-n\alpha}\|_{L^{\infty}}\Big(\sum_{m=0}^{r-1}\|\langle x \rangle^{n}\partial_x^m u\|_{L^{\infty}}\Big)^{2\beta}\Big(\sum_{m=1}^{r}\|\langle x \rangle^n \partial_x^m u\|_{L^2}\Big)  \\
         &\leq \big(\lambda^{-|\alpha-2\beta|}+\|u\|_{\mathcal{X}}^{|\alpha-2\beta|}\big)\sum_{\substack{ 0\leq m_j\leq \ell_j \\ j=1,\dots, \beta}} \big((n+1)C\big)^{2\beta}\tbinom{\ell_1}{m_1}\dots \tbinom{\ell_{\beta}}{m_{\beta}}\|u\|_{\mathcal{X}}^{2\beta+1}. 
    \end{aligned}
\end{equation}
Thus, collecting \eqref{eqfirst1}, \eqref{eqfirst2} and \eqref{eqfirst3}, we conclude

\begin{align}\label{eqfisrt4} 
\|\langle x \rangle^n\partial_x^k(|u|^{\alpha}u)\|_{L^2}\leq  \|u\|_{\mathcal{X}}^{\alpha+1}&+C_1\big(2+(n+1)C\sum_{m=1}^{k-1}\tbinom{k}{m}\big)\big(\lambda^{-|\alpha-2|}+\|u\|_{\mathcal{X}}^{|\alpha-2|}\big)\|u\|_{\mathcal{X}}^3 \notag\\
    &+\sum_{\gamma=1}^{k}\sum_{\substack{1\leq \beta \leq \gamma \\ \beta \neq 1, \text{ if } \gamma=k}}\sum_{\substack{\ell_1+\dots+\ell_{\beta}=\gamma \notag \\ 
    \ell_j \geq 1}{}}\tbinom{k}{\gamma} |c_{\gamma,\ell_1,\dots,\ell_{\beta}}|\big(\lambda^{-|\alpha-2\beta|}+\|u\|_{\mathcal{X}}^{|\alpha-2\beta|}\big) \\
    &\hspace{1cm}\times\sum_{\substack{0\leq m_j\leq \ell_j \\ j=1,\dots,\beta}}((n+1)C)^{2\beta}\tbinom{\ell_1}{m_1}\dots \tbinom{\ell_{\beta}}{m_{\beta}}\|u\|_{\mathcal{X}}^{2\beta+1},
\end{align}
where the constant in the first line $0\leq C_1\leq 2^{k}|\alpha|$.

We next compute the same norm of the difference $|u|^{\alpha}u-|v|^{\alpha}v$ for $u, v \in \mathcal{X}$. Let $1\leq k \leq r$. Using \eqref{GeneralNLderivative}, we spit the difference as follows
\begin{equation}\label{identdiff}
\partial_{x}^{k}(|u|^{\alpha}u)- \partial_{x}^{k}(|v|^{\alpha}v)  
=:\mathcal{A}_1+\mathcal{A}_2+\mathcal{A}_3+\mathcal{A}_4,
\end{equation}
where
{\small
\begin{align*}
    \mathcal{A}_1:=&(|u|^{\alpha}-|v|^{\alpha})\partial_x^k u+|v|^{\alpha}\partial_x^k(u-v),\\
\mathcal{A}_2:=&\sum_{\gamma=1}^{k}\sum_{\beta=1}^{\gamma}\sum_{\substack{\ell_1+\dots+\ell_{\beta}=\gamma \\ \ell_j \geq 1}}\tbinom{k}{\gamma}c_{\gamma,\ell_1,\dots,\ell_{\beta}}\big(|u|^{\alpha-2\beta}-|v|^{\alpha-2\beta}\big) \partial_{x}^{\ell_1}(|u|^2)\dots \partial_{x}^{\ell_{\beta}}(|u|^2)\partial_x^{k-\gamma}u, \\
\mathcal{A}_3:=&\sum_{\gamma=1}^{k}\sum_{\beta=1}^{\gamma}\sum_{\substack{\ell_1+\dots+\ell_{\beta}=\gamma \\ \ell_j \geq 1}}\sum_{j=1}^{
    \beta}\tbinom{k}{\gamma}c_{\gamma,\ell_1,\dots,\ell_{\beta}} 
    \times |v|^{\alpha-2\beta} \partial_{x}^{\ell_1}(|v|^2)\dots \partial_{x}^{\ell_j}(|u|^2-|v|^2)\dots \partial_{x}^{\ell_{\beta}}(|u|^2)\partial_x^{k-\gamma}u, \\
    \mathcal{A}_4:=& \sum_{\gamma=1}^{k}\sum_{\beta=1}^{\gamma}\sum_{\substack{\ell_1+\dots+\ell_{\beta}=\gamma \\ 
    \ell_j \geq 1}}\tbinom{k}{\gamma}c_{\gamma,\ell_1,\dots,\ell_{\beta}}|v|^{\alpha-2\beta} \partial_{x}^{\ell_1}(|v|^2)\dots \partial_{x}^{\ell_{\beta}}(|v|^2)\partial_x^{k-\gamma}(u-v),
\end{align*}
}
with the constants $c_{\gamma,\ell_1,\dots,\ell_{\beta}}$ satisfying \eqref{constantcond}. Applying \eqref{estim2} to the first term yields
\begin{align*}
\|\langle & x \rangle^n \mathcal{A}_1\|_{L^2}
\leq |\alpha|\big(\lambda^{-|\alpha-1|}+(\|u\|_{\mathcal{X}}+\|v\|_{\mathcal{X}})^{|\alpha-1|}\big)
       \times\|\langle x \rangle^{-n(\alpha-1)+n}(u-v)\partial^k_x u\|_{L^2}+\|\langle x \rangle^n|v|^{\alpha}\partial_x^k(u-v)\|_{L^2}\\ 
\leq & |\alpha|\big(\lambda^{-|\alpha-1|}+(\|u\|_{\mathcal{X}}+\|v\|_{\mathcal{X}})^{|\alpha-1|}\big)\|\langle x \rangle^{-n\alpha}\|_{L^{\infty}}\|\langle x \rangle^{n}(u-v)\|_{L^{\infty}}\|\langle x \rangle^n\partial^k_x u\|_{L^2}+\|v\|_{L^{\infty}}^{\alpha}\|\langle x \rangle^n\partial_x^k(u-v)\|_{L^2}\\
\leq & |\alpha|\big(\lambda^{-|\alpha-1|}+(\|u\|_{\mathcal{X}}+\|v\|_{\mathcal{X}})^{|\alpha-1|}\big)\|u\|_{\mathcal{X}}\|u-v\|_{\mathcal{X}}+\|v\|_{\mathcal{X}}^{\alpha}\|u-v\|_{\mathcal{X}}.
\end{align*}
We apply \eqref{estim2}, \eqref{suppcond} and similar arguments to those used in the deduction of \eqref{eqfisrt4} to get
{\small
\begin{align*}
\|\langle & x \rangle^n \mathcal{A}_2\|_{L^2}
\leq \sum_{\gamma=1}^k\sum_{\beta=1}^{\gamma}\sum_{\substack{\ell_1
+\dots+\ell_{\beta}=\gamma \\ \ell_j \geq 1}}\sum_{\substack{0\leq m_j \leq \ell_j \\ 1\leq j \leq \beta}} |\alpha-2\beta||c_{\gamma,\ell_1,\dots,\ell_{\beta}}|\tbinom{k}{\gamma}\tbinom{\ell_1}{m_1}\dots \tbinom{\ell_{\beta}}{m_{\beta}}\\
&\times\big(\lambda^{-|\alpha-2\beta-1|}+(\|u\|_{\mathcal{X}}+\|v\|_{\mathcal{X}})^{|\alpha-2\beta-1|}\big)
        \times\|\langle x \rangle^n(u-v)\|_{L^{\infty}}\Big(\sum_{m=0}^{r-1} \|\langle x \rangle^{n} u\|_{L^{\infty}} \Big)^{2\beta}\Big({\sum_{m=1}^{r} \|\langle x \rangle^{n}\partial_x^m u\|_{L^{2}} }\Big)\\
          \leq & \sum_{\gamma=1}^k\sum_{\beta=1}^{\gamma}\sum_{\substack{\ell_1+\dots+\ell_{\beta}=\gamma \\ \ell_j \geq 1}}\sum_{\substack{0\leq m_j \leq \ell_j \\ 1\leq j \leq \beta}} ((n+1)C)^{2\beta}|\alpha-2\beta||c_{\gamma,\ell_1,\dots,\ell_{\beta}}|\tbinom{k}{\gamma}\tbinom{\ell_1}{m_1}\dots \tbinom{\ell_{\beta}}{m_{\beta}}\\
        &\hspace{4cm}\times\big(\lambda^{-|\alpha-2\beta-1|}+(\|u\|_{\mathcal{X}}+\|v\|_{\mathcal{X}})^{|\alpha-2\beta-1|}\big)\|u\|_{\mathcal{X}}^{2\beta+1}\|u-v\|_{\mathcal{X}},
\end{align*}
}
where in the above estimates, we have assigned the $L^2$-norm for the term with the highest order derivative. 
On the other hand, writing
\begin{equation*}
    \partial_x^{\ell_j}(|u|^2-|v|^2)=\partial_x^{\ell_j}\big((u-v)\overline{u}\big)+\partial_x^{\ell_j}\big(v(\overline{u}-\overline{v})\big),
\end{equation*}
we apply \eqref{estim2} together with the above arguments to obtain
\begin{align*}
\|\langle x \rangle^n \mathcal{A}_3\|_{L^2}+\|\langle x \rangle^n\mathcal{A}_4\|_{L^2} 
          \leq & \sum_{\gamma=1}^k\sum_{\beta=1}^{\gamma}\sum_{\substack{\ell_1+\dots+\ell_{\beta}=\gamma \\ \ell_j \geq 1}}\sum_{\substack{0\leq m_j \leq \ell_j \\ 1\leq j \leq \beta}} 2k^2((n+1)C)^{2\beta}|c_{\gamma,\ell_1,\dots,\ell_{\beta}}|\tbinom{k}{\gamma}\tbinom{\ell_1}{m_1}\dots \tbinom{\ell_{\beta}}{m_{\beta}}\\
        &\hspace{.1cm}\times\big(\lambda^{-|\alpha-2\beta|}+(\|u\|_{\mathcal{X}}+\|v\|_{\mathcal{X}})^{|\alpha-2\beta|}\big)\big(\|u\|_{\mathcal{X}}+\|v\|_{\mathcal{X}}\big)^{2\beta}\|u-v\|_{\mathcal{X}}.
    \end{align*}
Collecting above estimates for $\mathcal{A}_j$, $j=1,\dots, 4$, completes the bound for $\|\langle x\rangle^n\partial_x^k(|u|^{\alpha}u-|v|^{\alpha}v)\|_{L^2}$, $1\leq k \leq r$.

{$\bullet$ \bf Unweighted $L^2$-norm estimates.} We first notice that the definition of the space $\mathcal{X}$, together with the interpolation inequality Lemma \ref{interlemma} imply that for all $1\leq l\leq M$, there exists a positive integer $j_l$  with $1\leq j_l\leq r$ such that
\begin{equation}\label{interequ}
   \| \langle x\rangle^{n\big(\frac{M-l}{M-j_l}\big)}\partial_x^l u\|_{L^2}\leq C\|u\|_{\mathcal{X}},
\end{equation}
where $C>0$ depends on $M,r,n$, but is independent of $\alpha$.  For example, when $1\leq l\leq r$, using the definition of the space $\mathcal{X}$, one can take $j_l=l$, and $C=1$ in \eqref{interequ}. When $l\geq r+1$, by Lemma \ref{interlemma}, one can take $j_l=r$ in \eqref{interequ}. Thus, for given $1\leq l\leq M-1$, Sobolev embedding and \eqref{interequ} yield
\begin{equation}\label{interequ2}
\begin{aligned}
   \| \langle x\rangle^{n\big(\frac{M-(l+1)}{M-j_{l+1}}\big)}\partial_x^l u\|_{L^{\infty}}\leq & C\| \langle x\rangle^{n\big(\frac{M-l}{M-j_{l}}\big)}\partial_x^l u\|_{L^{2}}+C\| \langle x\rangle^{n\big(\frac{M-(l+1)}{M-j_{l+1}}\big)}\partial_x^{l+1} u\|_{L^{2}} 
   \leq C\|u\|_{\mathcal{X}}.    
\end{aligned}
\end{equation}
We estimate \eqref{interequ2}: setting $r+1\leq k \leq M$, without incorporating the weight $\langle x \rangle^n$, we argue as in \eqref{estimexpre1} to obtain
\begin{equation*}
    \begin{aligned}
        \| \partial_x^k (|u|^{\alpha}u)\|_{L^{2}}
        \leq &\sum_{\gamma=1}^k\sum_{\beta=1}^{\gamma}\sum_{\substack{\ell_1+\dots+\ell_{\beta}=\gamma \\ \ell_j \geq 1}{}} \sum_{\substack{ 0\leq m_j\leq \ell_j \\ j=1,\dots, \beta}} \tbinom{\ell_1}{m_1}\dots \tbinom{\ell_{\beta}}{m_{\beta}} \tbinom{k}{\gamma}|c_{\gamma,\ell_1,\dots,\ell_{\beta}}|\\
        &\hspace{.2cm} \times \big\| |u|^{\alpha-2\beta}\partial_{x}^{m_1}u \partial_x^{\ell_1-m_1}u\dots\partial_{x}^{m_{\beta}}u\partial_{x}^{\ell_{\beta}-m_{\beta}}u \partial_x^{k-\gamma}u \big\|_{L^{2}}+\big\| |u|^{\alpha} \partial_x^k u \big\|_{L^{2}}.
    \end{aligned}
\end{equation*}
To estimate the above expression, we note that the factors in the above sum are of the form 
\begin{equation}\label{restricond0}
|u|^{\alpha-2\beta}\underbrace{\partial_{x}^{\nu_1}u\dots \partial_{x}^{\nu_{2\beta+1}}u}_{2\beta+1-\text{terms}},
\end{equation}
where
\begin{equation}\label{restricond}
\left\{\begin{aligned}
&r+1\leq k \leq M, \\
&1\leq \beta\leq k, \\
&\nu_1+\nu_2+\dots +\nu_{2\beta+1}=k.
\end{aligned}\right.    
\end{equation}
Without lost of generality, we can assume that the highest order derivative is $\nu_{2\beta+1}$ in \eqref{restricond0}. Then keeping the $L^2$-norm of $\partial_{x}^{\nu_{2\beta+1}}u$ with the weight provided by \eqref{interequ} and taking the $L^{\infty}$-norm of all other terms $\partial_{x}^{\nu_{w}}u$, $w=1,\dots, 2\beta$, with the weight given by \eqref{interequ2}, yields 
\begin{equation*}
\begin{aligned}
\big\| |u|^{\alpha-2\beta} & \partial_{x}^{\nu_1}u\dots \partial_{x}^{\nu_{2\beta+1}}u \big\|_{L^2} \\
\leq & (\lambda^{-|\alpha-2\beta|}+\|u\|_{\mathcal{X}}^{|\alpha-2\beta|}) \|\langle x \rangle^{-n(\alpha-2\beta)}\partial_{x}^{\nu_1}u\dots \partial_{x}^{\nu_{2\beta+1}}u\|_{L^2}\\ 
\leq & C^{2\beta}(\lambda^{-|\alpha-2\beta|}+\|u\|_{\mathcal{X}}^{|\alpha-2\beta|})\|\langle x\rangle^{-n(\alpha-2\beta)-n\Big(\frac{M-(\nu_1+1)}{M-j_{(\nu_1+1)}}\Big)\dots -n\Big(\frac{M-(\nu_{2\beta}+1)}{M-j_{(\nu_{2\beta}+1)}}\Big)-n\Big(\frac{M-\nu_{2\beta+1}}{M-j_{(\nu_{2\beta+1})}}\Big)}\|_{L^{\infty}}\\
& ~ \times \Big(\sum_{l=1}^{M-1}\| \langle x\rangle^{n\big(\frac{M-(l+1)}{M-j_{l+1}}\big)}\partial_x^l u\|_{L^{\infty}}+\|\langle x\rangle^n u\|_{L^{\infty}}\Big)^{2\beta}\Big( \sum_{l=1}^M\| \langle x\rangle^{n\big(\frac{M-l}{M-j_l}\big)}\partial_x^l u\|_{L^2}\Big),
\end{aligned}    
\end{equation*}
where we also used \eqref{estim1}. Consequently, the above expression is bounded as desired (with a constant independent of $\alpha>0$), if 
\begin{equation*}
\Gamma :=2\beta-\alpha-\Big(\frac{M-(\nu_1+1)}{M-j_{(\nu_1+1)}}\Big)\dots -\Big(\frac{M-(\nu_{2\beta}+1)}{M-j_{(\nu_{2\beta}+1)}}\Big)-\Big(\frac{M-\nu_{2\beta+1}}{M-j_{(\nu_{2\beta+1})}}\Big)\leq 0.    
\end{equation*}
Observe that $\Gamma\leq 0$ holds true, since $1\leq j_{(\nu_m+1)}, j_{(\nu_{2\beta+1})}\leq r$, $m=1\dots,2\beta$, and \eqref{restricond} imply
\begin{equation*}
\begin{aligned}
 \Gamma\leq& 2\beta-\alpha-\frac{1}{M-1}\big((2\beta+1)M-l_1-\dots-l_{2\beta+1}-2\beta\big)\\
=& 2\beta-\alpha-\frac{1}{M-1}\big((2\beta+1)M-k-2\beta\big) \\
=& -\alpha-1+\frac{(k-1)}{M-1}\leq -\alpha< 0.
\end{aligned}
\end{equation*}
Summarizing, we conclude that there exists a constant $C>0$ independent of $\alpha$ such that
\begin{equation}\label{eqfisrt5}
\begin{aligned}
\|\partial_x^k(|u|^{\alpha}u)\|_{L^2}
\leq 
\sum_{\gamma=1}^k\sum_{\beta=1}^{\gamma}
\sum_{\substack{\ell_1+\dots+\ell_{\beta}=\gamma \\ 
\ell_j \geq 1}{}} 
&\sum_{\substack{ 0\leq m_j\leq \ell_j \\ j=1,\dots, \beta}} \tbinom{\ell_1}{m_1}\dots \tbinom{\ell_{\beta}}{m_{\beta}} \tbinom{k}{\gamma}|c_{\gamma,\ell_1,\dots,\ell_{\beta}}|(\lambda^{-|\alpha-2\beta|}+\|u\|_{\mathcal{X}}^{|\alpha-2\beta|})\\
        & \hspace{3cm}
        \times C^{2\beta}\|u\|_{\mathcal{X}}^{2\beta+1}
        +\|u\|_{\mathcal{X}}^{\alpha+1}.
    \end{aligned}
\end{equation}
Next, we compute the $H^s$-norm of the difference $|u|^{\alpha}u-|v|^{\alpha}v$, for $u,v \in \mathcal{X}$. For $0\leq k \leq M$ the estimate follows from the decomposition \eqref{identdiff}, and proceeding as in each $\mathcal{A}_j$ estimate, $1\leq j\leq 4$, as well as the analysis carried out for \eqref{eqfisrt5}. A key observation is that the interpolation inequalities \eqref{interequ} and \eqref{interequ2} are also valid for the difference $u-v$. 
In summary, we conclude
\begin{align*}
\|\partial_x^k(|u|^{\alpha}u-|v|^{\alpha}v)\|_{L^2} 
        & \, \leq \, |\alpha|\big(\lambda^{-|\alpha-1|}+(\|u\|_{\mathcal{X}}+\|v\|_{\mathcal{X}})^{|\alpha-1|}\big)\|u\|_{\mathcal{X}}\|u-v\|_{\mathcal{X}} 
+\|v\|_{\mathcal{X}}^{\alpha}\|u-v\|_{\mathcal{X}}\\
&+\sum_{\gamma=1}^k\sum_{\beta=1}^{\gamma}\sum_{\substack{\ell_1+\dots+\ell_{\beta}=\gamma \\ \ell_j \geq 1}}\sum_{\substack{0\leq m_j \leq \ell_j \\ 1\leq j \leq \beta}} ((n+1)C)^{2\beta}|c_{\gamma,\ell_1,\dots,\ell_{\beta}}|\tbinom{k}{\gamma}\tbinom{\ell_1}{m_1}\dots \tbinom{\ell_{\beta}}{m_{\beta}}\\
        &\hspace{2cm}\times\Big(|\alpha-2\beta|\big(\lambda^{-|\alpha-2\beta-1|}+(\|u\|_{\mathcal{X}}+\|v\|_{\mathcal{X}})^{|\alpha-2\beta-1|}\big)\|u\|_{\mathcal{X}}^{2\beta+1}\\
        &\hspace{2.3cm}+ 2k^2\big(\lambda^{-|\alpha-2\beta|}+(\|u\|_{\mathcal{X}}+\|v\|_{\mathcal{X}})^{|\alpha-2\beta|}\big)\big(\|u\|_{\mathcal{X}}+\|v\|_{\mathcal{X}}\big)^{2\beta}\Big)\|u-v\|_{\mathcal{X}},
\end{align*}
completing the proof of the lemma.
\end{proof}

\section{Proof of local well-posedness in space ${\mathcal X}$
}\label{SectionLWPX}

We are now ready to prove the local well-posedness (Theorem \ref{LWP-X}) in the space ${\mathcal X}$ defined in \eqref{Xspace}.
We use the contraction mapping principle based on the linear and nonlinear estimates in Lemmas \ref{linearEst} and \ref{lemmanonlEs}. 

For $R >0$, $T>0$, and $\lambda>0$, define the space $\mathcal{E}_{R,T}$ by 
    \begin{equation}{\label{Espace}}
        \begin{aligned}
            \mathcal{E}_{R,T}  =  \Big\{v &\in C([-T,T]; \mathcal{X}): \\
            &\sup_{t\in [-T,T]}\|v(t)\|_{\mathcal{X}}=\sup_{t\in [-T,T]}\bigg( \|\langle x \rangle^n v(t)\|_{L^{\infty}} + \sum_{k=1}^{r} \|\langle x \rangle^{n} \partial_{x}^{k} v(t)\|_{L^2} + { \sum_{k=r+1}^{M} \| \partial_{x}^{k} v(t)\|_{L^2}}\bigg) \leq R,\\
            &\inf_{(x,t)\in \mathbb{R}\times[-T,T]}|\langle x \rangle^n v(x,t)| \geq \frac{\lambda}{2}\Big\}.
        \end{aligned}
    \end{equation}
Observe that $\mathcal{E}_{R,T}$, equipped with the function $ \sup \limits_{t\in[-T,T]}\|u(t)-v(t)\|_{\mathcal{X}}$, is a complete metric space. The Duhamel's formula, or the integral equation, associated to \eqref{NLS} is
    \begin{equation}{\label{Phi}}
        \Phi(u(t))(t) =e^{it\partial_x^2}u_0 +i\int_0^t     e^{i(t-\tau)\partial_x^2}\mathcal{N}(u(\tau))u(\tau) \, d\tau,
    \end{equation}
with $\Phi$ acting on the space $\mathcal{X}$. We aim at finding $R>0$ and small enough $T>0$ such that $\Phi$ defines a contraction on the complete metric space $\mathcal{E}_{R,T}$. Lemmas \ref{linearEst} and \ref{lemmanonlEs} imply that there exists a constant $C_1>0$ independent of $\alpha_k$ such that
\begin{equation}\label{estimcontrac1}
    \begin{aligned}
        \|\Phi (u(t))\|_{\mathcal{X}}  
        \leq & C_1\langle t \rangle^{n+1} \|u_0\|_{\mathcal{X}}+C\int_0^{|t|} \langle t-\tau \rangle^{n+1}\|\mathcal{N}(u(\tau))u(\tau)\|_{\mathcal{X}}\, d \tau \\
        \leq & C_1\langle t \rangle^{n+1} \|u_0\|_{\mathcal{X}}+C\sum_{k=0}^{\infty}|a_k|\int_0^{|t|} \langle t-\tau \rangle^{n+1}\||u|^{\alpha_k}(\tau)u(\tau)\|_{\mathcal{X}}\, d \tau \\
       \leq & C_1\langle T \rangle^{n+1} \|u_0\|_{\mathcal{X}}+CT\langle T \rangle^{n+1} \sum_{k=0}^{\infty}|a_k| \big(\sup_{t\in[-T,T]}\|u(t)\|_{\mathcal{X}}\big)^{\alpha_k+1}\\
       &+CT\langle T \rangle^{n+1} \sum_{k=0}^{\infty}\sum_{1\leq \beta\leq M}|a_k||C_{\beta,\alpha_k}| \\
       &\hspace{2cm}\times \Big(\Big(\frac{\lambda}{2}\Big)^{-|\alpha_k-2\beta|}+(\sup_{t\in[-T,T]}\|u(t)\|_{\mathcal{X}})^{|\alpha_k-2\beta|}\Big)\big(\sup_{[-T,T]}\|u(t)\|_{\mathcal{X}}\big)^{2\beta+1}\\
       \leq & C_1\langle T \rangle^{n+1} \|u_0\|_{\mathcal{X}}+T\langle T \rangle^{n+1} \mathcal{G}_1(R),
    \end{aligned}
\end{equation}
where $C>0$ is also independent of $\alpha_k$ for all $k$, and $\mathcal{G}_1(R)$ is defined by
\begin{equation}\label{G1function}
    \begin{aligned}
       \mathcal{G}_1(R):=& C \sum_{k=0}^{\infty}|a_k|R^{\alpha_k+1}+C \sum_{k=0}^{\infty}\sum_{1\leq \beta\leq M}|a_k||C_{\beta,\alpha_k}| \Big(\big(\tfrac{\lambda}{2}\big)^{-|\alpha_k-2\beta|}+R^{|\alpha_k-2\beta|}\Big)R^{2\beta+1}  
    \end{aligned}
\end{equation}
with
\begin{equation}\label{constadiff}
    |C_{\beta,\alpha_k}|\leq \underbrace{\Big|\tfrac{\alpha_k}{2}\Big(\tfrac{\alpha_k}{2}-1\Big)\dots \Big(\tfrac{\alpha_k}{2}-(\beta-1)\Big)\Big|}_{\beta-\text{times}}.
\end{equation}
Setting $R=2^{n+2}C_1\|u_0\|_{\mathcal{X}}$ and $0<T<1$, it follows that $\|\Phi(u(t))\|_{\mathcal{X}}\leq R$ if
\begin{equation}\label{contraceq1}
    \begin{aligned}
        T\langle T \rangle^{n+1} \mathcal{G}_1(R)\leq \frac{R}{2}.
    \end{aligned}
\end{equation}
Thus, there exists a small time $0<T<1$, where the above condition holds provided that
\begin{equation}
    \sum_{k=0}^{\infty}|a_k|R^{\alpha_k}<\infty,
\end{equation}
and
\begin{equation*}
    \sum_{k=0}^{\infty}\sum_{1\leq \beta\leq M}|a_k||C_{\beta,\alpha_k}| \Big(\big(\tfrac{\lambda}{2}\big)^{-|\alpha_k-2\beta|}+R^{|\alpha_k-2\beta|}\Big)<\infty.
\end{equation*}
Note that these conditions hold by the assumption \eqref{coeffcond}. On the other hand, by Lemmas \ref{linearEst} and \ref{lemmanonlEs}, we get
\begin{equation}\label{differenequa1}
\begin{aligned}
\|\Phi (u(t))&- \Phi(v(t))\|_{\mathcal{X}}\\
= &  \Big\| \int_0^t e^{i(t-\tau)\partial_x^2} \Big(\mathcal{N}(u(\tau))u(\tau)-\mathcal{N}(v(\tau))v(\tau)\Big)\, d\tau \Big\|_{\mathcal{X}} \\
\leq & \, CT \langle T \rangle^{n+1}\sum_{k=0}^{\infty} |a_k| |\alpha_k|\Big\{\big(\tfrac{\lambda}{2}\big)^{-|\alpha_k-1|}
+\big(\sup_{t\in [-T,T]}\|u(t)\|_{\mathcal{X}}+\sup_{t\in [-T,T]}\|v(t)\|_{\mathcal{X}})^{|\alpha_k-1|}\big) \Big\}\\
    &\hspace{4cm}\times\big(\sup_{t\in [-T,T]}\|u(t)\|_{\mathcal{X}}\big)\big(\sup_{t\in [-T,T]}\|u(t)-v(t)\|_{\mathcal{X}}\big)\\
    &+CT \langle T \rangle^{n+1}\sum_{k=0}^{\infty}|a_k|\big(\sup_{t\in[-T,T]}\|u(t)\|_{\mathcal{X}}+\sup_{t\in[-T,T]}\|v(t)\|_{\mathcal{X}}\big)^{\alpha_k}\Big(\sup_{t\in[-T,T]}\|u(t)-v(t)\|_{\mathcal{X}}\Big)\\
  &+T \langle T \rangle^{n+1}\sum_{k=0}^{\infty}\sum_{1\leq \beta\leq M} |a_k|C_{\beta,\alpha_k}\\
    &\hspace{1cm}\times\bigg\{|\tfrac{\alpha_k}{2}-\beta|\Big(\big(\tfrac{\lambda}{2}\big)^{-|\alpha_k-2\beta-1|}+(\sup_{t\in[-T,T]}\|u(t)\|_{\mathcal{X}}+\sup_{t\in[-T,T]}\|v(t)\|_{\mathcal{X}})^{|\alpha_k-2\beta-1|}\Big)\\
    &\hspace{4cm}\times\big(\sup_{t\in[-T,T]}\|u(t)\|_{\mathcal{X}}+\sup_{t\in[-T,T]}\|v(t)\|_{\mathcal{X}}\big)^{2\beta+1}\\
    &\hspace{1cm} +\Big(\big(\tfrac{\lambda}{2}\big)^{-|\alpha_k-2\beta|}+(\sup_{t\in[-T,T]}\|u(t)\|_{\mathcal{X}}+\sup_{t\in[-T,T]}\|v(t)\|_{\mathcal{X}})^{|\alpha_k-2\beta|}\Big)\\
    &\hspace{4cm}\times \big(\sup_{t\in[-T,T]}\|u(t)\|_{\mathcal{X}}+\sup_{t\in[-T,T]}\|v(t)\|_{\mathcal{X}}\big)^{2\beta}\bigg\}\Big(\sup_{t\in[-T,T]}\|u(t)-v(t)\|_{\mathcal{X}}\Big)\\
    \leq & \, T \langle T \rangle^{n+1}\mathcal{G}_2(R)\Big(\sup_{t\in[-T,T]}\|u(t)-v(t)\|_{\mathcal{X}}\Big),
\end{aligned}    
\end{equation}
where
\begin{equation}\label{G2def}
    \begin{aligned}
        \mathcal{G}_2(R)= & C\sum_{k=0}^{\infty} |a_k| |\alpha_k|\Big(\big(\tfrac{\lambda}{2}\big)^{-|\alpha_k-1|}+(2R)^{|\alpha_k-1|}\Big) R+C\sum_{k=0}^{\infty}|a_k|(2R)^{\alpha_k}\\
    &+\sum_{k=0}^{\infty}\sum_{1\leq \beta\leq M} |a_k|C_{\beta,\alpha_k}\bigg(|\tfrac{\alpha_k}{2}-\beta|\Big(\big(\tfrac{\lambda}{2}\big)^{-|\alpha_k-2\beta-1|}+(2R)^{|\alpha_k-2\beta-1|}\Big)(2R)^{2\beta+1} \\
    &\hspace{4cm} 
    +\Big(\big(\tfrac{\lambda}{2}\big)^{-|\alpha_k-2\beta|}+(2R)^{|\alpha_k-2\beta|}\Big)(2R)^{2\beta}\bigg),
    \end{aligned}
\end{equation}
with $C>0$ being independent of $\alpha_k$, and $C_{\beta,\alpha_k}>0$ satisfying \eqref{constadiff}. Then $\mathcal{G}_2(R)<\infty$, if
{\small
\begin{equation*}
    \begin{aligned}
        &\bullet \quad \sum_{k=0}^{\infty} |a_k| |\alpha_k|\Big(\big(\tfrac{\lambda}{2}\big)^{-|\alpha_k-1|}+(2R)^{|\alpha_k-1|}\Big)<\infty, \\
        &
        \bullet \quad \sum_{k=0}^{\infty}|a_k|(2R)^{\alpha_k}<\infty,\\
    & \bullet \quad \sum_{k=0}^{\infty}\sum_{1\leq \beta\leq M} |a_k|C_{\beta,\alpha_k}\Big\{ |\tfrac{\alpha_k}{2}-\beta| \Big(\big(\tfrac{\lambda}{2}\big)^{-|\alpha_k-2\beta-1|}+(2R)^{|\alpha_k-2\beta-1|}\Big) 
    +\Big(\big(\tfrac{\lambda}{2}\big)^{-|\alpha_k-2\beta|}+(2R)^{|\alpha_k-2\beta|}\Big)\Big\}<\infty.
    \end{aligned}
\end{equation*}
}
The above three conditions are satisfied by the assumption \eqref{coeffcond}.  Thus, by taking $T>0$ small such that
\begin{equation}\label{contraceq2} 
    \begin{aligned}
        T \langle T \rangle^{n+1}\mathcal{G}_2(R)\leq \tfrac{1}{2},
    \end{aligned} 
\end{equation}
we obtain
\begin{equation}\label{contraceq2.1} 
   \begin{aligned}
   \sup_{t\in[-T,T]}\|\Phi(u)(t)-\Phi(v)(t)\|_{\mathcal{X}} \leq \frac{1}{2} \sup_{t\in[-T,T]}\|u(t)-v(t)\|_{\mathcal{X}}.
   \end{aligned}
\end{equation}
Next, we use \eqref{eqliner2} and \eqref{estimcontrac1} to deduce
\begin{equation}\label{eqinfcond1}
    \begin{aligned}
        |\langle x \rangle^n \Phi(u)(x,t)| 
        \geq & |\langle x \rangle^n u_0(x)| -\|\langle x \rangle^n(e^{it\partial_x^2}u_0-u_0)\|_{L^{\infty}}-\|\int_0^t e^{i(t-\tau)\partial_x^2}\mathcal{N}(u(\tau))u(\tau)\, d\tau\|_{\mathcal{X}}\\
        \geq & \lambda-CT\langle T \rangle^n\|u_0\|_{\mathcal{X}}-T\langle T \rangle^{n+1}\mathcal{G}_1(R).
    \end{aligned}
\end{equation}
Thus, taking $T>0$ sufficiently small so that
\begin{equation}\label{contraceq3}
    CT\langle T \rangle^nR+T\langle T \rangle^{n+1}\mathcal{G}_1(R)\leq \tfrac{\lambda}{2},
\end{equation}
we obtain 
$$
|\langle x \rangle^n \Phi(u)(x,t)|\geq \tfrac{\lambda}{2}.
$$
Finally, taking $T>0$ small enough such that \eqref{contraceq1}, \eqref{contraceq2} and \eqref{contraceq3} hold true, we establish that $\Phi:\mathcal{E}_{R,T}\rightarrow \mathcal{E}_{R,T}$ defines a contraction on the complete metric space $\mathcal{E}_{R,T}$. Therefore, the existence of a solution for the integral equation associated to \eqref{NLS} follows from Banach fixed-point theorem. 

Let us now prove uniqueness in the class $C([-T,T];\mathcal{X})$ with condition \eqref {infcondsolut} (which is a \textit{larger} class than $\mathcal{E}_{R,T}$ in the contraction mapping argument). Assume that there are two solutions $u, v \in C([-T,T];\mathcal{X})$ of \eqref{NLS} with the same initial condition $u_0 \in \mathcal{X}$, satisfying \eqref{Xinf}. Take $R_1>0$ so that
\begin{equation*}
    \sup_{t\in [-T,T]}\big(\|u(t)\|_{\mathcal{X}}+\|v(t)\|_{\mathcal{X}}\big)\leq R_1.
\end{equation*}
Then by \eqref{infcondsolut}, it follows that
\begin{equation*}
    \inf_{(x,t)\in \mathbb{R}\times [-T,T]}|\langle x \rangle^{n}u(x,t)|\geq \frac{\lambda}{2} \quad \mbox{and} \quad \inf_{(x,t)\in \mathbb{R}\times [-T,T]}|\langle x \rangle^{n}v(x,t)|\geq \frac{\lambda}{2}.
\end{equation*}
Let $0<T_1\leq T$, using the integral formulation of $u$ and $v$, and arguing as in \eqref{differenequa1}, we have that for all $t\in[-T_1,T_1]$, 
\begin{equation}\label{ineqUNiqueness}
\|u(t)- v(t)\|_{\mathcal{X}}\leq  \, T_1 \langle T_1 \rangle^{n+1}\mathcal{G}_2(R_1)\Big(\sup_{t\in[-T_1,T_1]}\|u(t)-v(t)\|_{\mathcal{X}}\Big),   
\end{equation}
with $\mathcal{G}_2(\cdot)$ as in \eqref{G2def}, which we remark is independent of $T,T_1,t$. Then, if $0<T_1\leq T$ is such that $T_1 \langle T_1 \rangle^{n+1}\mathcal{G}_2(R_1)\leq \frac{1}{2}$, we have from \eqref{ineqUNiqueness} that
\begin{equation*}
    u\equiv v \text{ on  } [-T_1,T_1].
\end{equation*}
The idea now is to iterate the previous argument in steps of size $T_1$ until we extend uniqueness to the whole interval $[-T,T]$. To see this, if $T_1=T$, we are done. Otherwise, let $T_2=\min\{T,2T_1\}$. Then for any $T_1\leq t\leq T_2$, we write 
\begin{equation*}
\begin{aligned}
 u(t)=e^{i(t-T_1)\partial_x^2}u(T_1)+i\int_{T_1}^t e^{i(t-\tau)\partial_x^2}\mathcal{N}(u(\tau))u(\tau) \, d\tau, \\
v(t)=e^{i(t-T_1)\partial_x^2}v(T_1)+i\int_{T_1}^t e^{i(t-\tau)\partial_x^2}\mathcal{N}(v(\tau))v(\tau) \, d\tau. \end{aligned}
\end{equation*}
From the previous identities, and using that $u(T_1)=v(T_1)$ along with the argument in \eqref{differenequa1}, it follows that for all $T_1\leq t\leq T_2$
\begin{equation}\label{ineqUNiqueness2}
\begin{aligned}
\|u(t)- v(t)\|_{\mathcal{X}}\leq &  \, (T_2-T_1) \langle T_2-T_1 \rangle^{n+1}\mathcal{G}_2(R_1)\Big(\sup_{t\in[-T_2,T_2]}\|u(t)-v(t)\|_{\mathcal{X}}\Big) \\
\leq & \frac{1}{2}\Big(\sup_{t\in[-T_2,T_2]}\|u(t)-v(t)\|_{\mathcal{X}}\Big),    
\end{aligned}
\end{equation}
where we have also used that $T_2-T_1\leq T_1$, and thus $(T_2-T_1) \langle T_2-T_1 \rangle^{n+1}\mathcal{G}_2(R_1)\leq \frac{1}{2}$. A similar estimate can be deduced for negative times $-T_2\leq t \leq -T_1$. Hence, from \eqref{ineqUNiqueness2} it must follow that $ u\equiv v$ on $[-T_2,T_2]$.

Consequently, since $T>0$ is finite, the previous argument is repeated a finite number of times, say, up to the smallest integer $k$ such that $k\,T_1\geq T$, completing the proof.

Finally, the continuity of the data-to-solution map of \eqref{NLS} follows from the fact that the time of existence of solutions is continuous with respect to the initial data (see \eqref{contraceq1}, \eqref{contraceq2} and \eqref{contraceq3}), and using the same arguments as in the proof of \eqref{contraceq2.1}, we complete the proof of Theorem \ref{LWP-X}.


\section{Global results for the data in space $\mathcal X$ with quadratic phase}
\label{SectionScat}

Following the strategy in \cite{CazNaum2016} to deduce Theorem \ref{scatres}, we apply the pseudo-conformal transformation (invariance in the $L^2$ critical case) to the Cauchy problem \eqref{NLS} (e.g., see \cite[Proposition 4.1]{LinaresPonce2015} or \cite[Exercise 2.28]{Tao2006}). 

We next introduce some useful results from the fractional calculus. 
Let $s\in (0,1)$ and define one of the Stein's fractional derivatives by
\begin{equation}\label{Steinsquare}
\mathcal{D}^s f(x)=\left(\int_{\mathbb{R}}\frac{|f(x)-f(y)|^2}{|x-y|^{1+2s}}\, dy\right)^{1/2}, \quad  x\in \mathbb{R},
\end{equation}
where $f$ is a sufficiently regular function $f$. 
We remind the reader the notation $D^s=(-\partial_x^2)^{\frac{s}{2}}$.

To deal with fractional weights, we recall the following characterization of the spaces $L^p_s(\mathbb{R}^N) \equiv J^{-s}L^p(\mathbb{R}^N)$.

\begin{theorem}[\cite{S1961}]\label{TheoSteDer} 
Let $s\in (0,1)$ and $\frac{2}{1+2s}<p<\infty$. Then $f\in L_s^p(\mathbb{R})$ if and only if
\begin{itemize}
\item[(i)]  $f\in L^p(\mathbb{R})$ and
\item[(ii)]$\mathcal{D}^s f \in L^{p}(\mathbb{R})$
\end{itemize}
with 
\begin{equation*}
\|J^s f\|_{L^p}=\|(1-\partial_x^2)^{\frac{s}{2}} f \|_{L^p} \sim \|f\|_{L^p}+\|\mathcal{D}^s f\|_{L^p} \sim \|f\|_{L^p}+\|D^s f\|_{L^p}.
\end{equation*} 
\end{theorem}
When $p=2$ and $s\in (0,1)$, we use the following property,
\begin{equation} \label{prelimneq} 
\left\|\mathcal{D}^s(fg)\right\|_{L^2} \lesssim \left\|f\mathcal{D}^s g\right\|_{L^2}+\left\|g\mathcal{D}^sf \right\|_{L^2}.
\end{equation}

We will also use the following result from Nahas \& Ponce \cite{NahasPonce2009}.
\begin{theorem}[\cite{NahasPonce2009}]\label{derivexp} 
Let $s\in(0,1)$. For any $|t|>0$,
\begin{equation*}
\mathcal{D}^s\big(e^{it|x|^2}\big)\lesssim |t|^{\frac{s}{2}}+|t|^{s}|x|^s.
\end{equation*}
\end{theorem}

We next apply the pseudo-conformal transformation to \eqref{NLS}, allowing us to obtain a global solution of \eqref{NLS}, which corresponds to a local solution of the following non-autonomous equation  
\begin{equation}\label{NLSsolution}
   \left\{ \begin{aligned}
        &iv_{t} + \partial_{x}^{2}v + \mathcal{N}_1(v) \,v = 0, \\
        &v(x,0) = v_{0}, 
    \end{aligned}\right.
\end{equation}
where, given $b>0$,
\begin{equation}\label{nonlineaNLS2}
    \mathcal{N}_1(v)=\sum_{k=0}^{\infty}d_k (1-bt)^{\frac{\alpha_k}{2}-2}|v|^{\alpha_k}.
\end{equation}
We establish the existence of solutions for \eqref{NLSsolution}.
\begin{theorem}\label{existnonautequ}
Let $n,r,M\in \mathbb{Z}^{+}$ be given by \eqref{deffparame} and consider the space $\mathcal{X}$ defined in \eqref{Xspace}. Additionally, let $v_0 \in \mathcal{X}$ be such that
\begin{equation*}
    \inf_{x\in \mathbb{R}}|v_0(x)|\geq \lambda>0.
\end{equation*}
Let $\{d_k\}$ be a sequence of real numbers, $\{\alpha_k\}$ be a sequence of real numbers such that $\alpha_k>2$, and for all $R_0>0$
	\begin{equation}
		\sum_{k=0}^{\infty}\sum_{0\leq \beta \leq M}\frac{|d_k|}{\alpha_k-2}\big(1+C(\alpha_k,\beta))(R_0^{|\alpha_k-2\beta|}+|\alpha_k-2\beta|R_0^{|\alpha_k-2\beta-1|})<\infty,
	\end{equation}
where $C(\alpha_k,\beta)$ is defined in \eqref{E:C(a,b)}. Then there exists a unique solution $v$ of \eqref{NLSsolution} with initial data $v_0$ such that 
\begin{equation}
    v\in C([0,|b|^{-1}];\mathcal{X}),
\end{equation}
provided that $b>0$ is sufficiently large.
\end{theorem}

\begin{proof}
The proof uses the contraction mapping principle as in Theorem \ref{LWP-X} for the integral operator associated to the Cauchy problem in \eqref{NLSsolution}
\begin{equation}\label{integralcontrac}
    \begin{aligned}
       \Phi_1(v(t))=e^{it \partial_x^2}v_0+i\int_0^t e^{i(t-\tau)\partial_x^2}\mathcal{N}_1(v(\tau))v(\tau)\, d\tau,
    \end{aligned}
\end{equation}
acting on the space 
\begin{equation*}
        \begin{aligned}
            \mathcal{E}_{R,b} = \Big\{ &v \in C([0,b^{-1}], \mathcal{X}): \\
            &\sup_{t\in [0,b^{-1}]}\|v(t)\|_{\mathcal{X}}=\sup_{t\in [0,b^{-1}]}\bigg( \|\langle x \rangle^n v(t)\|_{L^{\infty}} + \sum_{k=1}^{r} \|\langle x \rangle^{n} \partial_{x}^{k} v(t)\|_{L^2} + { \sum_{k=r+1}^{M} \| \partial_{x}^{k} v(t)\|_{L^2}}\bigg) \leq R,\\
            &\inf_{(x,t)\in \mathbb{R}\times[0,b^{-1}]}|\langle x \rangle^n v(x,t)| \geq \frac{\lambda}{2}\Big\}.
        \end{aligned}
    \end{equation*}
Following similar steps as in \eqref{estimcontrac1}, that are based on the linear and nonlinear estimates in Lemmas \ref{linearEst} and \ref{lemmanonlEs}, we obtain 
\begin{equation}\label{pseudeq1}
    \begin{aligned}
        \|\Phi_1 (v(t))\|_{\mathcal{X}}\leq & C_1\langle b^{-1} \rangle^{n+1} \|v_0\|_{\mathcal{X}}+C\sum_{k=0}^{\infty}|a_k|\langle b^{-1} \rangle^{n+1}\int_0^{b^{-1}} (1-b\tau)^{\frac{\alpha_k}{2}-2} \||v|^{\alpha_k}(\tau)v(\tau)\|_{\mathcal{X}}\, d \tau \\
       \leq & C_1\langle b^{-1} \rangle^{n+1} \|v_0\|_{\mathcal{X}}+b^{-1}\langle b^{-1} \rangle^{n+1} \widetilde{\mathcal{G}}_1(R),
    \end{aligned}
\end{equation}
where $0\leq t \leq b^{-1}$ and
\begin{equation}\label{G1functionprim}
    \begin{aligned}
       \widetilde{\mathcal{G}}_1(R):=& C \sum_{k=0}^{\infty}|a_k|\frac{R^{\alpha_k+1}}{\alpha_k-2}+C \sum_{k=0}^{\infty}\sum_{1\leq \beta\leq M}|a_k||C_{\beta,\alpha_k}| \Big(\big(\tfrac{\lambda}{2}\big)^{-|\alpha_k-2\beta|}+R^{|\alpha_k-2\beta|}\Big)\frac{R^{2\beta+1}}{\alpha_k-2},  
    \end{aligned}
\end{equation}
with $C>0$ independent of $\alpha_k$ for all $k$ and $C_{\beta,\alpha_k}$ as in \eqref{constadiff}. Note that above we used 
\begin{equation*}
    \int_0^{b^{-1}} (1-b\tau)^{\frac{\alpha_k}{2}-2}\, d\tau=\frac{2b^{-1}}{\alpha_k-2},
\end{equation*}
which requires that $\alpha_k>2$ for all $k\geq 0$ (hence, the assumption in the theorem). On the other hand, by Lemmas \ref{linearEst}, \ref{lemmanonlEs} and the arguments in \eqref{differenequa1}, we have
\begin{equation}\label{pseudeq2}
\begin{aligned}
    \|\Phi_1(u(t))&-\Phi_1(v(t))\|_{\mathcal{X}}
    \leq& b^{-1} \langle b^{-1} \rangle^{n+1}\widetilde{\mathcal{G}}_2(R)\sup_{t\in[0,b^{-1}]}\|u(t)-v(t)\|_{\mathcal{X}},
\end{aligned}    
\end{equation}
where
\begin{equation*}
\begin{aligned}
  \widetilde{\mathcal{G}}_2(R)= & C\sum_{k=0}^{\infty} |a_k| |\alpha_k|\big(\big(\tfrac{\lambda}{2}\big)^{-|\alpha_k-1|}+(2R)^{|\alpha_k-1|}\big)\frac{R}{\alpha_k-2}+C\sum_{k=0}^{\infty}|a_k|\frac{(2R)^{\alpha_k}}{\alpha_k-2}\\
    &+\sum_{k=0}^{\infty}\sum_{1\leq \beta\leq M} |a_k|C_{\beta,\alpha_k}\Big\{|\tfrac{\alpha_k}{2}-\beta|\big(\big(\tfrac{\lambda}{2}\big)^{-|\alpha_k-2\beta-1|}+(2R)^{|\alpha_k-2\beta-1|}\big)\frac{(2R)^{2\beta+1}}{\alpha_k-2}\\
    &\hspace{3.5cm} +\big(\big(\tfrac{\lambda}{2}\big)^{-|\alpha_k-2\beta|}+(2R)^{|\alpha_k-2\beta|}\big)\frac{(2R)^{2\beta}}{\alpha_k-2}\Big\}.
    \end{aligned}
\end{equation*}
Finally, similarly to \eqref{eqinfcond1}, we obtain
\begin{equation}\label{pseudeq3}
    \begin{aligned}
        |\langle x \rangle^n &\Phi_1(v)(x,t)| \geq & \lambda-Cb^{-1}\langle b^{-1} \rangle^n\|u_0\|_{\mathcal{X}}-b^{-1}\langle b^{-1} \rangle^{n+1}\widetilde{\mathcal{G}}_1(R).
    \end{aligned}
\end{equation}
Therefore, setting $R=2^{n+2}C_1\|v_0\|_{\mathcal{X}}$, we combine \eqref{pseudeq1}, \eqref{pseudeq2} and \eqref{pseudeq3}, to show that for $b \gg 1$, the mapping $\Phi:\mathcal{E}_{R,b}\rightarrow \mathcal{E}_{R,b}$ defines a contraction on the complete metric space $\mathcal{E}_{R,b}$. The desired conclusion follows from the Banach fixed-point theorem.
\end{proof}

We are now ready to prove Theorem \ref{scatres}. { Note that we impose the condition $n>\frac12$ here, since differentiating the quadratic phase (see computations in \eqref{Jestiglobal2}) introduces a weight $\la x \ra^{s_n}$, which will need to be bounded in $L^2(\mathbb R)$.}


\subsection{Proof of Theorem \ref{scatres}}

We consider the initial condition $u_0 = e^{i\frac{b|x|^2}{4}}v_0$ with $v_0 \in \mathcal{X}$ satisfying \eqref{Xinf}. By Theorem \ref{existnonautequ} choose $b>0$ sufficiently large, such that there exists a unique solution $v\in C([0,b^{-1}];\mathcal{X})$ of \eqref{NLSsolution} with the initial condition $v_0$. Define
\begin{equation}{\label{Pseudo2}}       
    \begin{aligned}
        u(x,t) = \frac{1}{(1+b t)^{\frac{1}{2}}}e^{\frac{ib x^2}{4(1+b t)}} \, v\Big(\frac{x}{1+b t}, {\frac{t}{1+ b t}}\Big)
    \end{aligned}
\end{equation}
for all $0\leq t<\infty$, $x\in \mathbb{R}$. Let $s_n=1$ if $n>\frac{3}{2}$, and fix $s_n \in (0,n-\frac{1}{2})$ if $\frac{1}{2}<n\leq \frac{3}{2}$. It follows that
\begin{equation}\label{globalclass}
    u \in C([0,\infty);H^{s_n}(\mathbb{R}))\cap L^{\infty}\big([0,\infty);L^{\infty}(\langle x \rangle^{\frac{1}{2}}\, dx)\big)
\end{equation}
and $u$ solves \eqref{NLS} with $u_0=e^{i\frac{b|x|^2}{4}}v_0$. The fact that $u$ solves \eqref{NLS} is a consequence of the fact that $v$ solves \eqref{NLSsolution}. To verify \eqref{globalclass}, we use \eqref{Pseudo2} to obtain
\begin{equation*}
    \begin{aligned}
        \|\langle x \rangle^{\frac{1}{2}}u(x,t)\|_{L^{\infty}_x}
        \leq & \frac{C}{(1+bt)^{\frac{1}{2}}}\Big(\big\|v(x,{\frac{t}{1+ b t}}) \big\|_{L^{\infty}_x}+(1+bt)^{\frac{1}{2}} \big\| |x|^{\frac{1}{2}} v(x, {\frac{t}{1+ b t}}) \big\|_{L^{\infty}_x} \Big)\\
        \leq & C\, \big\|\langle x \rangle^{\frac{1}{2}}v(x, \frac{t}{1+bt}) \big\|_{L^{\infty}_x}.
    \end{aligned}
\end{equation*}
By taking the supremum on time $t\in [0,\infty)$ in the above inequality, and using the fact that $v\in C([0,b^{-1}];\mathcal{X})$, we conclude that $u\in L^{\infty}\big([0,\infty);L^{\infty}(\langle x \rangle^{\frac{1}{2}}\, dx)\big)$. Next, we show that $u\in C([0,\infty);{  H^{s_n}(\mathbb{R})})$. For now assume that $\frac{1}{2}<n \leq \frac{3}{2}$, and hence, $0<s_n<n-\frac{1}{2}\leq 1$. The proof for $n>\frac{3}{2}$, in which case we set $s_n=1$, follows from similar reasoning as described below, by replacing the fractional derivative by the local derivative, therefore, we omit the analysis of that case. We apply Theorem \ref{TheoSteDer}, together with \eqref{prelimneq} to deduce
\begin{equation}\label{Jestiglobal}
    \begin{aligned}
\|{J^{s_n}}u(x,t)\|_{L^2_x} &
        \leq {\frac{C}{(1+bt)^{\frac{1}{2}}}\|v(\frac{x}{1+bt},\frac{t}{1+bt})\|_{L^2_x}}+\frac{C}{(1+bt)^{\frac{1}{2}}}\|\mathcal{D}^{s_n}\big(e^{\frac{ib x^2}{4(1+bt)}}\big)v(\frac{x}{1+bt},\frac{t}{1+bt})\|_{L^2_x}\\
        &\quad +\frac{C}{(1+bt)^{\frac{1}{2}}}\|D^{s_n}\Big(v(\frac{x}{1+bt},\frac{t}{1+bt})\Big)\|_{L^2_x}\\
        & \leq  \, C \, \|v(x,\frac{t}{1+bt})\|_{L^2_x}+\frac{C}{(1+bt)^{s_n}}\|(D^{s_n}v)(x,\frac{t}{1+bt})\|_{L^2_x}\\
        & \quad +\frac{C}{(1+bt)^{\frac{1}{2}}}\|\mathcal{D}^{s_n}\big(e^{\frac{ib x^2}{4(1+bt)^2}}\big)v(\frac{x}{1+bt},\frac{t}{1+bt})\|_{L^2_x}.
    \end{aligned}
\end{equation}
Applying Theorem \ref{derivexp}, we have
\begin{equation}\label{Jestiglobal2}
    \begin{aligned}
\|\mathcal{D}^{s_n}\big(e^{\frac{ib x^2}{4(1+bt)}}\big) 
v(\frac{x}{1+bt},\frac{t}{1+bt})\|_{L^2_x}
        \leq & \frac{C}{(1+bt)^{\frac{s_n}{2}-\frac{1}{2}}}\|v(x,\frac{t}{1+bt})\|_{L^2_x}+C\|\frac{|x|^{s_n}}{(1+bt)^{s_n}}v(\frac{x}{1+bt},\frac{t}{1+bt})\|_{L^2_x}\\
        \leq & \Big(\frac{C}{(1+bt)^{\frac{s_n}{2}-\frac{1}{2}}}+C(1+bt)^{\frac{1}{2}}\Big)\|\langle x \rangle^{s_n}v(x,\frac{t}{1+bt})\|_{L^2_x}\\
        \leq & \Big(\frac{C}{(1+bt)^{\frac{s_n}{2}-\frac{1}{2}}}+C(1+bt)^{\frac{1}{2}}\Big)\sup_{t\in [0,b^{-1}]}\|\langle x \rangle^{n}v(x,t)\|_{L^{\infty}_{x}}\|\langle x \rangle^{s_n-n}\|_{L^2_x}.
    \end{aligned}
\end{equation}
Thus, from \eqref{Jestiglobal}, \eqref{Jestiglobal2} and the fact that $0<s_n<n-\frac{1}{2}\leq 1$, we conclude that $u\in C([0,\infty),H^{s_n}(\mathbb{R}))$. (Continuity follows by the same argument as above, applied to the difference of $u(t)$ at two different times). This completes the proof of Claim \eqref{globalclass}.

Next, we find the scattering profile. Recall the following identity due to Cazenave \& Weissler \cite{CazenaveThierry1991}
$$
e^{-it\partial_x^2}u(x,t) = e^{\frac{ibx^2}{4}}e^{-\frac{it}{1+bt}\partial_x^2}v(x, {\frac{t}{1+ b t}}).
$$
Let $u_{+}$ be defined as
\begin{equation}{\label{vplus}}
    u_{+}= e^{\frac{ibx^2}{4}}e^{-i\frac{1}{b}\partial_x^2}v(x,\frac{1}{b}).
\end{equation}
We claim that
\begin{equation}\label{claimscat}
e^{-it\partial_x^2}u(t)\xrightarrow[t \to \infty]{} u_{+} \, \, \text{ in } \, \,  H^{s_n}(\mathbb{R}).
\end{equation}
Indeed, since $v$ solves the integral equation associated to \eqref{NLSsolution} (see \eqref{integralcontrac}), we obtain   
\begin{equation*}
    e^{-it\partial_x^2}u(t) -u_+ = -ie^{i\frac{bx^2}{4}}\int_{\frac{t}{1+bt}}^{\frac{1}{b}}e^{-i\tau\partial_x^2}\mathcal{N}_1(v(\tau))v(\tau)\, d\tau,
\end{equation*}
where $\mathcal{N}_1(v)$ is given by \eqref{nonlineaNLS2}. By Theorems \ref{TheoSteDer}, \ref{derivexp}, and the inequality \eqref{prelimneq}, we get
\begin{equation}\label{scatterring1}
    \begin{aligned}
\|e^{-it\partial_x^{2}} u(t)-u_{+}\|_{H^{s_n}} 
        \leq & C \int_{\frac{t}{1+bt}}^{\frac{1}{b}}\|\mathcal{N}_1(v(\tau))v(\tau)\|_{H^{s_n}}\, d\tau+C\|\mathcal{D}^{s_n}\big(e^{i\frac{bx^2}{4}}\big)\int_{\frac{t}{1+bt}}^{\frac{1}{b}}e^{-i\tau\partial_x^2}\mathcal{N}_1(v(\tau))v(\tau)\, d\tau\|_{L^2}\\
        \leq & C \int_{\frac{t}{1+bt}}^{\frac{1}{b}}\|\mathcal{N}_1(v(\tau))v(\tau)\|_{H^{s_n}}\, d\tau+C\|\langle x \rangle^{s_n}\int_{\frac{t}{1+bt}}^{\frac{1}{b}}e^{-i\tau\partial_x^2}\mathcal{N}_1(v(\tau))v(\tau)\, d\tau\|_{L^2}.
    \end{aligned}
\end{equation}
We estimate each term on the right-hand side of the inequality above. Since $\mathcal{X}\hookrightarrow H^{s_n}(\mathbb{R})$, by Lemmas \ref{linearEst}, \ref{lemmanonlEs} and similar to \eqref{pseudeq1}, we have
\begin{equation}\label{scatterring2}
    \begin{aligned}
     \int_{\frac{t}{1+bt}}^{\frac{1}{b}}\|\mathcal{N}_1(v(\tau))v(\tau)\|_{H^{s_n}}\, d\tau \leq C b^{-1}\langle b^{-1} \rangle^{n+1}\widetilde{\mathcal{G}}_3(t,R),
    \end{aligned}
\end{equation}
where
\begin{equation}\label{scatterring3}
    \begin{aligned}
       \widetilde{\mathcal{G}}_3(t,R):=& C \sum_{k=0}^{\infty}|a_k|\frac{R^{\alpha_k+1}}{\alpha_k-2}\bigg(\frac{1}{1+bt} \bigg)^{\frac{\alpha_k}{2}-1}\\
       &+C \sum_{k=0}^{\infty}\sum_{1\leq \beta\leq M}|a_k||C_{\beta,\alpha_k}| \Big(\big(\tfrac{\lambda}{2}\big)^{-|\alpha_k-2\beta|}+R^{|\alpha_k-2\beta|}\Big)\frac{R^{2\beta+1}}{\alpha_k-2}\bigg(\frac{1}{1+bt} \bigg)^{\frac{\alpha_k}{2}-1}. 
    \end{aligned}
\end{equation}
Since $\big(\frac{1}{1+bt} \big)^{\frac{\alpha_k}{2}-1}\leq 1$, we have $\widetilde{\mathcal{G}}_3(t,R)\leq \widetilde{\mathcal{G}}_1(R)$, recalling the definition of $\widetilde{\mathcal{G}}_1$ in \eqref{G1functionprim}. Thus, we can apply Weierstrass M-Test to deduce that $\mathcal{G}_3(t,R) \to 0$ as $t\to \infty$. Similarly, applying $\mathcal{X} \hookrightarrow  L^{\infty}(\langle x \rangle^n \, dx)$, Lemmas \ref{linearEst}, \ref{lemmanonlEs} and the arguments in \eqref{pseudeq1}, we get
\begin{equation}\label{scatterring4}
    \begin{aligned}
        \|\langle x \rangle^{s_n}\int_{\frac{t}{1+bt}}^{\frac{1}{b}}e^{-i\tau\partial_x^2}\mathcal{N}_1(v(\tau))v(\tau)\, d\tau\|_{L^2} \leq & C\|\langle x \rangle^{-n+s_n}\|_{L^2}\int_{\frac{t}{1+bt}}^{\frac{1}{b}}\|\langle x \rangle^n e^{-i\tau\partial_x^2}\mathcal{N}_1(v(\tau))v(\tau)\|_{L^{\infty}}\,  d\tau \\
        \leq & C b^{-1}\langle b^{-1} \rangle^{n+1}\widetilde{\mathcal{G}}_3(t,R )\to 0 \quad \mbox{as} ~~ t\to \infty.
    \end{aligned}
\end{equation}
Plugging \eqref{scatterring2}, \eqref{scatterring4} into \eqref{scatterring1}, we conclude the scattering result \eqref{claimscat}. Since $\sup \limits_{0<t<|b|^{-1}} \|v(t)\|_{L^{\infty}}<\infty$, from \eqref{Pseudo2} it follows that $\sup \limits_{t>0} \, (1+t)^{\frac{1}{2}}\|u(t)\|_{L^{\infty}}<\infty$. Moreover, the uniqueness of solutions in the class \eqref{globalclass} follows from the energy estimates similar to those as in the proof of Theorem \ref{LWP-X}, noticing that this proof is simpler due to $\alpha_k>2$, which finishes the proof of Theorem \ref{scatres}.

\begin{remark}
The arguments in the proof of Theorem \ref{scatres} also establish the following limit
\begin{equation}
\lim_{\substack{t\to \infty \\ t>0}} \|\langle x \rangle^{s_n} \big(e^{-it\partial_x^2}u(t)-u_{+}\big)\|_{L^2}=0.
\end{equation}
\end{remark}


\section{Well-posedness and scattering in $H^1$}\label{WPHSCH1}

This part focuses on proving Theorem \ref{LocalwellpossH1}, in which we establish local well-posedness in $H^1$ and global existence and scattering for initial data of the form $e^{i\frac{b|x|^2}{4}}v_0$. The results presented here depend strongly on the fact that the sequence $\{\alpha_k\}$ in the nonlinearity \eqref{nonlinearterm} satisfies $\alpha_k\geq 1$ for all $k\geq 0$. This condition allows us to work directly in $H^1(\mathbb{R})$, without using the weighted spaces.

The proof of Theorem \ref{LocalwellpossH1} follows closely the ideas of Theorems \ref{LWP-X} and \ref{scatres}. However, in this section, we precisely detail the differences and main arguments needed to obtain Theorem \ref{LocalwellpossH1}. We divide our results into two parts: the first one concerns local existence of solutions, and the second one deals with global existence and scattering.


\subsection{Local existence results} We start with the following nonlinear estimates.
\begin{lemma}\label{H1nonlinear}
Let $u, v\in H^1(\mathbb{R})$ and $\alpha\geq 1$. Then
\begin{equation}\label{eqh1well}
  \||u|^{\alpha}u\|_{H^1} \leq c_1^{\alpha}|\alpha+2| \,\|u\|_{H^1}^{\alpha+1},
\end{equation}
and
\begin{equation}\label{eqh1well0}
\||u|^{\alpha}u - |v|^{\alpha}v\|_{H^1} 
  \leq  c_1^{\alpha}\Big(\Big|\frac{\alpha}{2}+1\Big||\alpha|+|\alpha+1|+\Big|\frac{\alpha}{2}\Big|(|\alpha-2|+2)\Big)\big(\|u\|_{H^1}+\|v\|_{H^1}\big)^{\alpha}\|u-v\|_{H^1},  
\end{equation}
where $c_1>0$ is a fixed constant such that $\|f\|_{L^{\infty}}\leq c_1\|f\|_{H^1}$ for all $f\in H^1(\mathbb{R})$.
\end{lemma}
\begin{proof}
The deduction of Lemma \ref{H1nonlinear} follows similar arguments as in the estimates for the $H^M$-norm in Lemma \ref{lemmanonlEs}.  However, to show the dependency of the constants, we describe the proofs of \eqref{eqh1well}  and \eqref{eqh1well0}. Writing $|u|^{\alpha}=(u\overline{u})^{\frac{\alpha}{2}}$, we have 
\begin{equation}\label{eqh1well1}
\begin{aligned}
    \partial_{x}(|u|^{\alpha} u)=\big(\frac{\alpha}{2}+1\big)|u|^{\alpha} \partial_{x}u +\frac{\alpha}{2}|u|^{\alpha-2}u^2\overline{\partial_{x} u}.
\end{aligned}   
\end{equation}
Applying the Sobolev embedding $L^{\infty}(\mathbb{R})\hookrightarrow H^{1}(\mathbb{R})$, we deduce
\begin{equation*}
\begin{aligned}
    \||u|^{\alpha} u\|_{H^1}\leq \big|\alpha+2\big|\|u\|_{L^{\infty}}^{\alpha} \|u\|_{H^1}\leq  c_1^{\alpha}\big|\alpha+2\big| \|u\|_{H^1}^{\alpha+1}.
\end{aligned}   
\end{equation*}
Next, we compute the $H^1$-norm of the difference $|u|^{\alpha} u-|v|^{\alpha} v$. We begin with the $L^2$-norm, first writing it as
\begin{equation*}
\begin{aligned}
    |u|^{\alpha} u-|v|^{\alpha} v=&\big(|u|^{\alpha}-|v|^{\alpha}\big) u+|v|^{\alpha}(u-v).
\end{aligned}   
\end{equation*}
By using the mean value inequality and the fact that $\alpha \geq 1$, we get
\begin{equation}\label{eqh1well2}
\begin{aligned}
    \big||u|^{\alpha}-|v|^{\alpha}\big|\leq |\alpha|\big(|u|+|v|\big)^{\alpha-1}|u-v|,
\end{aligned}   
\end{equation}
thus, an application of Sobolev embedding yields
\begin{equation}\label{eqh1well3}
\begin{aligned}
   \||u|^{\alpha} u-|v|^{\alpha} v\|_{L^2}\leq & |\alpha|\big(\|u\|_{L^{\infty}}+\|v\|_{L^{\infty}}\big)^{\alpha-1}\|u\|_{L^{\infty}}\|u-v\|_{L^2}+\|v\|_{L^{\infty}}^{\alpha}\|u-v\|_{L^2} \\
\leq&  c_1^{\alpha}\big|\alpha+1\big|\big(\|u\|_{H^1}+\|v\|_{H^1}\big)^{\alpha}\|u-v\|_{H^1},
\end{aligned}   
\end{equation}
where $c_1>0$ is defined in the statement of Lemma \ref{H1nonlinear}. Now, to estimate the $L^2$-norm of $\partial_{x}(|u|^{\alpha} u)-\partial_{x}(|v|^{\alpha} v)$, we use \eqref{eqh1well1} to write
\begin{equation*}
\begin{aligned}
\partial_{x}(|u|^{\alpha} u)-\partial_{x}(|v|^{\alpha} v)
    = &\big(\frac{\alpha}{2}+1\big)\big(|u|^{\alpha}-|v|^{\alpha}\big) \partial_{x}u+\big(\frac{\alpha}{2}+1\big)|v|^{\alpha}(\partial_{x}u-\partial_{x}v)\\ 
    & +\frac{\alpha}{2}\big(|u|^{\alpha-2}u^2-|v|^{\alpha-2}v^2\big)\overline{\partial_{x} u}+\frac{\alpha}{2}|v|^{\alpha-2}v^2(\overline{\partial_{x} u}-\overline{\partial_{x} v})\\
    =:& \mathcal{I}_1+\mathcal{I}_2+\mathcal{I}_3+\mathcal{I}_4.
\end{aligned}   
\end{equation*}
Using \eqref{eqh1well2}, the ideas in \eqref{eqh1well3}, and the Sobolev embedding, we deduce
\begin{equation*}
\begin{aligned}
\|\mathcal{I}_1\|_{L^2}+\|\mathcal{I}_2\|_{L^2}+\|\mathcal{I}_4\|_{L^2} &
    \leq \Big|\frac{\alpha}{2}+1\Big| |\alpha||\big(\|u\|_{L^{\infty}}+\|v\|_{L^{\infty}}\big)^{\alpha-1}\|u-v\|_{L^{\infty}}\|u\|_{H^{1}}+\big|\alpha+1\big|\|v\|_{L^{\infty}}^{\alpha}\|u-v\|_{H^{1}} \\
    & \leq \, c_1^{\alpha} \,\Big|\frac{\alpha}{2}+1\Big| |\alpha||\big(\|u\|_{H^1}+\|v\|_{H^1}\big)^{\alpha}\|u-v\|_{H^1}+c_1^{\alpha}\big|\alpha+1\big|\|v\|_{H^1}^{\alpha}\|u-v\|_{H^{1}}.
\end{aligned}   
\end{equation*}
On the other hand, assuming that $|u|\leq |v|$ (otherwise, just replace $u$ by $v$ in the argument below), we write
\begin{equation*}
\begin{aligned}
 \frac{\alpha}{2}\big(|u|^{\alpha-2}u^2-|v|^{\alpha-2}v^2\big)\overline{\partial_{x} u}=  \frac{\alpha}{2}\big(|u|^{\alpha-2}-|v|^{\alpha-2}\big)u^2\overline{\partial_{x} u}+\frac{\alpha}{2}|v|^{\alpha-2}(u^2-v^2)\overline{\partial_{x}u}.
\end{aligned}    
\end{equation*}
Since $\alpha \geq 1$, and  $|u|\leq |v|$, the mean value inequality yields
\begin{equation}
\begin{aligned}
\Big|\frac{\alpha}{2}\big(|u|^{\alpha-2}-|v|^{\alpha-2}\big)u^2\Big| \leq \Big|\frac{\alpha}{2}\Big||\alpha-2|\big(|u|+|v|\big)^{\alpha-1}|u-v|,  
\end{aligned}  
\end{equation}
and
\begin{equation*}
 \big|\frac{\alpha}{2} |v|^{\alpha-2}(u^2-v^2)\big|\leq  |\alpha|\big(|u|+|v|\big)^{\alpha-1}|u-v|.
\end{equation*}
Thus, using the Sobolev embedding and the inequalities above, we bound
\begin{equation*}
 \begin{aligned}
 \|\mathcal{I}_3\|_{L^2}\leq &  \Big|\frac{\alpha}{2}\Big|\big(|\alpha-2|+2\big)\big(\|u\|_{L^{\infty}}+\|v\|_{L^{\infty}}\big)^{\alpha-1}\|u-v\|_{L^{\infty}}\|u\|_{H^1}\\
 \leq &  c_1^{\alpha}\Big|\frac{\alpha}{2}\Big|\big(|\alpha-2|+2\big)\big(\|u\|_{H^1}+\|v\|_{H^1}\big)^{\alpha}\|u-v\|_{H^1}.
 \end{aligned}   
\end{equation*}
Collecting the estimates for $\mathcal{I}_j$, $j=1,2,3,4$, we complete the deduction of \eqref{eqh1well0}.

\end{proof}

We are now in a position to prove the first part of Theorem \ref{LocalwellpossH1}.

\begin{proof}[Proof of Theorem \ref{LocalwellpossH1} part (i)]
We start with the existence of solutions. For that we consider the function $\Psi$ defined by \eqref{Phi}, but here with $u_0\in H^1(\mathbb{R})$ fixed, and $u\in C([-T,T];H^1(\mathbb{R}))$. Using that $\{e^{it \partial_x^2}\}$ is a group of isometries in $H^1(\mathbb{R})$, and \eqref{eqh1well}, we deduce
 \begin{equation}\label{eqh1well4}
 \begin{aligned}
  \|\Psi(u)(t)\|_{H^1}\leq & \|u_0\|_{H^1}+\int_0^{|t|}\|\mathcal{N}(u(\tau))u(\tau)\|_{H^1}\, d\tau\\
\leq & \|u_0\|_{H^1}+\sum_{k=0}^{\infty}|d_k|\int_0^{|t|}\||u(\tau)|^{\alpha_k}u(\tau)\|_{H^1}\, d\tau  \\
\leq & \|u_0\|_{H^1}+T\sum_{k=0}^{\infty}|d_k|c_{1}^{\alpha_k}|\alpha_k+2|\Big(\sup_{t\in [-T,T]}\|u(t)\|_{H^1}\Big)^{\alpha_k+1},
 \end{aligned}    
 \end{equation}
and by the estimate \eqref{eqh1well0}, we find
\begin{equation}\label{eqh1well5}
 \begin{aligned}
  \|\Psi(u)(t)-\Psi(v)(t)\|_{H^1}\leq & \int_0^{|t|}\|\mathcal{N}(u(\tau))u(\tau)-\mathcal{N}(v(\tau))v(\tau)\|_{H^1}\, d\tau\\
\leq & \sum_{k=0}^{\infty}|d_k|\int_0^{|t|}\||u(\tau)|^{\alpha_k}u(\tau)-|u(\tau)|^{\alpha_k}u(\tau)\|_{H^1}\, d\tau  \\
\leq & T\sum_{k=0}^{\infty}|d_k|c_{1}^{\alpha_k}\Omega(\alpha_k)\Big(\sup_{t\in [-T,T]}(\|u(t)\|_{H^1}+\|v(t)\|_{H^1})\Big)^{\alpha_k}
\Big(\sup_{t\in [-T,T]}\|u(t)-v(t)\|_{H^1}\Big),
 \end{aligned}    
 \end{equation}
 where
 \begin{equation}\label{Omegadeff}
   \begin{aligned}
 \Omega(\alpha_k)=\Big(\Big|\frac{\alpha_k}{2}+1\Big||\alpha_k|+|\alpha_k+1|+\Big|\frac{\alpha_k}{2}\Big|(|\alpha_k-2|+2)\Big),
   \end{aligned}  
 \end{equation}
and $c_1>0$ is the constant provided by the Sobolev embedding  in Lemma \ref{H1nonlinear}. To put together the previous estimates, we define the space 
\begin{equation*}
\mathcal{B}_{T}=\big\{u\in C([-T,T];H^1(\mathbb{R})): \, \, \sup_{t\in [-T,T]}\|u(t)\|_{H^1}\leq 2\|u_0\|_{H^1}\big\},    
\end{equation*}
with the distance function $\sup_{t\in[-T,T]}\|u(t)-v(t)\|_{H^1}$. Notice that the hypothesis \eqref{coeffcond2} on the sequences $\{d_k\}$ and $\{\alpha_k\}$ imply 
\begin{equation*}
    \sum_{k=0}^{\infty}|d_k||\alpha_k+2|\big(2c_{1}\|u_0\|_{H^1}\big)^{\alpha_k}(2\|u_0\|_{H^1})<\infty
\end{equation*}
and
\begin{equation*}
\sum_{k=0}^{\infty}|d_k|\Omega(\alpha_k)\big(2c_1\|u_0\|_{H^1}\big)^{\alpha_k}<\infty.
\end{equation*}
Consequently, using \eqref{eqh1well4} and \eqref{eqh1well5}, one can find a time $T>0$ such that $\Psi$ defines a contraction on the complete space $\mathcal{B}_T$. Thus, existence is now a consequence of the Banach fixed-point theorem.

The uniqueness of solutions follows by energy estimates similar to those in the proof of Theorem \ref{LWP-X}. Finally, continuous dependence follows by standard arguments based on the estimates \eqref{eqh1well4} and \eqref{eqh1well5}.
\end{proof}

\subsection{Global existence and scattering} We follow a similar strategy as in the proof of Theorem \ref{scatres}. Thus,  we first solve the initial value problem associated to the non-autonomous equation \eqref{NLSsolution} with nonlinearity $\mathcal{N}_1(v)$ given by \eqref{nonlineaNLS2}.

\begin{theorem}\label{solequationb}
Let $\{d_k\}$ be a sequence of complex numbers and $\{\alpha_k\}$ be a sequence of real numbers with $\alpha_k> 2$. Suppose that for any $R_0>0$ the sequences $\{d_k\}$ and $\{\alpha_k\}$ satisfy \eqref{coeffcond3}. Let $v_0\in H^1(\mathbb{R})\cap L^2(|x|^2\, dx)$.  Then there exist $T>0$ and a unique solution $v$ of \eqref{NLSsolution} with initial data $v_0$ such that
\begin{equation*}
  v \in C([0,|b|^{-1}]; H^1(\mathbb{R})\cap L^2(|x|^2\, dx)),  
\end{equation*}
provided that $b>0$ is sufficiently large.
    
\end{theorem}

\begin{proof}
 We apply the contraction mapping principle to the integral operator \eqref{integralcontrac}, but with $v_0\in H^1(\mathbb{R})\cap L^2(|x|^2\, dx)$, and acting on the space
 \begin{equation*}
\begin{aligned}
     \widetilde{\mathcal{B}}_{R,T}=\big\{v\in C([0,b^{-1}]; & H^1(\mathbb{R})\cap L^2(|x|^2\, dx)): \sup_{t\in [0,b^{-1}]}\big(\|v(t)\|_{H^1}+\|x v(t)\|_{L^2}\big)\leq R
   \big\},
\end{aligned}   
 \end{equation*}
equipped with the distance function $\sup\limits_{t\in [0,b^{-1}]}\big(\|u(t)-v(t)\|_{H^1}+\|x(u(t)-v(t))\|_{L^2}\big)$, $u,v\in \widetilde{\mathcal{B}}_T$. By the arguments in \eqref{eqh1well4} and \eqref{eqh1well5}, which depend on the nonlinear estimates in Lemma \ref{H1nonlinear}, we deduce
\begin{equation}\label{globalH1eq1}
\begin{aligned}
  \|\Phi_1(v(t))\|_{H^1}\leq &\|v_0\|_{H^1}+\sum_{k=0}^{\infty}|d_k|c_1^{\alpha_k}R^{\alpha_k+1}\int_0^{b^{-1}}(1-b\tau)^{\frac{\alpha_k}{2}-2}\, d\tau\\
  \leq &  \|v_0\|_{H^1}+2b^{-1}\sum_{k=0}^{\infty}|d_k|c_1^{\alpha_k}\frac{R^{\alpha_k+1}}{\alpha_k-2},
\end{aligned}    
\end{equation}
and
\begin{equation}\label{globalH1eq2}
\begin{aligned}
\|\Phi_1(u(t))- \Phi_1(v(t))\|_{H^1}
 & \leq \bigg(\sum_{k=0}^{\infty}|d_k|c_1^{\alpha_k}\Omega(\alpha_k)(2R)^{\alpha_k}\int_0^{b^{-1}}(1-b\tau)^{\frac{\alpha_k}{2}-2}\, d\tau \bigg)\Big( \sup_{t\in [0,b^{-1}]}\|u(t)-v(t)\|_{H^1}\Big)\\
\leq &2b^{-1}\bigg(\sum_{k=0}^{\infty}|d_k|c_1^{\alpha_k}\Omega(\alpha_k)\frac{(2R)^{\alpha_k}}{\alpha_k-2}\bigg)\Big( \sup_{t\in [0,b^{-1}]}\|u(t)-v(t)\|_{H^1}\Big),
\end{aligned}    
\end{equation}
where $c_1>0$ is a fixed constant such that $\|f\|_{L^{\infty}}\leq c_1\|f\|_{H^1}$ for all $f\in H^1(\mathbb{R})$, and $\Omega(\cdot)$ is defined in \eqref{Omegadeff}. Next, we use Lemma \ref{derivexp2} to find a universal constant $C>0$ such that
\begin{equation}\label{globalH1eq3}
\begin{aligned}
\|x \Phi_1 (v(t))\|_{L^2} 
 \leq &C\,\langle b^{-1}\rangle(\|u_0\|_{H^1}+\|xu_0\|_{L^2}) \\
 &+C\langle b^{-1} \rangle \sum_{k=0}^{\infty}|d_k|\int_0^{b^{-1}}(1-b\tau)^{\frac{\alpha_k}{2}-2}\Big(\||v(\tau)|^{\alpha_k}v(\tau)\|_{H^1}+\|x\, |v(\tau)|^{\alpha_k}v(\tau)\|_{L^2}\Big)\, d\tau. 
\end{aligned}    
\end{equation}
To estimate the above expression, we use Sobolev embedding and the fact that $v\in \widetilde{\mathcal{B}}_T$ to get
\begin{equation}\label{globalH1eq4}
 \begin{aligned}
 \|x\, |v(\tau)|^{\alpha_k}v(\tau)\|_{L^2} &\leq \|v(\tau)\|^{\alpha_k}_{L^{\infty}}\|x v(\tau)\|_{L^2} \leq c_1^{\alpha_k} \sup_{t\in [0,b^{-1}]}\|v(t)\|_{H^1}^{\alpha_k}\|x v(t)\|_{L^2}\leq  c_1^{\alpha_k}R^{\alpha_k+1}.
 \end{aligned}  
\end{equation}
Then, using \eqref{globalH1eq4} and the estimate for the $H^1$-norm  in \eqref{globalH1eq1}, we find the following bound for \eqref{globalH1eq3}
\begin{equation}\label{globalH1eq5}
\begin{aligned}
 \|x \Phi_1 &(v(t))\|_{L^2} \leq &C\langle b^{-1}\rangle(\|u_0\|_{H^1}+\|xu_0\|_{L^2}) +4Cb^{-1}\langle b^{-1} \rangle \sum_{k=0}^{\infty}|d_k|c_1^{\alpha_k}\frac{R^{\alpha_k+1}}{\alpha_k-2}. 
\end{aligned}    
\end{equation}
Using again Lemma \ref{derivexp2}, we obtain
\begin{equation*}
\begin{aligned}
\|x (\Phi_1 (u(t))-\Phi_1 &(v(t)))\|_{L^2} \leq  C\langle b^{-1} \rangle \sum_{k=0}^{\infty}|d_k|\int_0^{b^{-1}}(1-b\tau)^{\frac{\alpha_k}{2}-2}\Big(\||u(\tau)|^{\alpha_k}u(\tau)-|v(\tau)|^{\alpha_k}v(\tau)\|_{H^1}\\
 &\hspace{4cm}+\|x\, \big(|u(\tau)|^{\alpha_k}u(\tau)- |v(\tau)|^{\alpha_k}v(\tau)\big)\|_{L^2}\Big)\, d\tau. 
\end{aligned}    
\end{equation*}
To estimate the previous inequality, by grouping factors, we use \eqref{eqh1well2} and the Sobolev embedding $H^1(\mathbb{R})\hookrightarrow L^{\infty}(\mathbb{R})$ to deduce
\begin{equation*}
\begin{aligned}
\|x\, \big(|u(\tau)|^{\alpha_k}u(\tau)-  |v(\tau)|^{\alpha_k} &v(\tau)\big)\|_{L^2}
 \leq  \|x\, \big(|u(\tau)|^{\alpha_k}- |v(\tau)|^{\alpha_k}\big)u(\tau)\|_{L^2}+\|x\big( |v(\tau)|^{\alpha_k}(u(\tau)-v(\tau))\big)\|_{L^2}   \\
\leq &|\alpha_k|\big(\|u(\tau)\|_{L^{\infty}}+\|v(\tau)\|_{L^{\infty}}\big)^{\alpha_k-1}\|u(\tau)-v(\tau)\|_{L^{\infty}}\|x u(\tau)\|_{L^2}\\
&+\|v(\tau)\|_{L^{\infty}}^{\alpha_k}\|x\big(u(\tau)-v(\tau)\big)\|_{L^2}\\
\leq & c^{\alpha_k}|\alpha_k|(2R)^{\alpha_k}\Big(\sup_{t\in [0,b^{-1}]}\|u(t)-v(t)\|_{H^1}\Big)+c^{\alpha_k}R^{\alpha_k}\Big(\sup_{t\in [0,b^{-1}]}\|x (u(t)-v(t))\|_{L^2}\Big).
\end{aligned}    
\end{equation*}
Consequently, we use the above estimate and the ideas in \eqref{globalH1eq2} to get
\begin{equation}\label{globalH1eq6}
\begin{aligned}
 \|x (\Phi_1 (u(t))-\Phi_1 (v(t)))\|_{L^2}  
&\leq 4Cb^{-1}\langle b^{-1} \rangle\bigg(\sum_{k=0}^{\infty}|d_k|c_1^{\alpha_k}\Omega(\alpha_k)\frac{(2R)^{\alpha_k}}{\alpha_k-2}\bigg)
\\
&\hspace{3cm}\times
\sup_{t\in [0,b^{-1}]}\Big(\|u(t)-v(t)\|_{H^1}+\|x (u(t)-v(t))\|_{L^2}\Big). 
\end{aligned}    
\end{equation}
By using hypothesis \eqref{coeffcond3}, the estimates \eqref{globalH1eq1}, \eqref{globalH1eq2}, \eqref{globalH1eq5} and \eqref{globalH1eq6}, it follows that there exist $R=R(\|u_0\|_{H^1},\|x u_0\|_{L^2})$ and $b>0$ large such that $\Phi_1(\cdot)$ is a contraction on the complete metric space $\widetilde{\mathcal{B}}_{R,T}$. This completes the existence part. Uniqueness follows by the energy estimates similar to those in the proof of Theorem \ref{LWP-X}.
\end{proof}

We can now conclude the proof of Theorem \ref{LocalwellpossH1}.

\begin{proof}[Proof of Theorem \ref{LocalwellpossH1} part (ii)]
    Let $u_0=e^{i\frac{b|x|^2}{4}}v_0$ with $v_0\in H^1(\mathbb{R})\cap L^2(|x|^2\, dx)$. By {Theorem} \ref{solequationb}, let $b>0$ be sufficiently large such that there exists a unique solution $v\in C([0,b^{-1}];H^1(\mathbb{R})\cap L^2(|x|^2\, dx))$ of \eqref{NLSsolution} with the initial condition $v_0$. We define $u(x,t)$ as in \eqref{Pseudo2}. From the fact that $v\in H^1(\mathbb{R})\cap L^2(|x|^2\, dx)$, we have
    \begin{equation*}
       u\in C([0,\infty);H^1(\mathbb{R})), 
    \end{equation*}
and $u$ solves \eqref{NLS} with initial condition $u_0=e^{i\frac{b|x|^2}{4}}v_0$. Notice that following the strategy for proving the uniqueness in Theorem \ref{LWP-X}, one can deduce the uniqueness in the class $ C([0,\infty);H^1(\mathbb{R}))$. Moreover, if $v_{+}$ is defined by \eqref{vplus}, we claim
\begin{equation}\label{claimscatH1}
e^{-it\partial_x^2}u(t)\xrightarrow[t \to \infty]{} u_{+} \, \, \text{ in } \, \,  H^{1}(\mathbb{R}).
\end{equation}
Since $v$ solves the integral equation associated to \eqref{NLSsolution}, we find  
\begin{equation*}
    e^{-it\partial_x^2}u(t) -u_+ = -ie^{i\frac{bx^2}{4}}\int_{\frac{t}{1+bt}}^{\frac{1}{b}}e^{-i\tau\partial_x^2}\mathcal{N}_1(v(\tau))v(\tau)\, d\tau,
\end{equation*}
where $\mathcal{N}_1(v)$ is given by \eqref{nonlineaNLS2}. Hence, the above representation allows us to follow similar ideas in \eqref{scatterring1}, which together with \eqref{globalH1eq1} and \eqref{globalH1eq5} yields \eqref{claimscatH1}. To avoid repetitions, we omit further details.  

\end{proof}


\section{Numerical examples}\label{SectionNum}

In this section, we provide several numerical confirmations of solutions and their behaviors to the NLS equation \eqref{NLS} that are given by Theorems \ref{LWP-X}, \ref{LocalwellpossH1}, and Corollaries \ref{corollwellposed1}, \ref{Cor1-exp}, \ref{Cor2-exp}. In some cases we extend our simulations to the cases of the NLS equation, where either the above theorems are not applicable to the specific initial data, or the well-posedness is not known at all.  

First, we examine the NLS equation \eqref{NLS} with a single nonlinear term with a small power nonlinearity $0<\alpha<1$ in \S \ref{S6.1} and show solutions behavior for the data with a slow polynomial decay. Then we investigate two combined nonlinearities in \S \ref{S6.2} by considering various cases of the nonlinearities and paying special attention to the cases where solitary waves form and whether they play a role of a threshold for the global behavior, as in the single nonlinearity case. Finally, in \S \ref{S6.3} we consider an infinite sum of combined nonlinearities, on an example of an exponential nonlinearity. 

In what follows, for the time evolution of our numerical simulations, we use IRK4 and the finite difference method, as described in \cite{RRY2022}, \cite{YRZ2018} and \cite{YRZ2019}. For the parameters, by $L$ we denote a spatial domain with the size in a range of $10\pi  \div 300\pi$ depending on the decay of the initial data; the number of nodes $N$ is typically on the order of $2^{12} \div 2^{16}$; the spatial step size is $dx = 2L/N$ and the time step size is typically $dt=0.01$ unless we specify otherwise and then we decrease it down to $0.05$ or even $0.001$ in some cases. 

\subsection{NLS with a low power single nonlinear term}\label{S6.1}

Recalling the NLS equation \eqref{NLS} with a single nonlinear term of a low power  $0 < \alpha < 1$,
\begin{equation}{\label{NLS_single}}
	\begin{cases}
		&iu_{t} + \partial_{x}^{2}u + \epsilon |u|^{\alpha}u = 0, \\
		&u(x,0) = u_{0}(x), \\
	\end{cases}
\end{equation}
we examine the time evolution of the initial data $u_0$ of slow polynomial decay that satisfy the assumptions of our Theorem \ref{LWP-X} for the local well-posedness discussed in Remark \ref{R:1}. Namely, for { $n>0$} we take  
$$
u_0(x)=\frac{A}{(1+|x|^2)^{n/2}} \equiv \frac{A}{\langle x \rangle^n},
$$
and examine the time evolution. We also consider a case of $n=\frac12$ as such data are not in $L^2$ and show a time evolution for such initial conditions in Figures \ref{Jap_1over9}, \ref{Jap0.5_1over9_1} and \ref{Jap0.5_1over9_2}, where we fix the nonlinearity power $\alpha = \frac12$ and the coefficient $\epsilon  = 0.5$ and consider examples with different initial amplitude $A$ and the spatial decay 
$$
u_0(x) \sim~ \tfrac1{|x|}, \tfrac1{|x|^{2/3}}, \quad \mbox{and} \quad \tfrac1{|x|^{1/2}} \quad \mbox{as} \quad |x| \to \infty.
$$
\begin{figure}[h!]
\begin{tabular}{m{6cm}m{6cm}m{6cm}}
		\includegraphics[width=\linewidth]{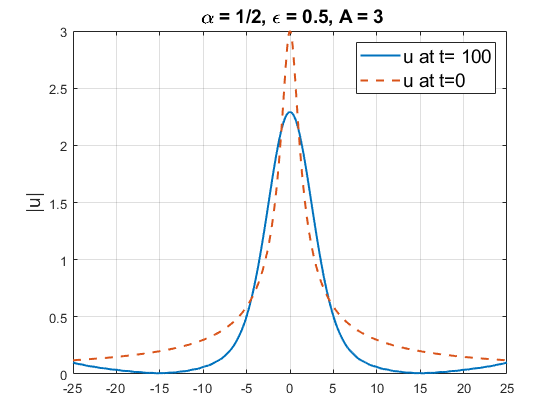}
		& \includegraphics[width=\linewidth]{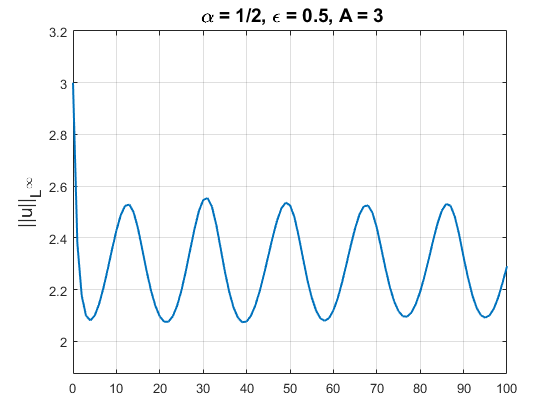}
		& \includegraphics[width=\linewidth]{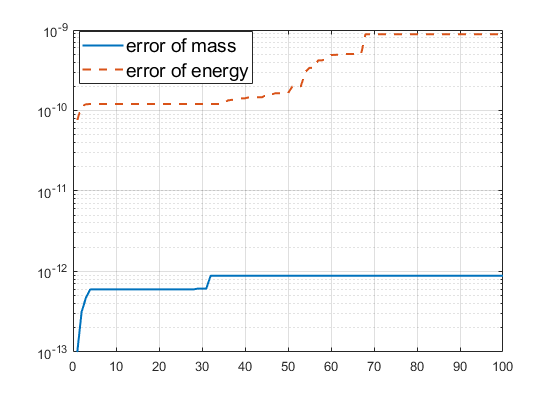}\\
		\end{tabular}
	\caption{\small {Time evolution of $u_0 = \frac{A}{({1 + x^2})^{1/2}}$ under the NLS flow \eqref{NLS_single} with parameters: $\alpha = \frac{1}{2}$, $\epsilon = 0.5$ and $A=3$.
    \label{Jap_1over9}}}
\end{figure}

In Figure \ref{Jap_1over9} we show our numerical simulations for the NLS evolution \eqref{NLS_single} with the initial condition $u_0(x) = \frac{3}{\langle x \rangle}$. 
In this simulation, the parameters are  $L = 150\pi$, $N = 2^{16}$, $dx = 2L/N \approx 0.0144$, and $dt = 0.01$. One can observe that the solution drops slightly from its initial amplitude and then oscillates  (approximately around 2.3, see left and middle plots in Fig. \ref{Jap_1over9}). This indicates that the solution oscillates around some final asymptotic state. 
We also confirm the accuracy of our numerical scheme, which preserves mass and energy close to machine precision; see the right plot in the same figure.

Taking the initial condition slower than in the first example, $u_0(x) = \frac{A}{\langle x \rangle^{2/3}}$, we show the time evolution in Figure 
\ref{Jap0.5_1over9_1} for $A=0.1$.
Since the decay is slower, we increased the computational domain and the time step, thus, the parameters are $L = 300\pi$, $N = 2^{14}$, $dx = 2L/N \approx 0.1150$, and $dt = 0.005$ in this simulation. 
One can observe that the oscillations are slower than in the first example, however, they still seem to oscillate around 0.07, see middle plot in Figure \ref{Jap0.5_1over9_1}. The error in the conserved quantities is of the machine precision as well. 
\begin{figure}[h!]
\begin{tabular}{m{6cm}m{6cm}m{6cm}}
        \includegraphics[width=\linewidth]{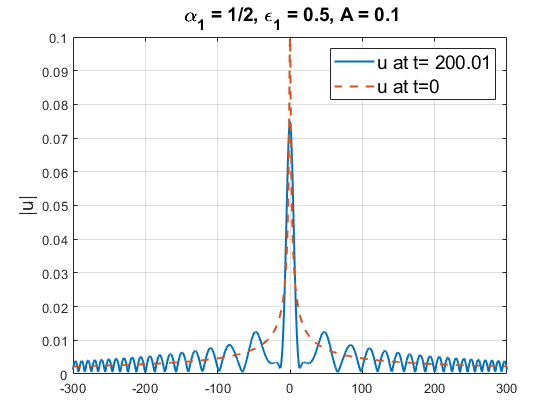}
		& \includegraphics[width=\linewidth]{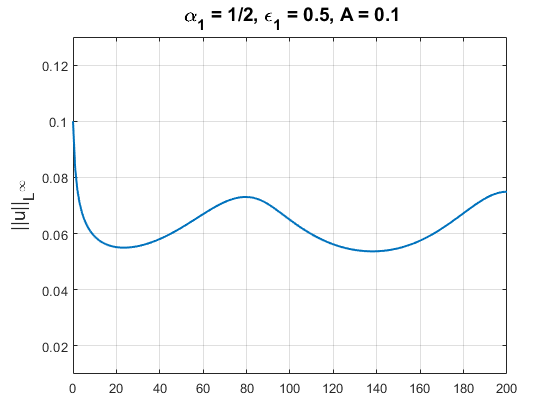}
		& \includegraphics[width=\linewidth]{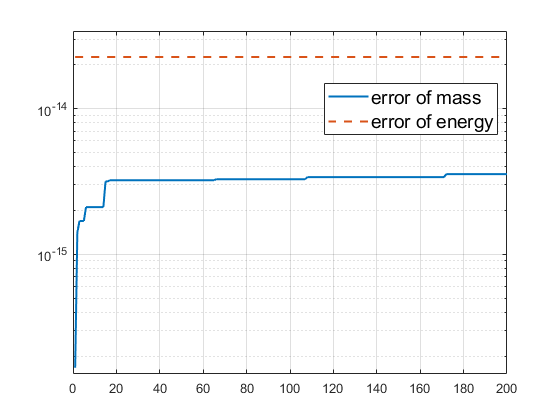}
		\end{tabular}
	\caption{\small {Time evolution of $u_0 = \frac{A}{(1 + x^2)^{1/3}}$ under the NLS flow \eqref{NLS_single} with parameters: $\alpha = \frac{1}{2}$, $\epsilon = 0.5$, and $A=0.1$.
    }}
    \label{Jap0.5_1over9_1}
\end{figure}

In the third example, we are able to consider {the case of} data with a very slow decay, which just misses being in $L^2$, namely $u_0(x) =  \frac{A}{\langle x \rangle^{1/2}}$, see Figure \ref{Jap0.5_1over9_2}. The computational parameters (length $L$, number of nodes $N$, spatial and time step sizes) are the same as in the second example; the computational errors of mass and energy are similarly of machine precision, see the right plot in this figure.  
Quantitatively, the behavior of the solution is similar to the previous examples, the oscillations of the solutions are slow as in the second example.
\begin{figure}[h!]
\begin{tabular}{m{6cm}m{6cm}m{6cm}}
        \includegraphics[width=\linewidth]{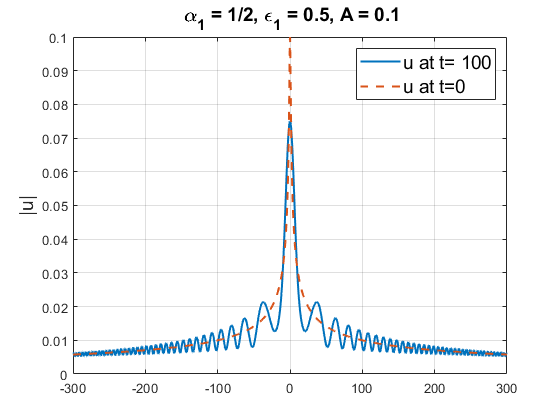}
		& \includegraphics[width=\linewidth]{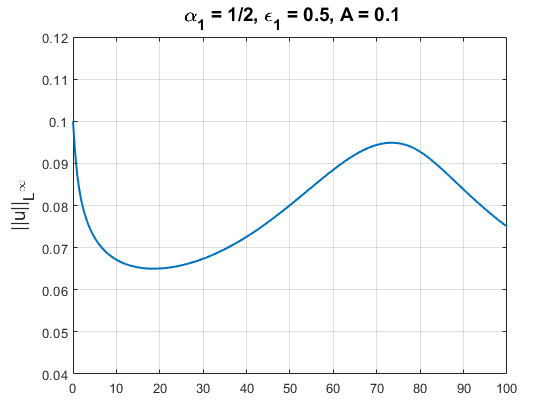}
		& \includegraphics[width=\linewidth]{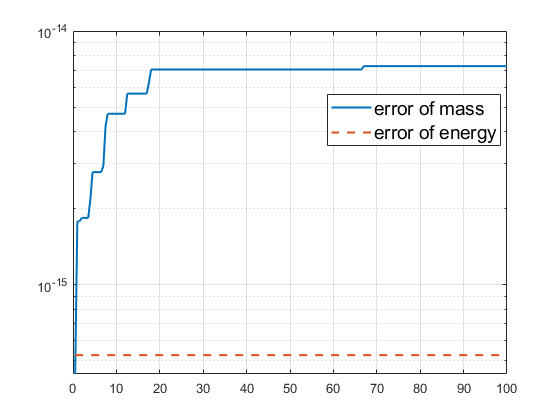}
		\end{tabular}
	\caption{\small {Time evolution of $u_0 = \frac{A}{(1 + x^2)^{1/4}}$ under the NLS flow \eqref{NLS_single} with parameters: $\alpha = \frac{1}{2}$, $\epsilon = 0.5$, and $A=0.1$. 
    }\label{Jap0.5_1over9_2}} 
    
    \end{figure}

For comparison, we also examine initial data with a faster, exponential decay that satisfy the $H^1$ assumption of Theorem \ref{LocalwellpossH1}. For that we recall the ground state solution $u(t,x) = e^{it}Q(x)$, where $Q$ is the unique positive $H^1$ solution of
$$
-Q + Q'' + \epsilon Q^{\alpha+1} = 0.
$$
As we are in the 1d setting, the ground state is given explicitly by
\begin{equation}\label{E:sech-alpha}
Q(x) = \epsilon^{-\frac{1}{\alpha}}\left(\frac{\alpha+2}{2}\right)^\frac{1}{\alpha}\mbox{sech}^\frac{2}{\alpha} \left(\frac{\alpha}{2}x\right),
\end{equation}
which is convenient for testing our numerical scheme and tracking errors. 
For these simulations, it is sufficient to take a shorter domain $L = 10\pi$, as well as the number of points $N = 2^{12}$, the spatial and temporal step sizes are $dx = 0.0153$ and $dt = 0.01$. 
Taking the initial condition $u_0=Q(x)$ for the same nonlinearity as in the above three examples $\alpha=\frac12$ and $\epsilon = 0.5$, we confirm that it is indeed a stationary solution (i.e., $|Q(x)|$ is constant in time, and the difference with the exact solution is on the order of the machine precision). We omit the figure for conciseness. We then consider perturbations of the ground state, $u_0 = AQ$, with $A >1 $ or $A<1$, and track their time evolutions, which tend to oscillate around a certain final state, see Figure \ref{Q1over2} middle plot ($L^\infty$ norm evolution in time), similar to the slow decaying solutions (discussed previously in Figures \ref{Jap_1over9}, \ref{Jap0.5_1over9_1} and \ref{Jap0.5_1over9_2}). This would be expected, since the nonlinearity is small and subcritical. The error plot on the right shows that the energy and mass are preserved as in the previous examples to almost machine precision. 
\begin{figure}[h!]
	\begin{tabular}{m{6cm}m{6cm}m{6cm}}
		\includegraphics[width=\linewidth]{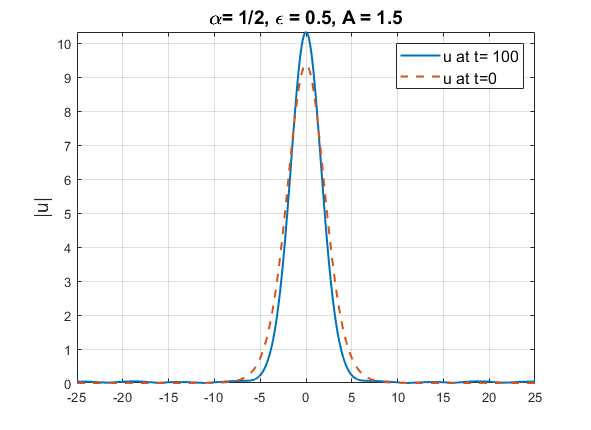}
		& \includegraphics[width=\linewidth]{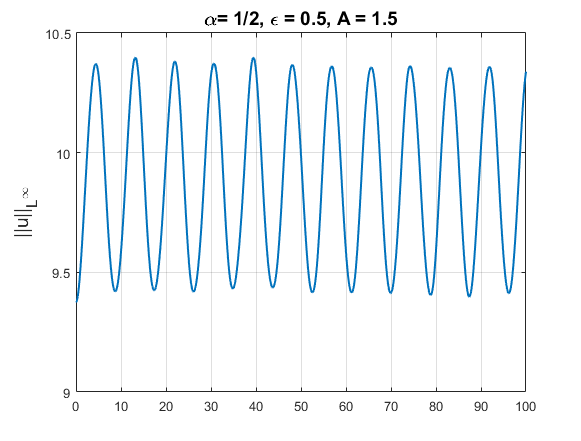}
		& \includegraphics[width=\linewidth]{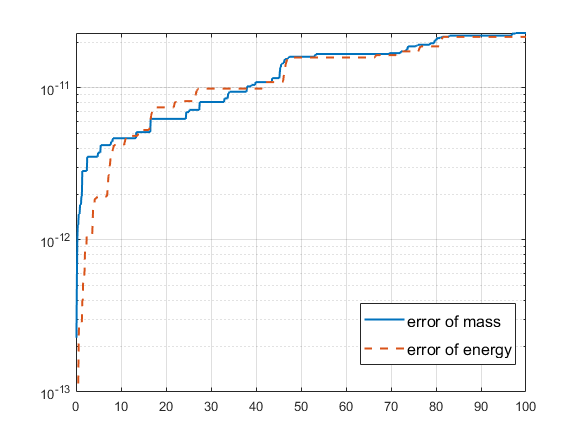}\\
	\end{tabular}
	\caption{\small {Time evolution of $u_0 = 1.5 Q$ under \eqref{NLS_single} with parameters: $\alpha = \frac{1}{2}$ and $\epsilon = 0.5$. 
    }\label{Q1over2}}
\end{figure}

We summarize that in all considered cases here (data with slow polynomial decay and data with fast exponential decay), solutions appear to oscillate around some finite state with radiation dispersing away. 

\subsection{NLS with two combined nonlinearities}\label{S6.2}
Next, we show solutions to the NLS equation with double nonlinearity, considering a range of positive powers, including small nonlinear powers (less than 1) and other positive integer powers, partially, for comparison and also for further understanding of the global behavior in this case. Namely, we show a few examples of solutions to the following cNLS equation
\begin{equation}{\label{NLS-double}}
	\begin{cases}
		&iu_{t} + \partial_{x}^{2}u + \epsilon_{1} |u|^{\alpha_1}u + \epsilon_{2} |u|^{\alpha_2}u = 0, \quad \alpha_i>0, ~~ \epsilon_i \in \mathbb R, i=1,2,\\
		&u(x,0) = u_{0}(x). \\
	\end{cases}
\end{equation} 
We emphasize that since $\alpha_1 \neq \alpha_2$, there is no scaling in this equation, so it can be challenging to establish local well-posedness in some cases of the combined nonlinearities, and even more difficult to study long term behavior of solutions. Nevertheless, Theorem \ref{corollwellposed1} guaranties the local well-posedness of slowly decaying initial data solutions to this equation (also for powers larger than 2 it provides some global results with a quadratic phase and scattering). Before we examine data with slow polynomial decay (such as $u_0 = \frac{A}{\la x \ra^n}$ with varying amplitude $A$), we discuss a test example of a known ground state solution.  

\subsubsection{Ground states for the NLS with double nonlinearity}
Substituting $u(x,t) = e^{i\omega t }Q(x)$, $\omega>0$, into \eqref{NLS-double}, and looking for a positive, smooth, vanishing at infinity solution of
\begin{equation}\label{cGS}
	-\omega Q + Q'' + \epsilon_1 Q^{\alpha_1+1} + \epsilon_2 Q^{\alpha_2+1} = 0,
    \end{equation}
one can obtain ground state solutions in this case, e.g., see \cite{Ohta1995}.
It is known that in some special cases of $\alpha_1$ and $\alpha_2$, namely, when $\alpha_2 = 2\alpha_1$, and $\epsilon_2<0$ (defocusing larger nonlinearity)
there are explicit solutions to \eqref{cGS}, see \cite{Ohta1995}, given by
\begin{align}\label{CombGStateSol}
\qquad  Q(x) = & \left(\frac{\omega}{a + \sqrt{a^2 + b\, \omega\,}\cosh{(\alpha_1\sqrt{\omega} \,x})}\right)^{{1}/{\alpha_1}}, 
\quad a = \frac{\epsilon_1}{\alpha_1 +2}, ~~b = \frac{\epsilon_2}{\alpha_1 + 1},
\end{align} 
provided 
$0< \omega < \frac{\epsilon_1^2}{|\epsilon_2|} \frac{\alpha_1+1}{(\alpha_1+2)^2}$.

We take a specific (known) case of the focusing-defocusing cubic-quintic nonlinearity (smaller focusing power $\alpha_1=2$ and larger defocusing $\alpha_2=4$) to test our numerical scheme. 
We apply Petviashvili’s iteration to obtain the ground state solution of \eqref{cGS}, for a brief overview of this method,  
refer, for instance, to \cite[Section 4]{RRY2022} or \cite{HRW2025} and references therein. 
The ground state example of this cubic-quintic nonlinearity with $\omega = 0.15$  is shown in Figure \ref{fig:GstateComb}, where we compare the exact solution $Q_{exact}$ vs. its numerical approximation $Q_{num}$. (Note that ground state solutions exist in this case for $0<\omega<3/16$; for further studies of this combined NLS see \cite{CS2021} and \cite{CKS2023}.) Parameters used here are  $L = 100\pi$, $N = 2^{16}$, and $dx = 0.0144$. Note that the difference is at least on the order of $10^{-11}$, see right plot in Figure \ref{fig:GstateComb}.
\begin{figure}[h!]
	\begin{tabular}{cc}
        \includegraphics[width=.49\hsize,height=5.8cm]{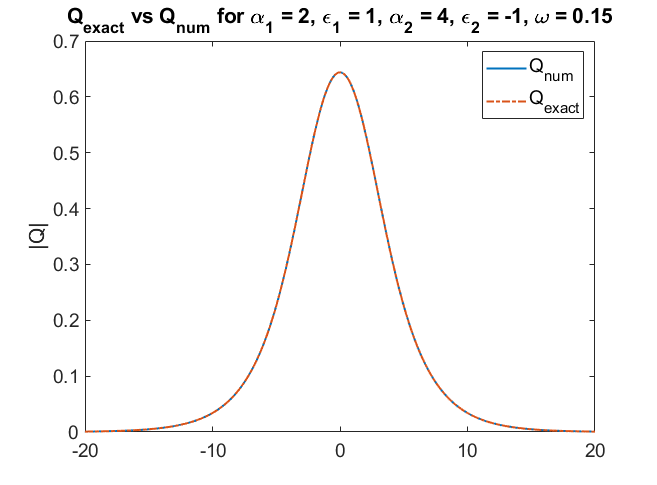} \includegraphics[width=.49\hsize,height=5.8cm]{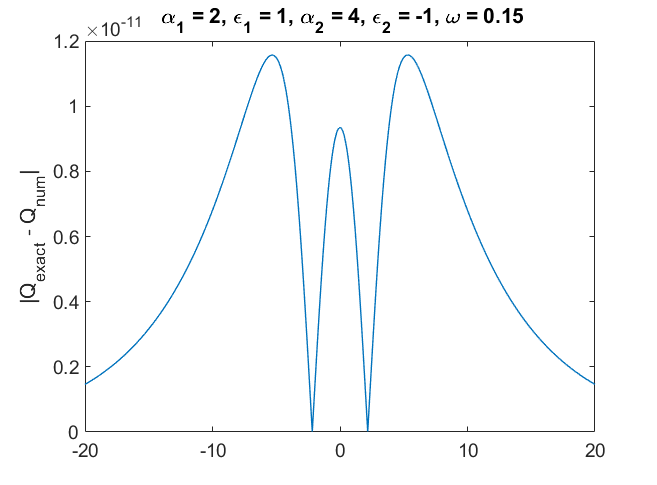}
	\end{tabular}
	\caption{\small Left: Comparison of ground states numerically calculated $Q_{num}$ vs. the explicit $Q_{exact}$ from \eqref{CombGStateSol} in the cubic-quintic case ($\alpha_1 = 2, \alpha_2 = 4$). Right: difference of $Q_{num}$ and $Q_{exact}$.
    }
	\label{fig:GstateComb}
\end{figure}

We show a couple of other examples of numerically obtained ground states for the double nonlinearity \eqref{cGS} (computational parameters here $L = 100\pi$, $N = 2^{16}$, $dx = 0.0144$): 

(i) with small powers $\mathcal N(u) = \epsilon_1 |u|^\frac19  + \epsilon_2|u|^\frac79$ in the left plot of Figure \ref{fig:GstateComb2}, fixing $\omega=0.2$ and the coefficient $\epsilon_2=1$ of the lager (focusing) nonlinearity and varying the coefficient $\epsilon_1$ of the lower power, allowing it to change from negative values to positive;

(ii) with a small power and a critical power 
$\mathcal N(u) = \epsilon_1 |u|^\frac12  + \epsilon_2|u|^4 $ in the right plot of Figure \ref{fig:GstateComb2}, fixing $\epsilon_2=0.9$ and $\omega=0.1$ and varying $\epsilon_1$ from zero down to several negative values.

(Due to the height or amplitude, we only show ground state profiles in the given cases of $\epsilon_1$, however, it is also possible to obtain for other positive and negative values, though the heights will be either very small or very large to properly fit on a plot.) 
\begin{figure}[h!]
\begin{tabular}{cc}
\includegraphics[width=.49\hsize,height=5.8cm]{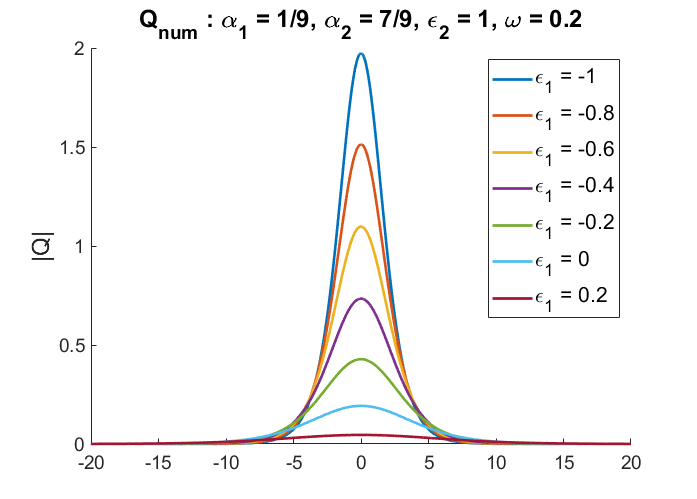}
\includegraphics[width=.49\hsize,height=5.9cm]{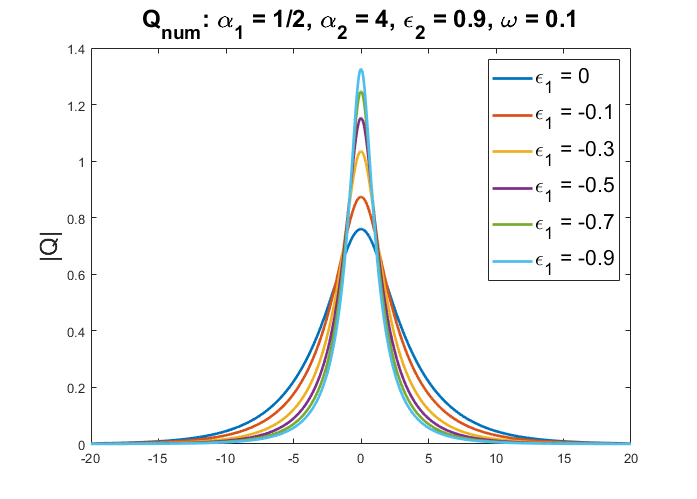}
\end{tabular}
	\caption{\small Numerically computed ground states in \eqref{cGS} compared for different strength of the smaller nonlinearity. Left: $\mathcal N(u)u = \epsilon_1 |u|^\frac19 u + |u|^\frac79 u$, $-1\leq\epsilon_1\leq 0.2$. Right: $\mathcal N(u)u = \epsilon_1 |u|^\frac12 u + 0.9 |u|^4 u$, $-0.9\leq\epsilon_1\leq 0$.}
	\label{fig:GstateComb2}
\end{figure}

\subsubsection{Time evolution of solutions to the NLS with double nonlinearity} We next show examples of the dynamics and time evolution of solutions to the combined NLS equation \eqref{NLS-double}. 
\smallskip

\underline{\bf Example 1.} Our first example is the combined NLS with the nonlinearity  
\begin{equation}\label{E:small-1}
\mathcal N(u) u = \pm |u|^{\frac19} u +  |u|^\frac79 u.
\end{equation}
We start with considering initial data with a slow decay to confirm our well-posedness results. In Figure \ref{fig:alph1_1over9_alph_2_7over9_weight}
we take initial condition
\begin{equation}\label{E:small-2}
u_0=\frac{A}{\la x\ra},
\end{equation}
and vary the amplitude $A$. 
On the left side of Figure \ref{fig:alph1_1over9_alph_2_7over9_weight} we show the time evolution of the NLS with the nonlinearity \eqref{E:small-1} taking the defocusing minus sign of the smaller power and slow polynomial decay initial data from \eqref{E:small-2}. While we have tried various $A$, here we show the snapshots of profiles and time evolution of the $L^\infty$ norm for $A=1, 0.8$.  On the right side of the same figure we consider the focusing sign of the smaller nonlinearity, namely, $\mathcal N(u) u = |u|^{\frac19} u +  |u|^\frac79 u$ with $A=0.8, 0.5$. The parameters in these simulations are $L = 150\pi$, $N = 2^{16}$, $dx = 0.0144$, $dt = 0.01$.

\begin{figure}[h!]
 \centering
 \begin{minipage}[t]{0.49\linewidth}
        \centering
        \begin{tabular}{cc}
            &\includegraphics[width=0.5\linewidth]{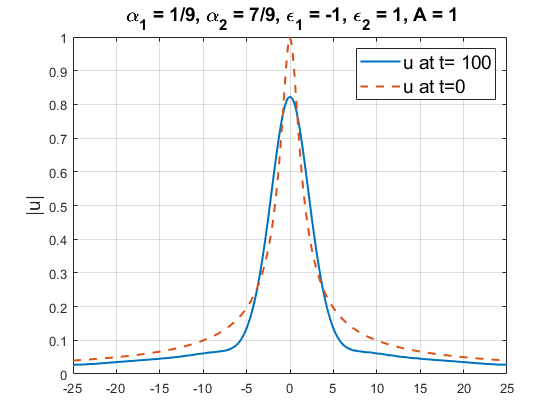}
            \includegraphics[width=0.5\linewidth]{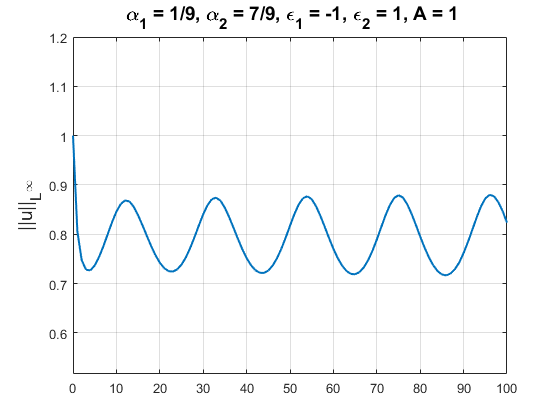}\\
            &\includegraphics[width=0.5\linewidth]{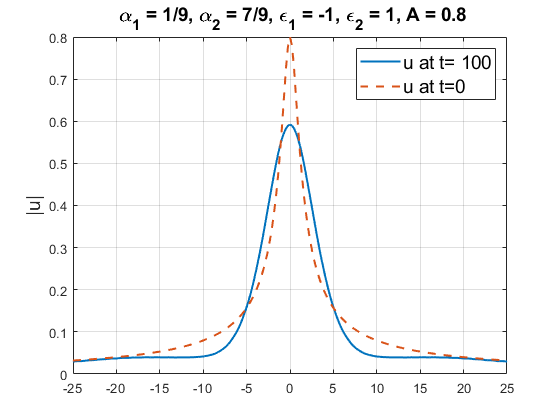}
            \includegraphics[width=0.5\linewidth]{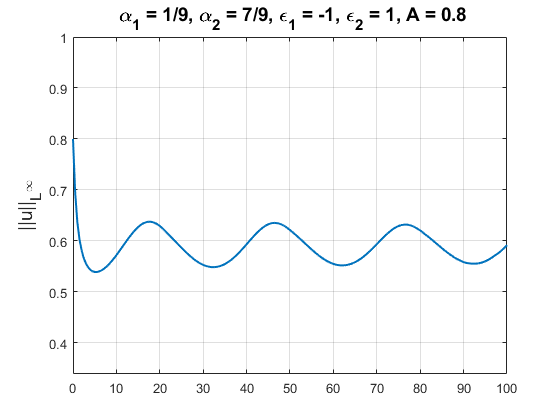}
        \end{tabular}
        \captionsetup{labelformat=empty}
        \captionof{table}{\footnotesize $\mathcal N(u)u=-|u|^{\frac19}u+|u|^{\frac79}u$}
    \end{minipage}
    \begin{minipage}[t]{0.49\linewidth}
        \centering
        \begin{tabular}{cc}
            &\includegraphics[width=0.5\linewidth]{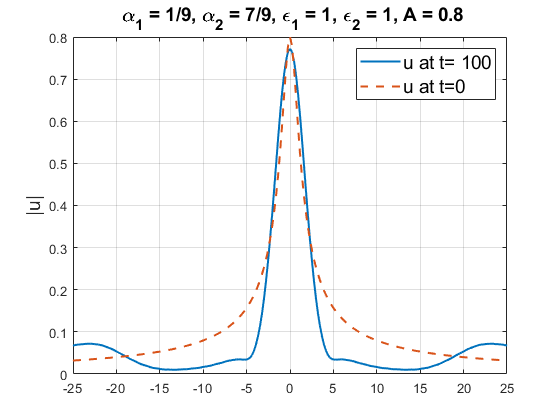}
            \includegraphics[width=0.5\linewidth]{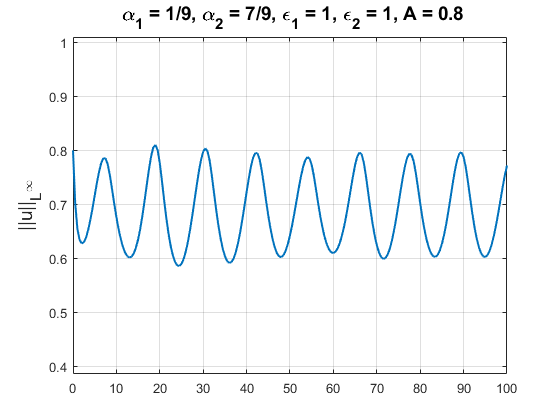}\\
            &\includegraphics[width=0.5\linewidth]{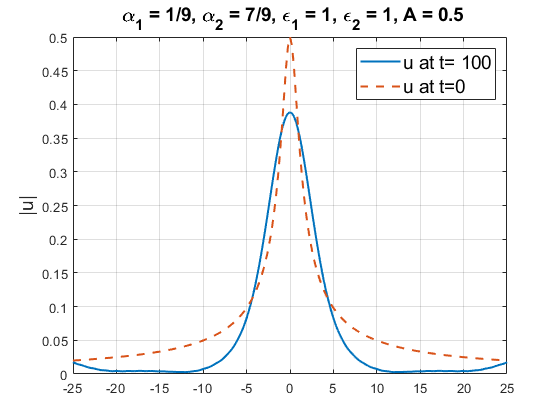}
            \includegraphics[width=0.5\linewidth]{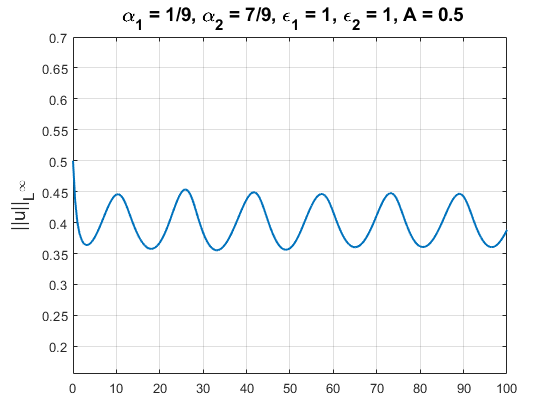}
        \end{tabular}
        \captionsetup{labelformat=empty}
        \captionof{table}{\footnotesize $\mathcal N(u)u=|u|^{\frac19}u+|u|^{\frac79}u$}
    \end{minipage}
        \captionof{figure}{\small Time evolution of the cNLS flow \eqref{NLS-double} with stated nonlinearities $\mathcal N(u)$ and the initial data $u_0 = \frac{A}{\sqrt{1 + x^2}}$ with $A$ as indicated.
        } 
    \label{fig:alph1_1over9_alph_2_7over9_weight}
\end{figure}

We note that in all cases the solutions from slow decay initial data tend to oscillate around a certain final state (dropping slightly in the amplitude) and dispersing some small radiation away (not show in Figure \ref{fig:alph1_1over9_alph_2_7over9_weight}), thus, exhibiting a stable global behavior.
\medskip

\underline{\bf Example 2.}
For a comparison of solution behavior for data with different rate of decay, we take the faster exponentially decaying ground state initial data
$$
\quad u_0 = A\, Q
$$ 
and show solutions behavior for the same double nonlinearity \eqref{E:small-1}.
In Figure \ref{fig:alph1_1over9_alph_2_7over9} the snapshots of profiles and time evolution of the initial condition $u_0 =AQ$ for the same nonlinearities as above and with $A \sim 1$ are shown. The parameters in this simulation is $L = 100\pi$, $N = 2^{14}$, $dx = 0.0383$, and $dt$ is varied between $0.01 - 0.001$. We also take a longer time of simulations (up to $t=200$) to get a better understanding of solutions behavior. 

One can observe that the global behavior is somewhat similar in nature as in the slow decaying initial data shown in Figure \ref{fig:alph1_1over9_alph_2_7over9_weight} with an oscillatory convergence to some final state for the nonlinearity with the defocusing lower power (on the left of Figure \ref{fig:alph1_1over9_alph_2_7over9}) and some oscillations in the case of both focusing nonlinearities (on the right of the same figure). In the case of the nonlinearity $\mathcal N(u) = |u|^{1/9} + |u|^{7/9}$ we point out that the ground state is extremely small in height (on the order of $10^{-7}$) and the oscillations are very slow (thus, we ran simulations for longer time to understand the behavior), however, they do seem to oscillate around some final state as in the previous case. We also plotted $A=1$ to separate values of $A>1$ and $A<1$, and confirm that the $L^\infty$ norm stays constant for the ground state solution itself. 
\begin{figure}[h!]
    \centering
    \begin{minipage}[t]{0.49\linewidth}
        \centering
        \begin{tabular}{cc}
            &\includegraphics[width=0.5\linewidth]{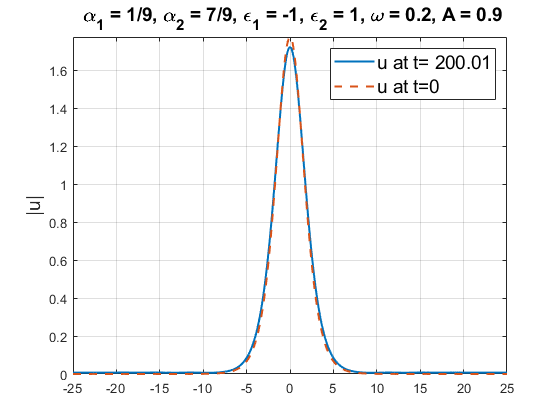}
            \includegraphics[width=0.5\linewidth]{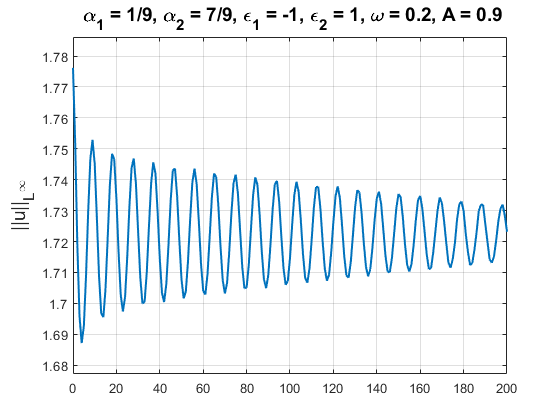}\\
            &\includegraphics[width=0.5\linewidth]{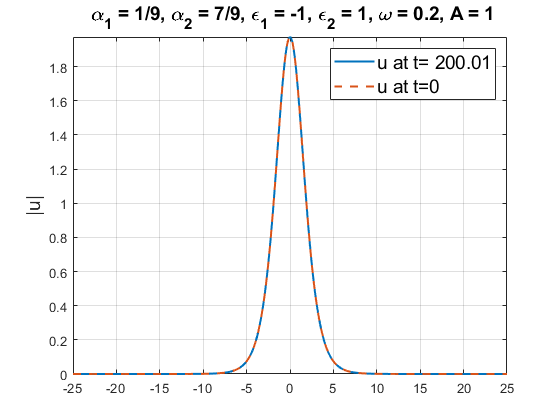}
            \includegraphics[width=0.5\linewidth]{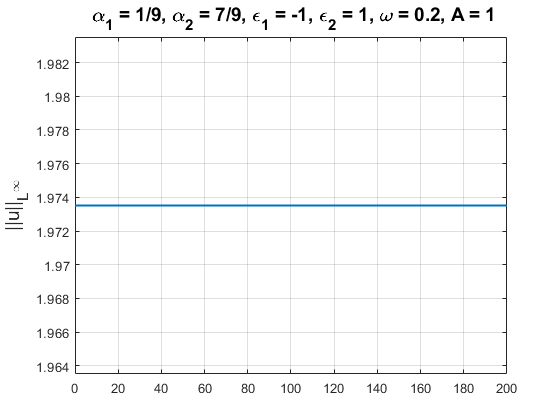}\\
            &\includegraphics[width=0.5\linewidth]{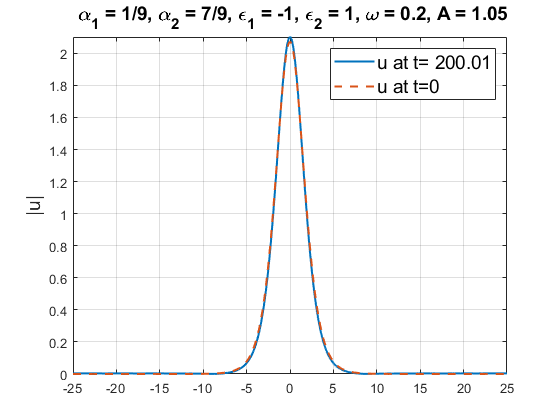}
            \includegraphics[width=0.5\linewidth]{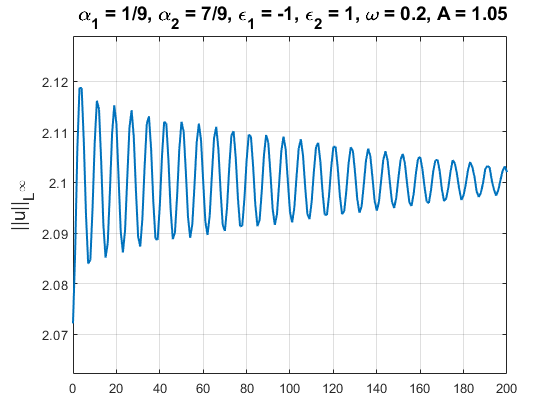}\\
        \end{tabular}
        \captionsetup{labelformat=empty}
        \captionof{table}{\footnotesize $\mathcal N(u)u=-|u|^{\frac19}u+|u|^{\frac79}u$}
    \end{minipage}
    \hspace{-0.1cm}
    \begin{minipage}[t]{0.49\linewidth}
        \centering
        \begin{tabular}{cc}
            &\includegraphics[width=0.5\linewidth]{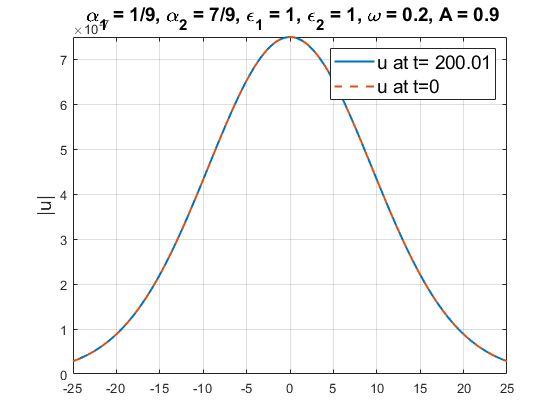}
            \includegraphics[width=0.5\linewidth]{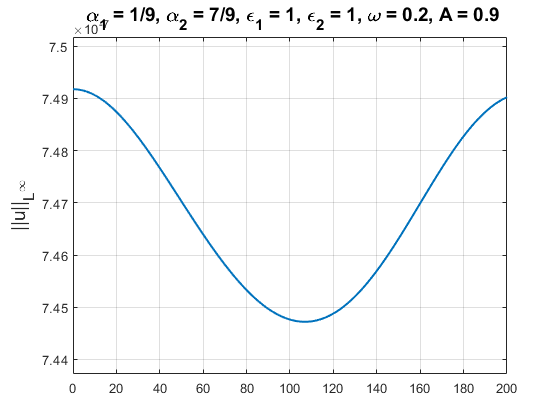}\\
            &\includegraphics[width=0.5\linewidth]{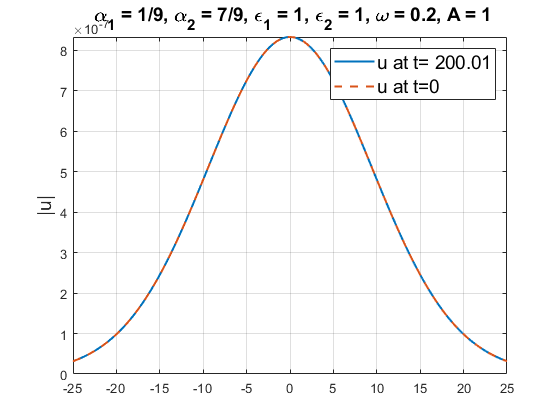}
            \includegraphics[width=0.5\linewidth]{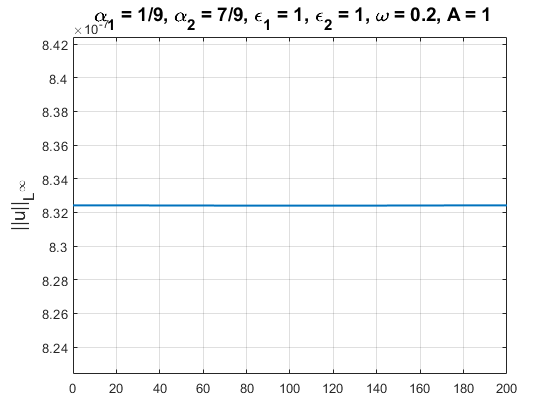}\\
            &\includegraphics[width=0.5\linewidth]{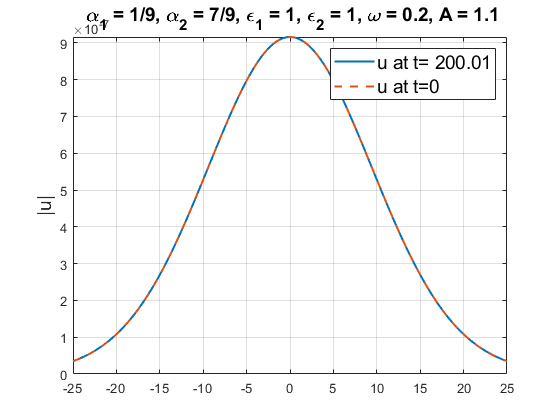}
            \includegraphics[width=0.5\linewidth]{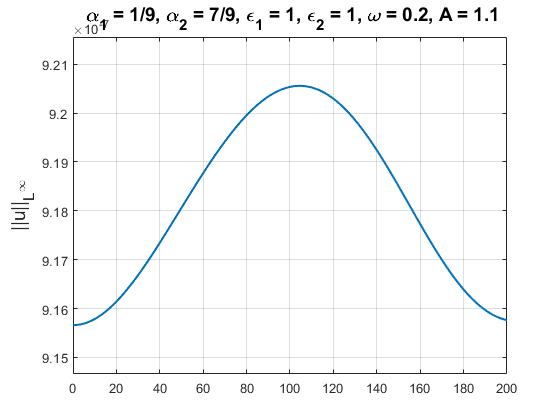}
        \end{tabular}
        \captionsetup{labelformat=empty}
        \captionof{table}{\footnotesize $\mathcal N(u)u=|u|^{\frac19}u+|u|^{\frac79}u$}
    \end{minipage}
    \captionof{figure}{\small Time evolution of NLS flow with nonlinearity $\mathcal N(u)$ indicated for the initial data $u_0 = AQ$ with $Q$ computed numerically with $\omega=0.2$ and $A$ as indicated. }
    \label{fig:alph1_1over9_alph_2_7over9}
\end{figure}

\underline{\bf Example 3.}
Our next example shows that even a small perturbation of the critical nonlinearity by a low power significantly affects the global behavior, including no longer existence of the sharp threshold for scattering vs. blow-up behavior. Here, we consider
$$
\mathcal N(u)u = 0.1|u|^{1/2}u + 0.9|u|^4u,
$$
and take initial data of the perturbed ground state $u_0 = AQ$ (refer to Figure \ref{fig:GstateComb2} for profiles of ground states in this case for other strength $\epsilon_1$ of lower nonlinearity; the reason we select this case is that the ground state mass $M[Q] \approx 4.9713$, which is lower than the mass of the critical case ground state, i.e., with a single nonlinear term $|u|^4u$, or in other words when $\epsilon_1=0$, where the ground state mass is $M[Q_{crit}] \approx 5.7357$).

In Figure \ref{fig:alph1_05_alph2_4_computed} we show cases when $A = 1, 1.05$, and $1.2$, exhibiting very different global dynamics. The parameters used here are $L = 100\pi$, $N = 2^{16}$, $dx = 0.0096$, and $dt$ varied between $0.01$ and $0.001$. 
First, one can note that even if $A>1$, we still get solutions which do not blow up ($A=1.05$ in the second row of Figure \ref{fig:alph1_05_alph2_4_computed}), thus, the ground state $Q$ in this case does not play a role of a sharp threshold, as the scaling in this problem is broken and also $M[Q_{crit}] \approx 1.15 M[Q]$. Secondly, for $A=1.2 > 1.15$ (and hence, $M[AQ]>M[Q_{crit}]$) we observe a blow-up, see the bottom row in the same picture. For a benchmark, we also include the case of $A=1$ (to confirm that it is indeed a non-scattering solution), see the top row in the same figure. Therefore, it would be interesting to investigate further influence of the combined nonlinearities on the thresholds and global behavior of solutions.

\begin{figure}[h!]
    \centering
        \begin{tabular}{cc}
        &\includegraphics[width=0.3\hsize]{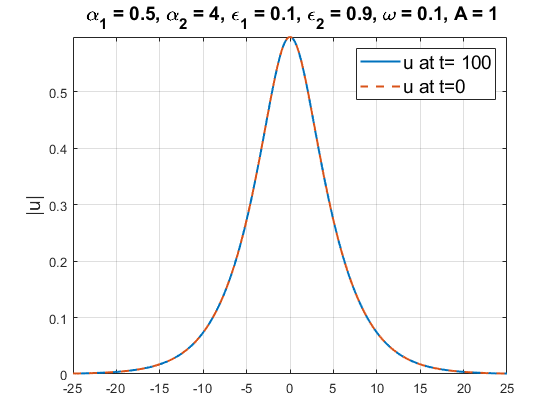}
    		\includegraphics[width=0.3\hsize]{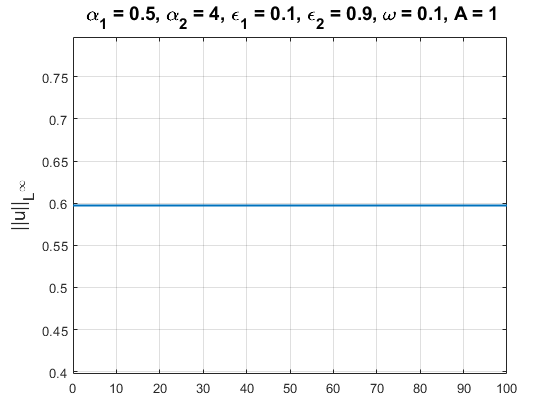}\\
        &\includegraphics[width=0.3\hsize]{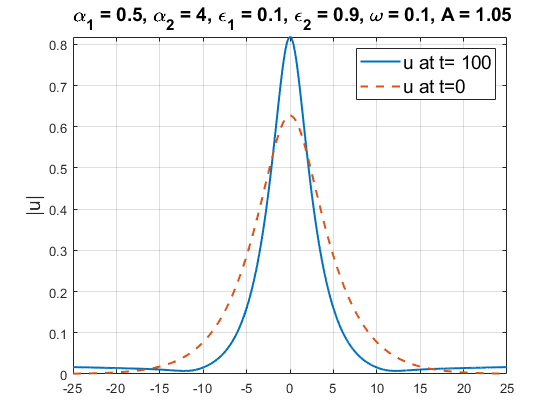}
    		\includegraphics[width=0.3\hsize]{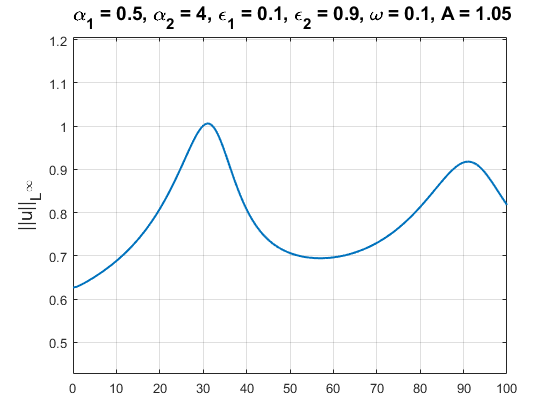}\\
        &\includegraphics[width=0.3\hsize]{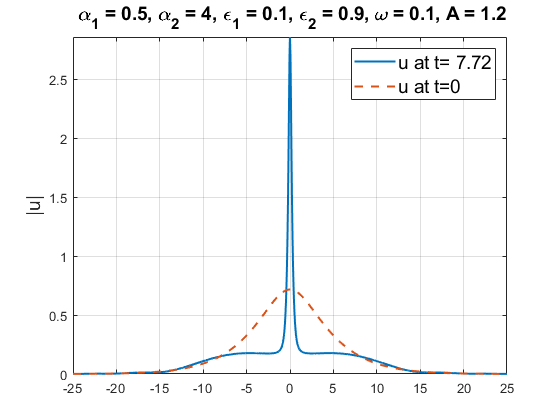}
    		\includegraphics[width=0.3\hsize]{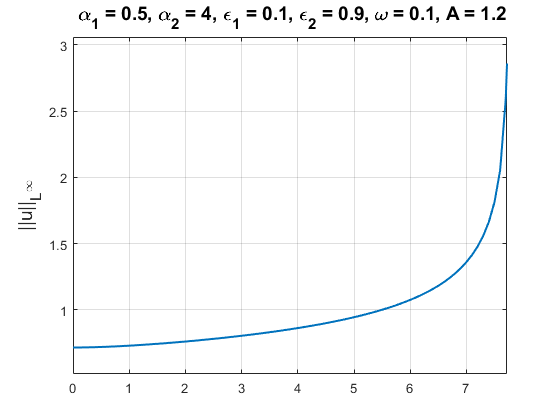}\\
        \end{tabular}
    \caption{\small  Time evolution for solutions to the NLS equation with $\mathcal N(u)u = 0.1|u|^{\frac{1}{2}}u + 0.9|u|^{4}u$ (profiles and $L^\infty$ norms) for the initial data $u_0 = AQ$ with $A=1, 1.05, 1.2$.}
    \label{fig:alph1_05_alph2_4_computed}
\end{figure}

\underline{\bf Example 4.}
For an example of double nonlinearity with integer powers we consider
$$
\mathcal N(u)u = |u|^3u \pm |u|^6u,
$$
for which the local well-posedness in $H^1$ follows, for example, from our Corollary \ref{Cor10}, and we show global behavior of initial data of the type $u_0 = AQ$
for different values of $A \sim 1$ in Figure \ref{fig:a1_3_a2_6}.
In this simulation the parameters we used $L = 100\pi$, $N = 2^{14}$, $dx = 0.0383$, and $dt$ varied between $0.01 - 0.001$. 

In the left two columns of Figure \ref{fig:a1_3_a2_6} we show dynamics of solutions for the focusing-defocusing nonlinearity $\mathcal N(u) = |u|^3 -|u|^6$ (larger defocusing nonlinearity guarantees that solutions never blow up) and on the right two columns of the same figure we show solutions behavior for the focusing-focusing nonlinearity $\mathcal N(u) = |u|^3 + |u|^6$, solutions of which can blow up in finite time, since the larger power is focusing and is supercritical. (For blow-up criteria in the combined NLS case, for example, see a recent work of the second author \cite{R2026}.)

One difference to notice in the global behavior, compared to the single (focusing) nonlinearity case, is that there is {\it no sharp threshold} in solutions behavior between global existence and scattering vs. finite time blow up (a similar phenomenon as in Example 3 above). 
Note that very close to $A=1$, solutions neither scatter nor blow-up in finite time, but rather oscillate to an asymptotically stable state (indicating that the ground state $Q$ (or its rescaling) is an asymptotically stable state). However, stepping further away from $A=1$ value, for smaller $A$ solutions scatter in both cases (see top row in Figure \ref{fig:a1_3_a2_6}) or in the case of larger focusing nonlinearity blow up in finite time (see the bottom right plot of the same figure). 

\begin{figure}[h!]
    \centering
    \begin{minipage}[t]{0.49\linewidth}
        \centering
        \begin{tabular}{cc}
           &\includegraphics[width=0.49\hsize]{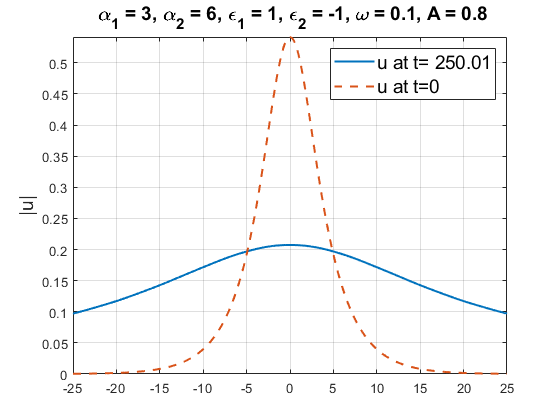} 
            \includegraphics[width=0.49\hsize]{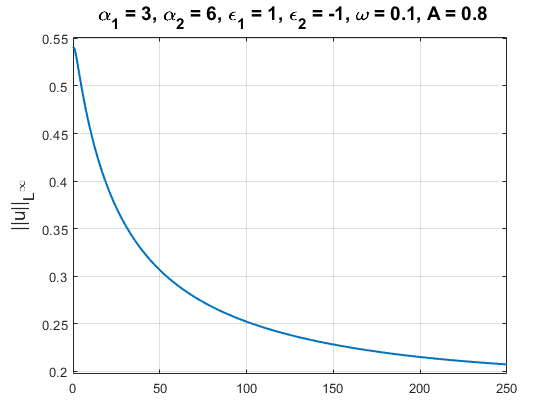}\\
            &\includegraphics[width=0.49\hsize]{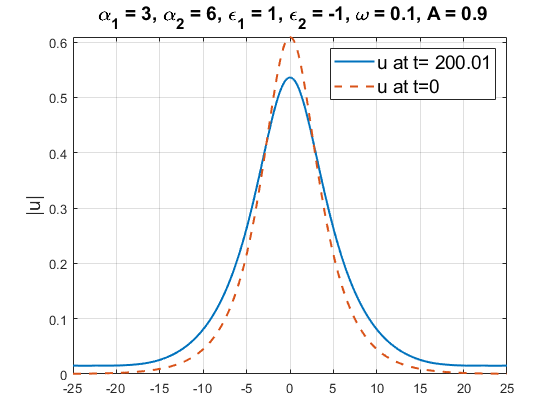} 
            \includegraphics[width=0.49\hsize]{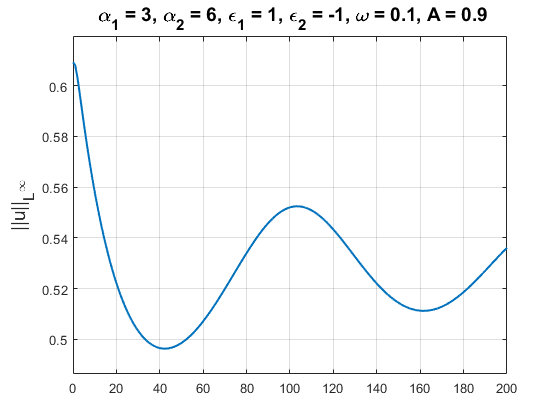}\\
            &\includegraphics[width=0.49\hsize]{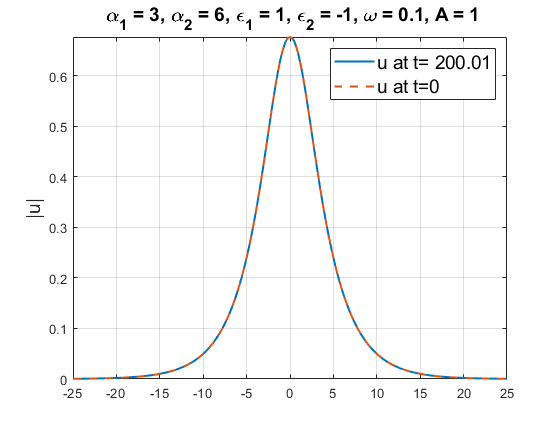} 
            \includegraphics[width=0.49\hsize]{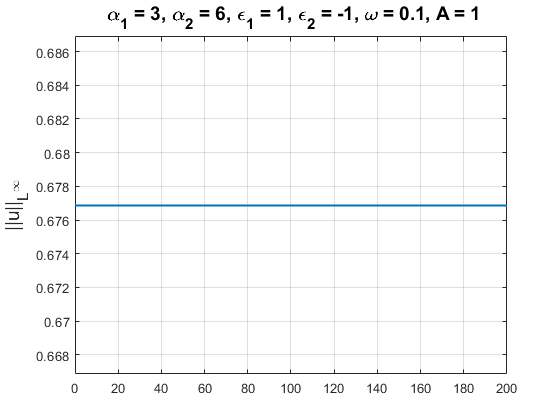}\\
            &\includegraphics[width=0.49\hsize]{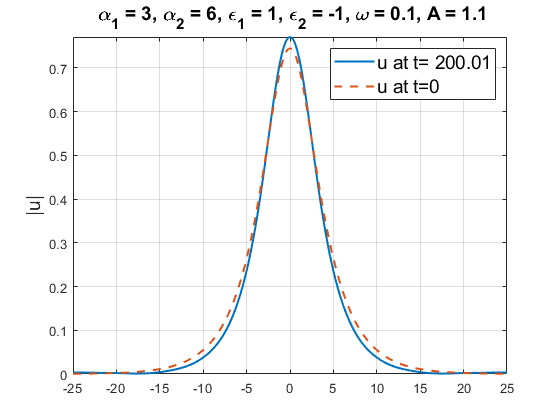} 
            \includegraphics[width=0.49\hsize]{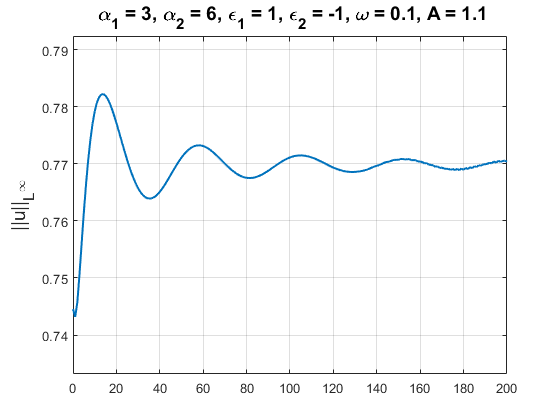}\\
            &\includegraphics[width=0.49\hsize]{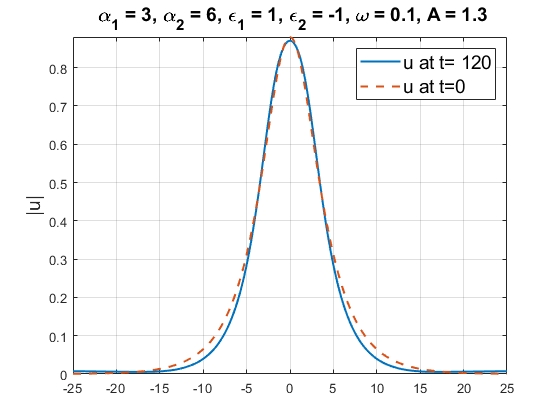} 
            \includegraphics[width=0.49\hsize]{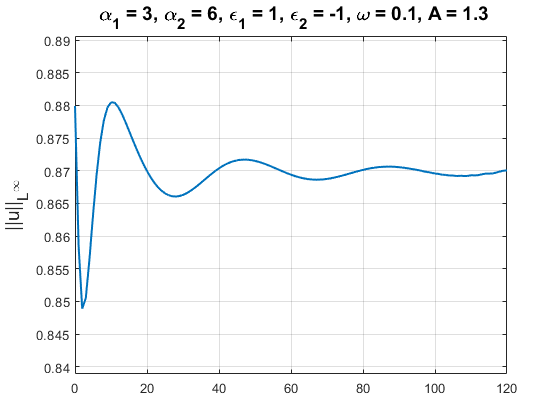}
        \end{tabular}
        \captionsetup{labelformat=empty}
        \captionof{table}{\footnotesize $\mathcal N(u)u = |u|^3u-|u|^6u$}
         \label{fig:a1_3_a2_6_part1}
        \end{minipage}
    \hspace{-0.3cm}
    \begin{minipage}[t]{0.49\linewidth}
        \centering
        \begin{tabular}{cc}
            &\includegraphics[width=0.49\hsize]{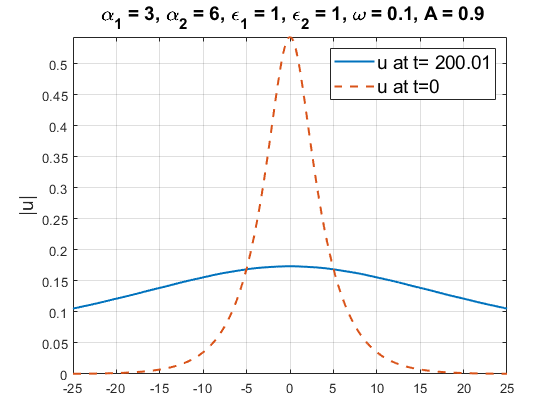} 
            \includegraphics[width=0.49\hsize]{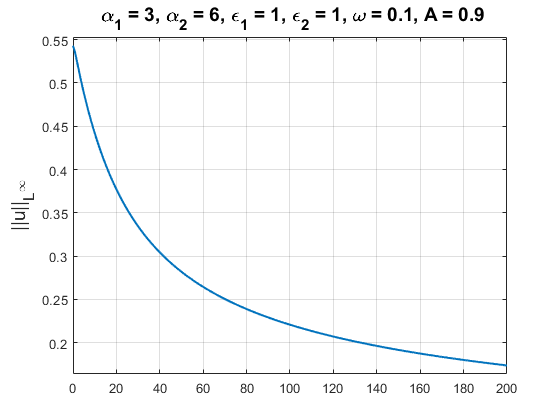}\\
            &\includegraphics[width=0.49\hsize]{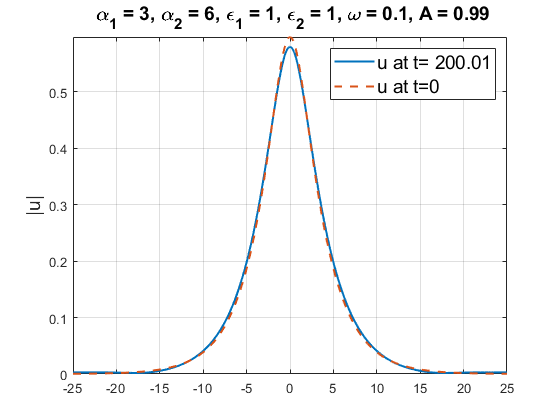} 
            \includegraphics[width=0.49\hsize]{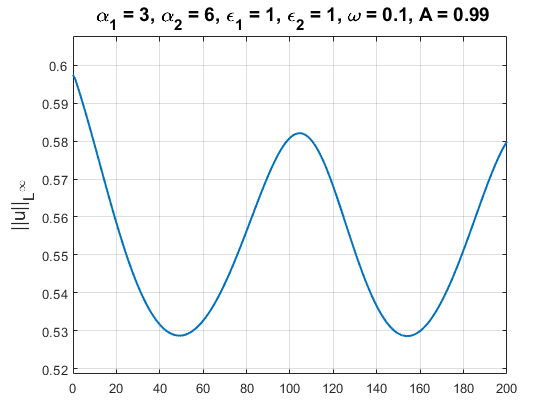}\\
            &\includegraphics[width=0.49\hsize]{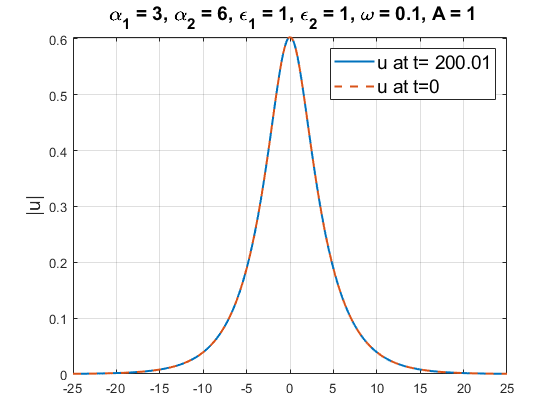} 
            \includegraphics[width=0.49\hsize]{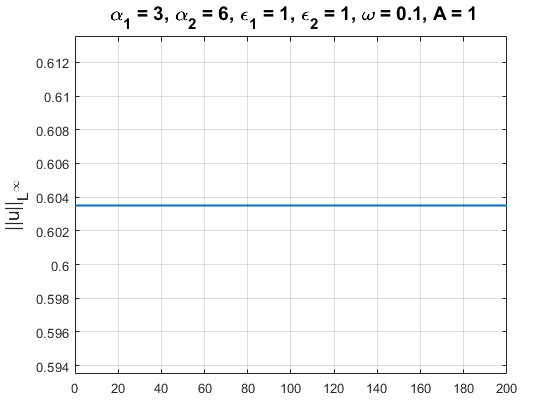}\\
            &\includegraphics[width=0.49\hsize]{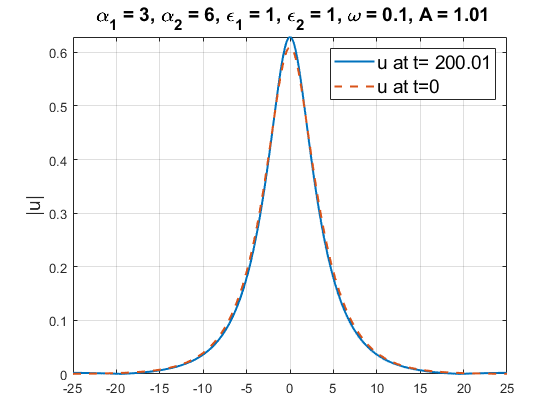} 
            \includegraphics[width=0.49\hsize]{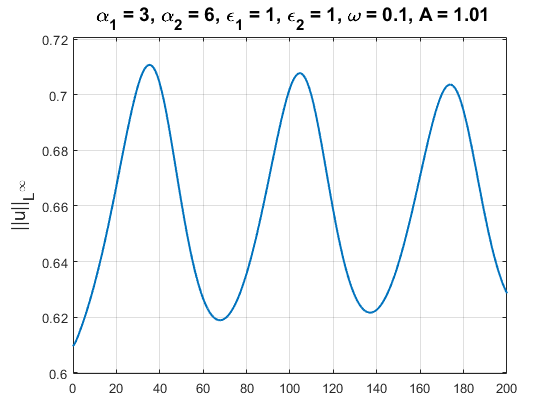}\\
            &\includegraphics[width=0.49\hsize]{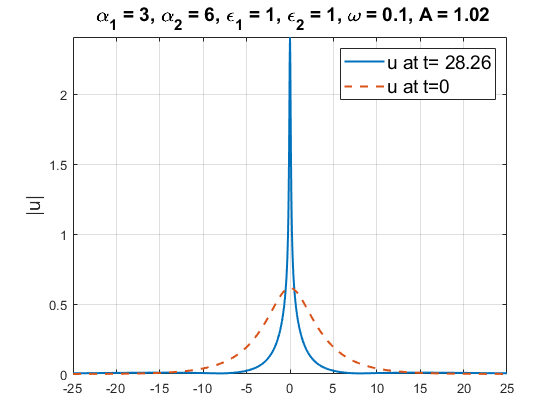} 
            \includegraphics[width=0.49\hsize]{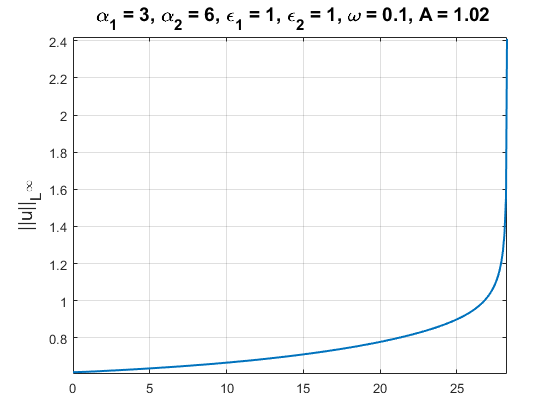}
        \end{tabular}
        \captionsetup{labelformat=empty}
        \captionof{table}{{\footnotesize
        $\mathcal N(u)u = |u|^3u + |u|^6u$ }}
    \end{minipage}
    \captionof{figure}{\small Time evolution (profiles and $L^\infty$ norms) of the NLS flow with the indicated nonlinearity $\mathcal N(u)$ and the initial data $u_0 = AQ$ with the ground state $Q$ from \eqref{cGS},  $\omega=0.1$ and amplitudes $A=0.8, 0.9, 1, 1.1, 1.3$ in the left two columns and $A=0.9, 0.99, 1, 1.01, 1.02$ in the right two columns. }
    \label{fig:a1_3_a2_6}
\end{figure}

\subsection{NLS with an exponential nonlinearity}\label{S6.3}

Finally, we consider the NLS with an exponential nonlinearity, which can also be thought of as an infinite sum of combined nonlinearities $e^{|u|^r} = \displaystyle \sum_{ k=0}^\infty\frac{|u|^{r k}}{k!}$, namely,
\begin{equation}\label{NLS-exponential}
	\begin{cases}
		&iu_{t} + \partial_{x}^{2}u + \epsilon e^{|u|^r}u = 0, \\
		&u(x,0) = u_{0}, \\
	\end{cases}
\end{equation}
where we consider $r\geq 1$, and thus, local well-posedness follows from our Corollary \ref{Cor2-exp} and global behavior for data with quadratic phase from Corollary \ref{C:gwp-exp}. We first discuss ground state solutions and then dynamics of its perturbations and also solutions with polynomial decay. 

\subsubsection{Ground states for the exponential NLS}\label{S:Gstate_exp}
Substituting $u(x,t) = e^{i\omega t}Q(x)$ into \eqref{NLS-exponential}, we obtain
\begin{equation}\label{GSTATE-exp}
    -\omega Q + Q^{\prime \prime} + \epsilon e^{|Q|^r}Q = 0.
\end{equation}
We mention that the existence of ground states for exponential nonlinearities has been investigated, for instance, in \cite{JeanjeanTanaka2003,AlvesSoutoMontenegro2012,RS2013}.
We use the equation \eqref{GSTATE-exp} to solve for a ground state (positive, smooth, vanishing at infinity) solution, computing them numerically, in a similar manner as we did earlier in this section. 

In Figure \ref{fig:ExpGstateComb} we show ground state solutions of \eqref{GSTATE-exp} with $r=1$ and $\omega=0.1$ with the parameter $\epsilon$ varying from $0.025$ to $0.09$. 
For comparison, we also plot the ground state solutions for $r=2$, i.e., 
\begin{equation}\label{GSTATE-exp^2}
    -\omega Q + Q^{\prime \prime} + \epsilon e^{|Q|^2}Q = 0,
\end{equation}
and the same $\omega=0.1$ 
(parameters used here are $L = 40\pi$, $N = 2^{16}$, $dx = 0.0038$).
\begin{figure}[ht!]
\includegraphics[width=0.48\hsize,height=0.34\hsize]{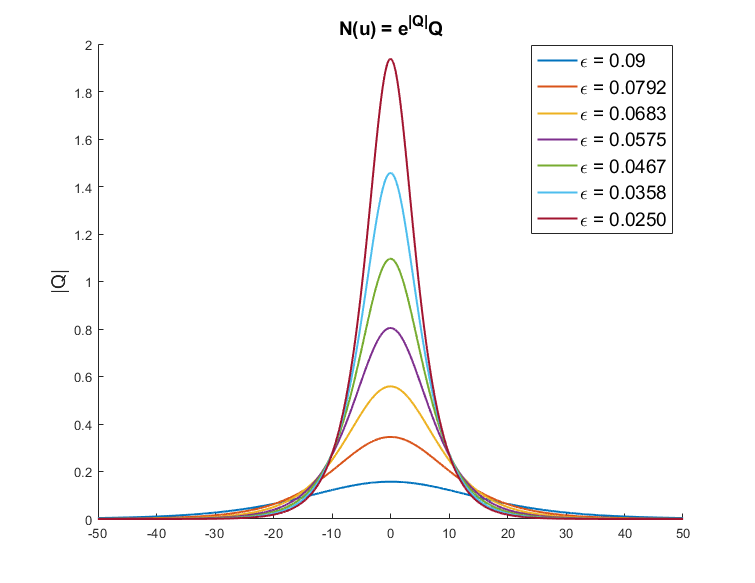}
    \includegraphics[width=0.49\hsize, height=0.34\hsize]{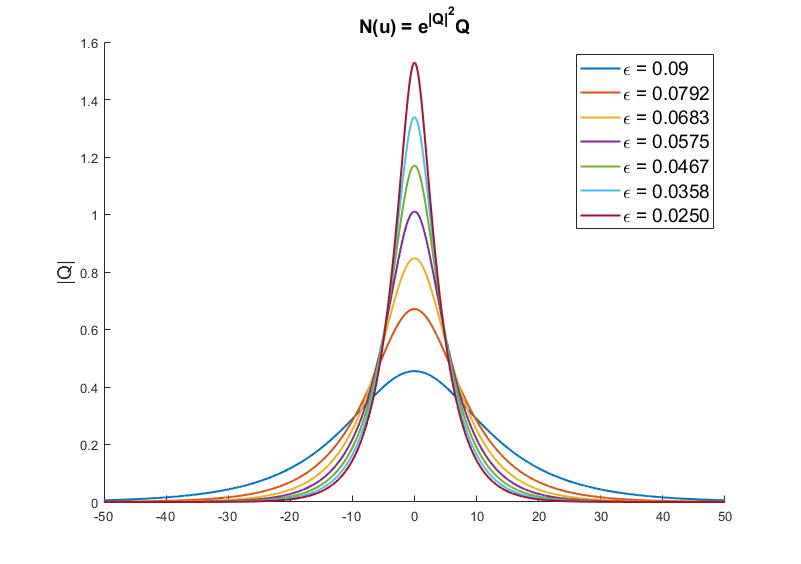}
	\caption{\small Ground state solutions to the NLS with exponential nonlinearity. Left: ground states of \eqref{GSTATE-exp} with $\omega = 0.1$ and different values of $\epsilon$. Right: ground states of \eqref{GSTATE-exp^2} for $\omega = 0.1$ and different values of $\epsilon$.
  	\label{fig:ExpGstateComb}}
\end{figure}

\subsubsection{Time evolution of solutions to the NLS with exponential nonlinearity.} 
We consider the NLS equation \eqref{NLS-exponential} and simulate its dynamics for several cases.
\smallskip

\underline{\bf Example 1.} We fix $\epsilon = 0.5$ and study the time evolution of the ground state and its perturbations, that is, we take $u_0= AQ$ with $Q$ from either \eqref{GSTATE-exp} or \eqref{GSTATE-exp^2} and $A =0.9, 0.99, 1, 1.1, 1.2$. 
\smallskip

In Figure \ref{fig:Exp0025_Exp005} the profiles and $L^\infty$ norm are shown for different values of $A$ (parameters used here are $L = 40\pi$, $N = 2^{14}$, $dx = 0.0153$, $dt = 0.001$).
Observe that similar to the double combined NLS equation in this case of the NLS with an exponential nonlinearity we do not observe any sharp threshold for global solutions behavior. Note that very close to $A=1$ (with $A=0.99$ and $1.1$) the solutions oscillate around some state, and only for much smaller than 1 or larger than 1 values of $A$, the initial condition $u_0 = AQ$ produces solutions that scatter for $A \ll 1$ or blow up in finite time for $A\gg 1$, see numerical confirmations of that in Figure \ref{fig:Exp0025_Exp005} (for the coefficient $\epsilon =0.025$). We also performed simulations with other $\epsilon$ coefficients as well as different $r \geq 1$ values, and found similar dynamics.
\begin{figure}[h!]
    \centering
        \begin{tabular}{cc}
        &\includegraphics[width=0.3\hsize]{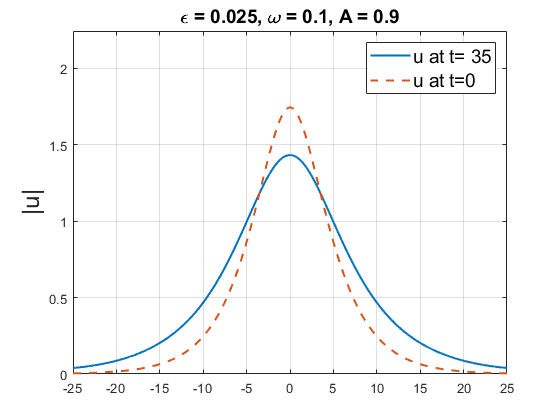}
    		\includegraphics[width=0.3\hsize]{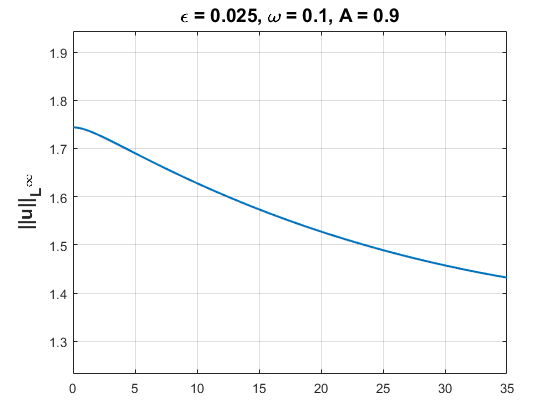}\\
        &\includegraphics[width=0.3\hsize]{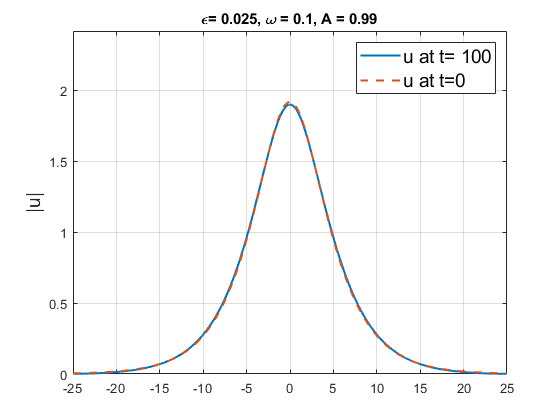}
    		\includegraphics[width=0.3\hsize]{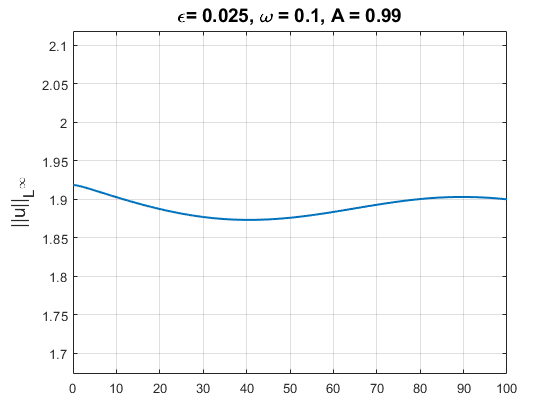}\\
    		  &\includegraphics[width=0.3\hsize]{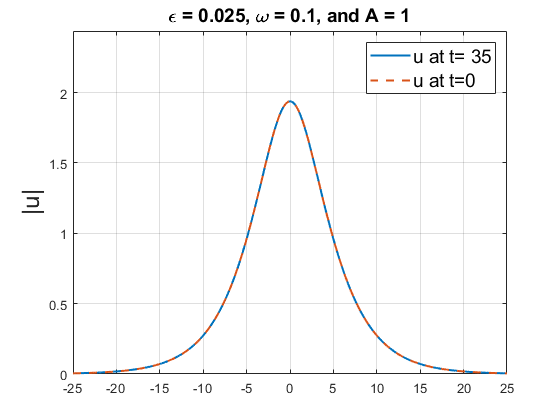}
    		  \includegraphics[width=0.3\hsize]{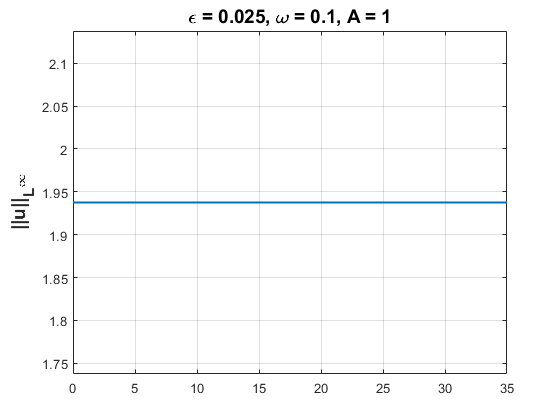} \\
        &\includegraphics[width=0.3\hsize]{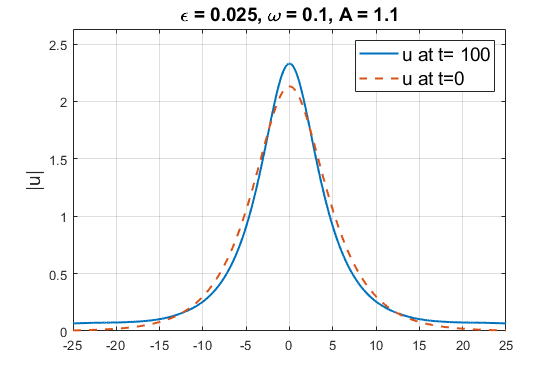}
    		\includegraphics[width=0.3\hsize]{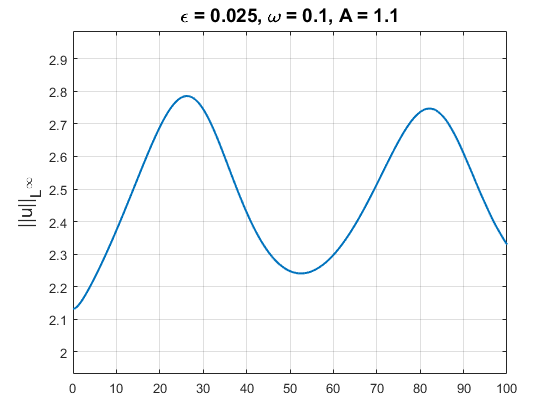}\\
    		  &\includegraphics[width=0.3\hsize]{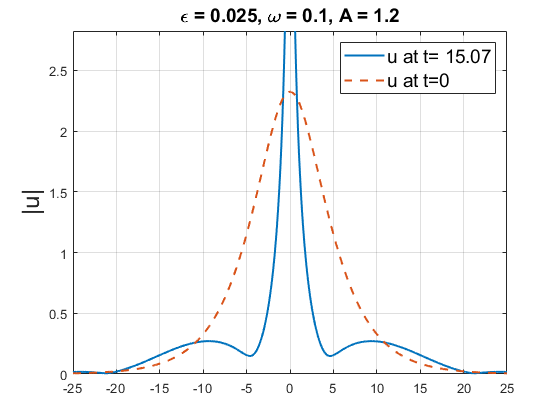}
    		  \includegraphics[width=0.3\hsize]{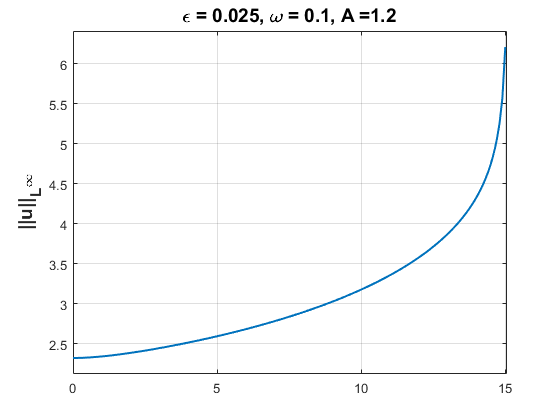}
        \end{tabular}
    \caption{\small  Time evolution for solutions to the NLS equation with $\mathcal N(u)u = 0.025e^{|u|}u$ (profiles and $L^\infty$ norms) for the initial data $u_0 = AQ$ with the ground state $Q$ from \eqref{GSTATE-exp} computed numerically.}
    \label{fig:Exp0025_Exp005}
\end{figure}

\clearpage

\underline{\bf Example 2. }
To confirm the findings of Theorem \ref{scatres} on global existence and scattering with a quadratic phase, we consider the polynomially decaying initial data $P(x)$, of the form $P(x) = \frac{A}{(1 +x^2)^n}$ with quadratic phase, and study the time evolution of the following data
\begin{equation}
    \label{initDataPseudo}
    u_0(x) =  e^{i\frac{bx^2}{4}}P(x), \qquad b \in \mathbb{R}.
\end{equation}

Recall that Theorem \ref{scatres} proves that for large enough $b > 0$, solutions of \eqref{initDataPseudo} exist globally and scatter; we confirm this numerically, see Figure \ref{fig:Exp_Poly_Overlap_and_Bup}. Furthermore, we observe that for {\it any positive} $b>0$ solutions scatter, see Figure \ref{fig:Exp_Poly_Overlap_and_Bup} (parameters used here are: $L = 10\pi$, $N = 2^{14}$, $dx = 0.0038$, $dt = 0.00001$).
\begin{figure}[h!]
	\begin{tabular}{cc}
        \includegraphics[width=0.45\hsize,height=0.35\hsize]{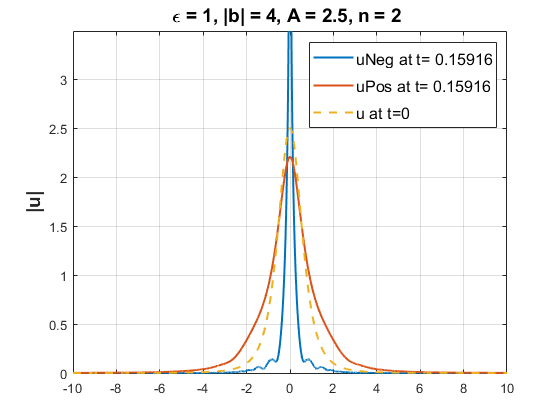}
    	   &\includegraphics[width=0.45\hsize,height=0.35\hsize]{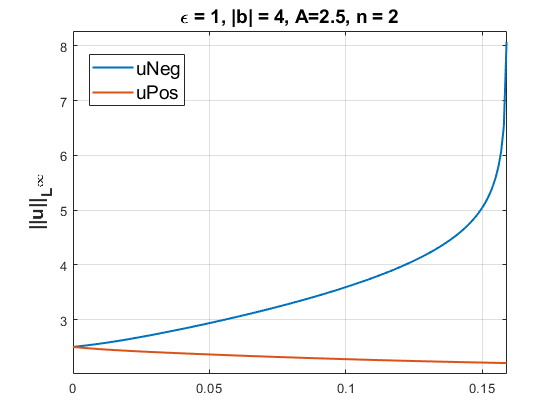}
	\end{tabular}
	\caption{\small Profiles and time evolution of the $L^\infty$ norm for the NLS \eqref{NLS-exponential} with exponential nonlinearity, $r=1$, $\epsilon=1$, and initial data with a quadratic phase \eqref{initDataPseudo}, $b$ = $\pm 4$ and $P(x) = \frac{A}{(1+x^2)^n}$, $A=2.5, n=2$. 
     }
	\label{fig:Exp_Poly_Overlap_and_Bup}
\end{figure}

We next note that if we consider $b < 0$ in \eqref{initDataPseudo}, the solutions start growing rapidly suggesting blow up in finite time.
We use data similar to \eqref{initDataPseudo} with $P = \frac{A}{(1+x^2)^n}$, $n>1/2$, and investigate the threshold via $b$ values in global existence vs. blow-up in finite time. 
For a visualization of the case with $n=2$, $A=2.5$, and $b=\pm 4$, 
see Figure \ref{fig:Exp_Poly_Overlap_and_Bup}.    
\smallskip

\underline{\bf Example 3.}
We investigate quadratic phase behavior with the ground state, namely, instead of a polynomial decay in $P(x)$ in \eqref{initDataPseudo}, we use the ground state data with a quadratic phase
$$
u_0 = e^{i\frac{bx^2}4} Q(x),
$$
where the ground state $Q$ from \eqref{GSTATE-exp} is numerically computed.
We show the profiles and time evolution of the $L^\infty$ norm in 
Figure \ref{fig:Exp_Overlap_and_Bup}, where the red lines show profile (left) and $L^\infty$ norm (right) of scattering solutions, thus, confirming our conjecture about scattering for smaller positive values of $b$, thus, confirming the Conjecture \ref{Conj} for $b > 0$. We also investigate the behavior of solutions for $b < 0$, where in this case the solution dynamics indicates a  blow-up in finite time, see blue lines in Figures \ref{fig:Exp_Overlap_and_Bup}. 
It is especially plausible confirmation for $b < 0$ in Figure \ref{fig:Exp_Overlap_and_Bup} as the $L^\infty$ norm grows quickly in time indicating the blow up behavior. 

\begin{figure}[h!]
	\begin{tabular}{cc}
          \includegraphics[width=0.45\hsize,height=0.35\hsize]{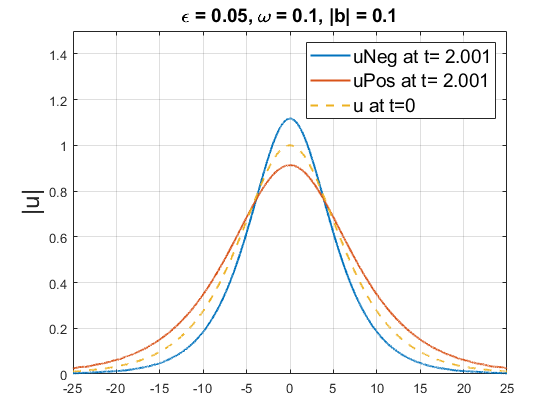}
		  &\includegraphics[width=0.45\hsize,height=0.35\hsize]{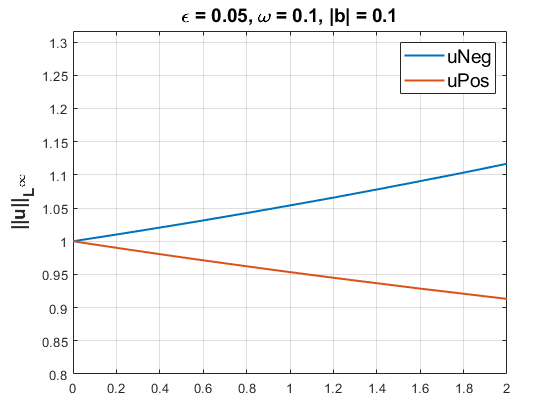}\\
          \includegraphics[width=0.45\hsize,height=0.35\hsize]{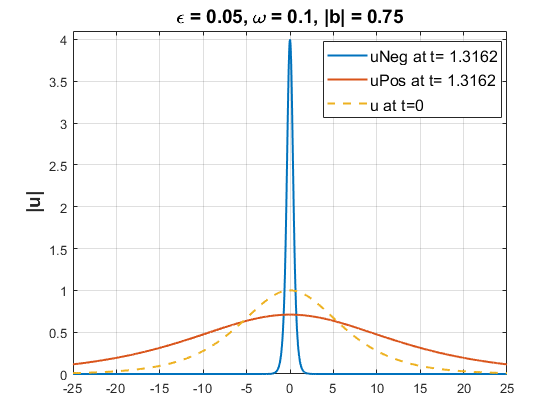}
		  &\includegraphics[width=0.45\hsize,height=0.35\hsize]{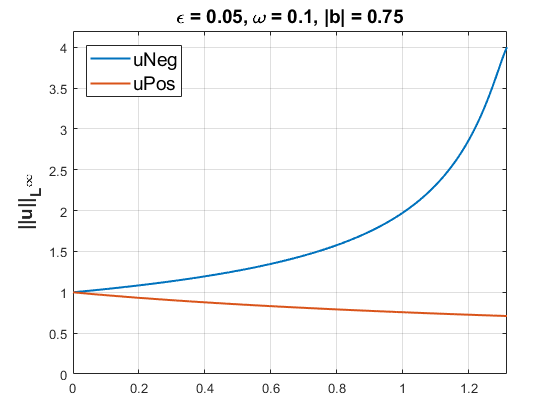}
	\end{tabular}
	\caption{\small Profiles and time evolution of the $L^\infty$ norm for the NLS \eqref{NLS-exponential} with an exponential nonlinearity and initial data $u_0 = e^{i \frac{bx^2}4}Q(x)$. Here we consider a quadratic phase $b$ = $\pm$0.1 (top) and $b$ = $\pm$0.75 (bottom).}
	\label{fig:Exp_Overlap_and_Bup}
\end{figure}

We have therefore given positive confirmation to the Conjecture \ref{Conj} that solutions with initial data \eqref{initDataPseudo} exist globally and scatter for $b>0$ and blow up in finite time for $b<0$.
\medskip

We conclude with mentioning that simulations which suggest existence of blow-up solutions (Figures \ref{fig:a1_3_a2_6}, \ref{fig:Exp0025_Exp005}, \ref{fig:Exp_Poly_Overlap_and_Bup} and \ref{fig:Exp_Overlap_and_Bup}) support the recent results in \cite{R2026}, where the second author studied various blow-up criteria for these equations for both positive and negative energy. There, various examples are also provided that satisfy their respective blow-up criterion.

\newpage
\bigskip

\bigskip

{\bf Conflict of Interest:} The authors declare that they have no
conflicts of interest.

\bigskip

{\bf ORCID}

{\it Oscar Ria\~no}
\qquad {https://orcid.org/0000-0002-6325-8848}

{\it Alex D. Rodriguez}
\qquad {https://orcid.org/0000-0003-4382-7430}

{\it Svetlana Roudenko}
\qquad {https://orcid.org/0000-0002-7407-7639}

\bigskip

\bibliographystyle{abbrv}

\bibliography{references}

\end{document}